\documentclass[12pt]{amsart}
\pdfoutput=1
\usepackage{latexsym, amsmath, amssymb, amsthm, mathrsfs,wasysym,mathtools}
\usepackage[alwaysadjust]{paralist}
\usepackage{caption}
\usepackage{subcaption}
\usepackage[T1]{fontenc}
\usepackage{lmodern}
\usepackage[backref,colorlinks=true,linkcolor=blue,urlcolor=blue, citecolor=blue]%
  {hyperref}
\usepackage[shortalphabetic]{amsrefs}
\usepackage[all]{xy}
\usepackage[T1]{fontenc}
\usepackage{mathptmx}
\usepackage{extarrows}
\usepackage{microtype}
\usepackage{framed}
\usepackage{tikz}
\usepackage{tikz-cd}
\usepackage{graphicx}
\usepackage[alwaysadjust]{paralist}
\makeatletter
\providecommand\@enum@widestlabel{7}
\makeatother

\newcommand{\simplex}{%
	\begin{tikzpicture}%
		\draw [thick] (0,0) -- (1.5ex,0) -- (0,1.8ex) -- (0,0);%
	\end{tikzpicture}%
}

\newcommand{\inverted}{%
	\begin{tikzpicture}%
		\draw [thick]  (0,0) -- (1.2ex,1.5ex);%
		\draw  [thick] (0,0) -- (2ex,0);%
		\draw [thick] (2ex,0) -- (1.2ex,1.5ex);%
	\end{tikzpicture}%
}

\newcommand{\invertsimplex}{%
	\begin{tikzpicture} 
\draw[thick] (0,0)--(1.5ex,0)--(1.5ex,1.5ex)--cycle;	
\end{tikzpicture} 
\,
}

\newcommand{\polytope}{%
	\begin{tikzpicture}%
		\draw [thick] (0,0) -- (0,1.8ex);%
		\draw  [thick] (0,0) -- (1.5ex,0) -- (1.5ex,1ex) -- (0,1.8ex);%
	\end{tikzpicture}%
\, 
}

\newcommand{\invertpoly}{%
	\begin{tikzpicture}%
		\draw  [thick] (0,0) -- (2ex,0) -- (2ex,1.5ex) -- (1.2ex,1.5ex) -- (0,0);%
	\end{tikzpicture}%
}

\usepackage[centering, includeheadfoot, hmargin=.75in, vmargin=.75in,
  headheight=30.4pt]{geometry}
\newtheorem{lemma}{Lemma}[section]
\newtheorem{theorem}[lemma]{Theorem}
\newtheorem{corollary}[lemma]{Corollary}
\newtheorem{proposition}[lemma]{Proposition}

\theoremstyle{definition}
\newtheorem{definition}[lemma]{Definition}
\newtheorem{remark}[lemma]{Remark}
\newtheorem{example}[lemma]{Example}

\renewcommand{\theequation}%
{\arabic{section}.\arabic{lemma}.\arabic{equation}}
\newcommand{\CC}{\ensuremath{\mathbb{C}}} 
\newcommand{\NN}{\ensuremath{\mathbb{N}}} 
\newcommand{\PP}{\ensuremath{\mathbb{P}}} 
\newcommand{\QQ}{\ensuremath{\mathbb{Q}}} 
\newcommand{\RR}{\ensuremath{\mathbb{R}}} 
\newcommand{\ZZ}{\ensuremath{\mathbb{Z}}} 
\newcommand{\sI}{\ensuremath{\kern -1pt \mathscr{I}\kern -2pt}} 
\newcommand{\sJ}{\ensuremath{\kern -2pt \mathscr{J}\kern -2pt}} 

\newcommand{\sO}{\ensuremath{\mathscr{O}}}

\newcommand{\bb}{\ensuremath{\mathfrak{b}}}

\renewcommand{\geq}{\geqslant}
\renewcommand{\leq}{\leqslant}

\DeclareMathOperator{\mult}{mult}

\DeclareMathOperator{\Eff}{\overline{Eff}}
\DeclareMathOperator{\Mov}{Mov}

\DeclareMathOperator{\Sym}{Sym}

\DeclareMathOperator{\vol}{vol}

\newcommand{\deq}{\ensuremath{\stackrel{\textrm{def}}{=}}}

\newcommand{\Bplus}{\ensuremath{\textbf{\textup{B}}_{+} }}

\newcommand{\Bstable}{\ensuremath{\textbf{\textup{B}} }}

\newcommand{\nob}[2]{\ensuremath{\Delta_{#1}(#2)}}
\newcommand{\inob}[2]{\ensuremath{{\Delta}_{#1}(#2)}}
\newcommand{\sinob}[2]{\ensuremath{{\Delta}^S_{#1}(#2)}}

\input xy
\xyoption{all}

\usepackage{comment}
\usepackage{framed}
\usepackage{color}
\definecolor{shadecolor}{gray}{0.875}

\let\cal\mathcal
\let\frak\mathfrak
\let\bb\mathbb

\begin{document}

\title{Infinitesimal successive minima, partial jets and convex geometry}

\author{Mihai Fulger}
\address{University of Connecticut, Department of Mathematics, Storrs CT 06269, USA}
\address{Institute of Mathematics of the Romanian Academy, Bucharest 010702, Romania}
\email{mihai.fulger@uconn.edu}

\author{Victor Lozovanu}
\address{Dipartimento di Matematica, Universit\`a Degli Studi Di Genova, Genova 16146, Italy}
\email{victor.lozovanu@unige.it}

\begin{abstract}
	We introduce two sets of invariants for a line bundle at a point: infinitesimal successive minima and asymptotic partial jet separation. They are inspired by the local analogue of Ambro--Ito, and by the jet-theoretic interpretation of the Seshadri constant respectively. Under mild restrictions the two sets are equal. Moving to convex geometry, we prove that the lengths of the maximal simplex inside the generic infinitesimal Newton--Okounkov body (iNObody) of the line bundle at the point are precisely the successive minima. As application we characterize when this body is simplicial, and give examples when it is not. When the point is very general the convex body has a shape that we call Borel-fixed, a property inspired by generic initial ideals. Borel-fixed convex bodies satisfy simplicial lower bounds and polytopal upper bounds determined by their widths. For the generic iNObody of the line bundle at very general points these widths are again the infinitesimal successive minima. 
\end{abstract}
\maketitle
%\tableofcontents

\section{Introduction}

Demailly's  work on Seshadri constants \cite{Dem92} suggests the existence of a rich structure behind the data quantifying how much of the global positivity of a line bundle is visible at a point. The goal of this article is to investigate a deep synergy  between three different perspectives on this picture. Starting from geometry, we seek to understand the variation of certain base loci on the blow-up of the point. From a differential viewpoint we develop the idea of partial jet separation. Finally, we bring these to convex geometry through Newton--Okounkov bodies. This grants us access to techniques from the theory of generic initial ideals. At a basic level these perspectives underline the importance of first nonvanishing jets of Taylor expansions around a point for the study of linear series.

\subsection{Infinitesimal successive minima} 
Let us start with the geometric perspective. For a Cartier divisor $L$ on an $n$-dimensional normal projective variety $X$, one defines rational maps $X\dashrightarrow\bb P(|mL|)$ for any integer $m\geq 1$. 
The \emph{augmented base locus} ${\bf B}_+(L)$ is the largest closed subset where these maps are not embeddings. It is a numerical invariant whose extremes hold if ${\bf B}_+(L)=\varnothing$ when $L$ is ample, and  ${\bf B}_+(L)= X$ when $L$ is not \emph{big}. Here we study the variation of certain infinitesimal augmented base loci.

Fix $x\in X$ a smooth point and $\pi:{\rm Bl}_xX\to X$ the blow-up of $x$ with exceptional divisor $E$. On ${\rm Bl}_xX$ we have the ray of divisors $L_t\deq\pi^*L-tE$, for $ t> 0$, that give rise to an increasing family of base loci 
\[
({\bf B}_+(L_t))_{t>0} \ \subseteq \ {\rm Bl}_xX\ .
\]
This captures known invariants when $x\not\in{\bf B}_+(L)$. The \emph{moving Seshadri constant} $\epsilon(||L||;x)$ from \cite{Nak03,ELMNP09} measures when the ${\bf B}_+(L_t)$ start intersecting $E$. 
The \emph{Fujita--Nakayama invariant} $\mu(L;x)$ measures when ${\bf B}_+(L_t)={\rm Bl}_xX$, or equivalently $E\subset{\bf B}_+(L_t)$. There is a clear gap between the two extremes. Our \emph{infinitesimal successive minima} of $L$ at $x$ detect the dimension jumps. Specifically, for $i\geq 1$ set 
\[
\epsilon_i(L;x)\deq\min\bigl\{t>0\ \mid\ \dim{\bf B}_+(L_t)\cap E\geq i-1 \bigr\}\ .
\]
They form a nondecreasing sequence of positive real numbers, depend only on the numerical class of $L$, and satisfy homogeneity, superadditivity and continuity properties. Consequently, their definition can be extended to the entire open convex cone in $N^1(X)_{\bb R}$ consisting of those real classes $\xi$ such that $x\notin \Bplus(\xi)$.

Our particular definition was inspired by a different, local version of Ambro--Ito, introduced in \cite{AI} (see \S\ref{ss:otherwork}). Generally in the literature successive minima (or jumping numbers) measure a change in a ray $t>0$ of objects and appear in many areas of geometry. These include the successive minima of a convex body in Minkowski's Second Main Theorem; arithmetic successive minima (\cite{Zha95,FLQ24}); jumping numbers of multiplier ideals; the Hassett--Keel program \cite{Has05}; and more generally MMP with scaling \cite{BCHM10}; in $K$-stability \cite{Fuj19,Li17,BX19}, etc.

\subsection{Partial jet separation}
Our second perspective has a differential background. Recall that $|L|$ \emph{separates $s$-jets} at $x$ if the evaluation map $H^0(X,L)\to\cal O_X/\frak m_x^{s+1}$
is surjective. \cite{ELMNP09} interpret the moving Seshadri constant $\epsilon(||L||;x)=\epsilon_1(L;x)$ as asymptotic measure of jet separation. In turn, the Fujita--Nakayama invariant $\mu(L;x)=\epsilon_{n}(L;x)$  measures ``jet nontriviality'' by \cite{AI}. 
Infinitesimally we can restrict to the exceptional divisor $E$ and make use of general linear subspaces of $E$ to control the gap between the two extreme behavior of jets. We say that $L$ \emph{$i$-partially separates $s$-jets} at $x$ if the natural composition 
	\[H^0({\rm Bl}_xX,L_s)\to H^0(E,\cal O(s))\to H^0(\bb P^{n-i},\cal O(s))\]
	is surjective for some (general) linear subspace $\bb P^{n-i}\subseteq\bb P^{n-1}\simeq E$. The induced asymptotic invariant
	\[s_i(L;x)\deq\sup\bigl\{\frac pq\ \mid\ qL\ i\text{-partially separates }p\text{-jets at }x\} \ \]
shares many of the properties of $\epsilon_i(\xi;x)$. In particular, it is well-defined for real classes $\xi$ as above. The similarities are not accidental:
\begin{theorem}[Infinitesimal successive minima via partial jet separation]\label{thm:mintrosuccessjets}
	Let $X$ be a normal projective variety of dimension $n$. Let $\xi\in N^1(X)_{\bb R}$ and $x\in X\setminus{\bf B}_+(\xi)$ a smooth point. Then
	\[\epsilon_i(\xi;x)=s_i(\xi;x)\]
	for all $i=1,\ldots ,n$.
\end{theorem}

The key of the proof is Proposition \ref{prop:vanishing}, a global generation result along subschemes that avoid the augmented base locus of a big line bundle.

\subsection{Infinitesimal Newton--Okounkov bodies}
This section sets notation from the theory of Newton--Okounkov bodies, introduced by Okounkov \cite{Oko96}, and generalized by Kaveh--Khovanski\u{\i} \cite{KK12} and independently by Lazarsfeld--Musta\c t\u a \cite{LM09}. We focus on the infinitesimal version from the latter, and on a ``straightening'' of the resulting body.
Our results presented in the subsequent sections will bring several convex geometric interpretations to the infinitesimal successive minima. 
  
When $x$ is smooth, consider a complete linear flag
 \[
 Y_{\bullet}: {\rm Bl}_xX=Y_0\supset E=Y_1\supset Y_2\supset\ldots\supset Y_n\ ,
 \]
 where  $Y_i$  is a linear $\bb P^{n-i}$ in $E\simeq \bb P^{n-1}$ for $i\geq 1$. 
 For $D\sim_{\bb Q}L$ an effective $\bb Q$-divisor on $X$, one associates a valuative vector $\nu_{Y_{\bullet}}(D)=(\nu_1(D),\ldots,\nu_n(D))\in\bb Q^n_+$
 such that $\nu_1(D)=\mult_xD$ is the degree of the projectivized tangent cone $C_xD\subseteq E$, while loosely speaking $\nu_i(D)$ measures the intersection multiplicity of $C_xD$ with $Y_{i-1}$ along $Y_i$ for $i\geq 2$. See Definition \ref{def:inob} for the detailed construction. The \emph{infinitesimal Newton--Okounkov body (iNObody) of $L$ at $x$ with respect to $Y_{\bullet}$} is the topological closure in Euclidean space
\[
\nob{Y_{\bullet}}{L} \ \deq \ \overline{\{\nu_{Y_{\bullet}}(D) \ \mid \ D\geq 0, D\sim_{\bb Q}L\}}\ \subset\bb \ \RR^n_+
\]
By \cite{LM09}, we have $\vol(L)=n!\cdot {\rm vol}_{{\bb R}^n}\bigl(\nob{Y_{\bullet}}{L}\bigr)$, where the latter is the usual Euclidean volume. Using $(\nu_1,\ldots,\nu_n)\in\bb R^n$ as coordinates, then $\mu(L;x)$ is the $\nu_1$-width of $L$ at $x$, i.e., the largest $\nu_1$-coordinate of any element of $\nob{Y_{\bullet}}{L}$.
 If $x\not\in{\bf B}_+(L)$ and $0\leq t<\mu(L;x)$, the vertical slices $\nob{Y_{\bullet}}{L}_{\nu_1=t}\deq\nob{Y_{\bullet}}{L}\cap\{t\}\times\bb R^{n-1}_+$ are determined in \cite{LM09} by the restricted linear series of $L_t$ to $E$ and by $Y_{\bullet}$. Hence, $\nob{Y_{\bullet}}{L}$ encodes how the tangent cones of sections of powers of $L$ vanish along the linear flag $Y_{\bullet}$.

With the partial jet separation in mind, we study iNObodies with respect to very general\footnote{A property holds for the \emph{very general} object in a family of objects parameterized by an irreducible base scheme, if it is true outside of an at most countable union of proper closed subsets of the base.} linear flags in $E$. 
Similar to the construction of the generic initial ideal \cite{Har66,Gra72}, under this assumption \cite{LM09} shows that $\nob{Y_{\bullet}}{L}\subset\bb R^n_+$ is independent of $Y_{\bullet}$. We denote it by $\inob{x}{L}\subseteq \RR^n_{+}$ and call it \emph{the generic iNObody of $L$ at $x$}. Similarly, when $x$ is very general in $X$ and $Y_{\bullet}$ very general over $x$, the resulting body $\inob{X}{L}\subseteq \RR^n_+$ is the \emph{generic iNObody of $L$}. The genericity condition is a double-edged sword. The bodies $\inob{x}{L}$ and $\inob{X}{L}$ may be difficult to compute, but their shapes tend to be simpler. For example, in Figure \ref{fig:genericvsspecial} the triangle on the left is the body of $\cal O_{\bb P^1\times\bb P^1}(1,1)$ at any point $x$, when the point $Y_2$ in the infinitesimal flag over $x$ is not torus fixed. The shaded parallelogram is the result when $Y_2$ is torus fixed.
\begin{figure}
	\begin{tikzpicture}[scale=1.50]
		\draw (0,0) node[below]{$(0,0)$};
		\draw (2,0) node[below]{$(2,0)$};
		\draw (1,1) node[above]{$(1,1)$};
		\draw (2,1) node[above]{$(2,1)$};
		\draw (1,0) node[below]{$(1,0)$};
		\draw[thick, fill=gray, opacity=0.4] (0,0)--(2,0)--(1,1)--cycle;
		\draw[dashed, fill=gray!20, opacity=0.3] (0,0)--(1,1)--(2,1)--(1,0);
		\draw (5,0) node[left]{$(0,0)$};
		\draw (6,0) node[right]{$(1,0)$};
		\draw (5,2) node[left]{$(0,2)$};
		\draw (5,1) node[left]{$(0,1)$};
		\draw (6,1) node[right]{$(1,1)$}; 
				\draw[thick, fill=gray, opacity=0.4] (5,0)--(6,0)--(5,2)--cycle;
						\draw[dashed, fill=gray!20, opacity=0.3] (5,0)--(6,0)--(6,1)--(5,1);
	\end{tikzpicture}
	\centering\caption{Generic vs. special flag choices for $\inob{Y_{\bullet}}{L}$ and respectively $\sinob{Y_{\bullet}}{L}$}\label{fig:genericvsspecial}
\end{figure}

To simplify the convex shapes we work with, we consider usually the ``\emph{straightened up}'' (see Figure \ref{fig:genericvsspecial}) version
$\Delta^S=S(\Delta)$ of any convex set $\Delta\subseteq \RR^n$, where $S$ is the volume-preserving linear transformation
\begin{equation}\label{eq:introS} 
	\bb R^n\overset{S}{\to}\bb R^n: (\nu_1,\ldots,\nu_n)\mapsto (\alpha_1,\ldots,\alpha_n)=(\nu_2,\nu_3,\ldots,\nu_n,\nu_1-\nu_2-\ldots-\nu_n)\ .
\end{equation} 
The straightened up bodies can also be constructed directly on $X$ by using the degree lexicographic valuation. This was also considered by Witt Nystr\"om in \cite{WN18, WN23}.

\subsection{The shape of generic iNObodies at arbitrary smooth points}

Our starting point is \cite{KL17} which highlights the connection between iNObodies and simplices by proving:
\begin{equation}\label{eq:introKL17}
	\simplex\,_{\epsilon(||L||;x)}\ \subseteq\ \sinob{Y_{\bullet}}{L}\ \subseteq\ \simplex\,_{\mu(L;x)}\ ,
\end{equation}
where $\simplex\,_t$ is the standard $n$-dimensional simplex of size $t$ in $\bb R^n$.
Furthermore, $\epsilon(||L||;x)$ and $\mu(L;x)$ are maximal and respectively minimal with this property and can be read off from \emph{every} $\nob{Y_{\bullet}}{L}$.

 When working with very general flags we will prove that the other infinitesimal successive minima can also be recovered from the iNObody. Consider the simplex
\[\simplex(t_1,\ldots,t_{n})\deq\textup{convex hull}(t_1{\bf e}_1,\ldots,t_n{\bf e}_n)\ ,\]
so that $\simplex\,_t=\simplex(t,\ldots,t)$. For example, on the right in Figure \ref{fig:genericvsspecial} we see $\simplex(1,2)$.
\begin{theorem}[Convex geometric interpretations of the infinitesimal successive minima]\label{thm:introreadsuccessive}
With the assumptions of Theorem \ref{thm:mintrosuccessjets},
	\begin{equation}\label{eq:introsimplexinclusion}\simplex\bigl(\epsilon_1(\xi;x),\ldots,\epsilon_{n}(\xi;x)\bigr)\subseteq\sinob{x}{\xi} \ ,\end{equation} 
			and $\epsilon_i(\xi;x)$ are maximal with this property. In particular $\vol(\xi)\geq\prod\nolimits_{i=1}^{n}\epsilon_i(\xi;x)$. Furthermore, $\sinob{x}{\xi}$ is simplicial if and only if $\vol(\xi)=\prod_{i=1}^n\epsilon_i(\xi;x)$, in which case $(\ref{eq:introsimplexinclusion})$ is an equality.
\end{theorem}

\noindent 
A version of the resulting volume inequality also appeared in \cite{AI}. The simplicial criterion says that computing the numerical invariants $\epsilon_i(L;x)$ and $\vol(L)$ sometimes determines what tangent cones we can expect to find for divisors in $|mL|$. For very general $x$, we will see in Theorem \ref{thm:mintroupperbounds} that the $\epsilon_i(L;x)$ bound these tangent cones even in the non-simplicial case.

For proving inclusion \eqref{eq:introsimplexinclusion}, we use the partial jet interpretation of $\epsilon_i(\xi;x)$ to construct sections whose valuations determine the vertices of the simplex. 
We note that unlike \eqref{eq:introKL17}, valid for all infinitesimal flags, the refined inclusion \eqref{eq:introsimplexinclusion} may fail for special flags, as seen in Figure \ref{fig:genericvsspecial}, where $\epsilon_1(\xi;x)=1$, $\epsilon_2(\xi;x)=2$, and $\vol(\xi)=2$. Other examples, including non-simplicial ones, are described in Section \ref{s:examples}.

The optimality of the inclusion \eqref{eq:introsimplexinclusion} (and hence the remaining parts) follows from Theorem \ref{thm:mintroslicebounds}.(3) below. Its roots lie in the algebraic side of our convex constructions, since all the data is computed on $E\simeq \PP^{n-1}$. As motivation note that by "sliding down" the parallelogram in the special flag case from the left picture in Figure \ref{fig:genericvsspecial},  we obtain the triangle in the generic case. Its shape reminds us of the structure of exponent vectors of monomials in a Borel-fixed ideal (equivalently a generic initial ideal). Due to this, we define a convex set $P\subset\bb R^{n}_+$ to be \emph{Borel-fixed}\footnote{We do not claim an actual Borel group action on our sets, the name is simply terminology.} if for all $i=1,\ldots ,n$, the following property holds
\begin{equation}\label{eq:introborel}
(\nu_1,\ldots,\nu_n)\in P\Longrightarrow (\nu_1,\ldots,\nu_{i-1},0,\nu_i+\nu_{i+1},\nu_{i+2},\ldots,\nu_n)\in P \ .
\end{equation}
When $i=n$, we understand the point on the right as $(\nu_1,\ldots,\nu_{n-1},0)$.  

 Typical examples of Borel-fixed shapes are the simplex $\simplex(t_1,\ldots,t_n)$, and the polytope
\[
\polytope(t_1,\ldots, t_n)\ \deq \ \{(\nu_1,\ldots ,\nu_n)\in \RR^n_+ \ \mid \  \nu_1+\ldots+\nu_i\leq t_i\ \text{ for all }i=1,\ldots, n\} \ ,
\] 
both considered for $0\leq t_1\leq\ldots\leq t_n$. The polytope above appears arbitrary, but it is in fact generated by one of its vertices as a Borel-fixed body. When $n=2$, any convex shape $\simplex(t_1,t_2)\subseteq P\subseteq\polytope(t_1,t_2)$ is Borel-fixed, but the property is more restrictive when $n\geq 3$.  A version of $\polytope$ was also used by \cite{AI}.
%More generally, $\simplex(t'_1,\ldots,t'_n)\subseteq\polytope(t_1,\ldots,t_n)$ if and only if $t'_i\leq t_i$ for all $i$.

For a Borel-fixed convex shape $P$ we consider its $\nu_i$-\emph{width}, that is defined to be 
\[
w_i(P)\deq\max\{w\ \mid\ \exists (\nu_1,\ldots,\nu_n)\in P\ ,\ \nu_i=w\} \ .
\]
The Borel-fixed property forces the $w_i(P)$ to form a nondecreasing sequence $0\leq w_1(P)\leq\ldots\leq w_n(P)$, just like the infinitesimal successive minima. It also imposes polytopal bounds
\[\simplex\bigl(w_1(P),\ldots,w_n(P)\bigr)\ \subseteq\ P\ \subseteq\ \polytope\bigl(w_1(P),\ldots,w_n(P)\bigr)\ .\]

The iNObody $\inob{x}{\xi}$ is not Borel-fixed, as we see in Figure \ref{fig:genericvsspecial} on the left (note for future reference that on the right in the same picture, $\sinob{x}{\xi}$ is Borel-fixed). The situation improves for the vertical slices $\inob{x}{\xi}_{\nu_1=t}$. They are constructed by evaluating a graded linear system on $E\simeq \PP^{n-1}$ with respect to a very general linear flag. Its connection to degree lexicographic orders allows us to use ideas developed for generic initial ideals in \cite{Gal74, BS87} and show parts $(1)$ and $(2)$ of the following theorem.
\begin{theorem}[Slice bounds]\label{thm:mintroslicebounds}
	With the assumptions of Theorem \ref{thm:mintrosuccessjets}, for all $0< t<\mu(\xi;x)$ we have 
	\begin{enumerate}
	\item  $\inob{x}{\xi}_{\nu_1=t}\subset\{t\}\times\bb R^{n-1}_+$ is Borel-fixed with respect to the coordinates $(\nu_2,\ldots,\nu_n)$. 
	\smallskip
	
	\item $\{t\}\times\simplex\bigl(w_2(t),\ldots,w_n(t)\bigr)\ \subseteq\ \inob{x}{\xi}_{\nu_1=t}\ \subseteq\ \{t\}\times\polytope\bigl(w_2(t),\ldots, w_n(t)\bigr)$, where  $w_i(t)$ is the $\nu_i$-width of $\inob{x}{\xi}_{\nu_1=t}$ for all $2\leq i\leq n$. Moreover, $0<w_2(t)\leq\ldots\leq w_n(t)\leq t$. 
	\smallskip
	
	\item $\epsilon_i(\xi;x)=\max\{t\in[0,\mu(\xi;x)]\ \mid\ w_{i+1}(t)=t\}$. In particular, the segment joining the origin to $\epsilon_i(\xi;x)({\bf e}_1+{\bf e}_{i+1})$ is an edge of $\inob{x}{\xi}$.
	%$\epsilon_i(\xi;x)$ is the length of the edge of $\inob{x}{L}$ on the ray $\RR_+\cdot ({\bf e}_1+{\bf e}_{i+1})$, where for $i=n$ we consider the ray $\RR\cdot {\bf e}_1$.
		\end{enumerate}
\end{theorem}

\noindent% We prove these in Corollary \ref{cor:borelone}, Theorem \ref{thm:slices} and Theorem \ref{thm:successiveminimaviaslicewidths}. 
The convex geometric interpretation in the last part uses an inductive B\' ezout argument detailed in Theorem \ref{thm:successiveminimaviaslicewidths}. As previously mentioned, it is the key for proving that \eqref{eq:introsimplexinclusion} is optimal.
When $n=2$, $\inob{x}{\xi}_{\nu_1=t}$ is the vertical segment $\{t\}\times [0,w_2(t)]$. When $n=3$, the slices are two dimensional and Figure \ref{fig:mslicebounds} illustrates the bounds. Finally, $(3)$ supports the idea that local positivity of $\xi$ at $x$ can be studied via the ray of divisors $\xi_t$. 

\begin{figure}
	\begin{tikzpicture}[scale=1.5]
		\draw [->](-0.25,0)--(1.25,0) node[above right]{$\nu_2$};
		\draw [->](0,0) node[below left]{$(t,0,0)$}--(0,2.25) node[above]{$\nu_3$};
		\draw[thick,dashed] (1,0) node[below right]{$(t,w_2(t),0)$}--(0,2) node[right]{$(t,0,w_3(t))$};
		\draw[thick,dashed] (1,0)--(1,1)--(0,2);
		\draw[ultra thick] (0,2) to [out=310,in=100](1,0);
		\draw[fill=gray!20,opacity=0.5](0,0) -- (0,2) to [out=310,in=100](1,0) -- (0,0)--cycle;
	\end{tikzpicture}
	\centering\caption{Vertical slice of $\inob{x}{\xi}_{\nu_1=t}$ in dimension $3$.}\label{fig:mslicebounds} 
\end{figure}

\subsection{The shape of generic iNObodies for a very general point}\label{ss:verygenera}
When $x\in X$ is a very general point, differentiation techniques similar to \cite{EKL95}, but adjusted to jets, allow us to improve Theorem \ref{thm:mintroslicebounds} from the slices of $\inob{x}{\xi}$ to the entire straightened body $\sinob{x}{\xi}$ (see \eqref{eq:introS}).  In Theorem \ref{thm:upperbound} we prove:

\begin{theorem}[Upper bounds at very general points]\label{thm:mintroupperbounds}
	Let $X$ be a normal complex projective variety of dimension $n$. Let $\xi$ be a big numerical class in $N^1(X)_{\bb R}$. Let $\epsilon_i(\xi)$ and $\sinob{X}{\xi}\subset\bb R^n_+$ denote the common value of $\epsilon_i(\xi;x)$ and respectively $\sinob{x}{\xi}\subset\bb R^n_+$ when $x\in X$ is very general. Then $\sinob{X}{\xi}$ is Borel-fixed and its width with respect to the $\alpha_i$ coordinate is $\epsilon_{i}(\xi)$. In particular
	\[\simplex\bigl(\epsilon_1(\xi),\ldots,\epsilon_{n}(\xi)\bigr)\ \subseteq\ \sinob{X}{\xi}\ \subseteq\ \polytope\bigl(\epsilon_1(\xi),\ldots,\epsilon_{n}(\xi)\bigr)\ ,\]
	and $\frac{1}{n!}\cdot \vol(\xi)\leq{\rm vol}_{\bb R^n}\polytope(\epsilon_1(\xi),\ldots,\epsilon_n(\xi))\leq \prod_{i=1}^n\epsilon_i(\xi)$.
\end{theorem}

\noindent Recall that we use coordinates $(\alpha_1,\ldots,\alpha_n)$ for the ambient $\bb R^n_+$ of the straightening $\sinob{x}{\xi}$. A version of the volume inequalities was also proved by \cite{AI}. The theorem provides yet another convex geometric interpretation for the infinitesimal successive minima at a very general point. They are the widths of the largest box that contains $\sinob{x}{\xi}$. Similarly to the simplicial criterion, if ${\rm vol}(L)=n!\cdot{\rm vol}_{\bb R^n}\polytope(\epsilon_1(L),\ldots,\epsilon_n(L))$, then $\sinob{x}{L}$ equals the upper bound $\polytope(\epsilon_1(L),\ldots,\epsilon_n(L))$, and so again the $n+1$ numbers $\vol(L)$ and $\epsilon_i(L)$ essentially tell us what tangent cones appear for divisors in $|mL|$. 

 Regarding the necessity of the very general assumption, for arbitrary point $x\in X$ and $1\leq i\leq n-1$, the function $w_{i+1}(t)$ is concave on $(0,\mu(\xi;x))$ and the interval where $w_{i+1}(t)=t$ is $(0,\epsilon_i(\xi;x)]$ by Theorem \ref{thm:mintroslicebounds}.  For very general $x$, the function is furthermore nonincreasing for $t\geq\epsilon_i(\xi)$. This result generalizes the surface case in \cite[Proposition 4.2.(2).(b)]{KL18}.
Example \ref{ex:verygeneralpointnecessary} shows that the global widths of $\inob{x}{\xi}$ when $x$ is not very general can be larger than $\epsilon_i(\xi;x)$. 
In Example \ref{ex:upperboundsharp} at very general points, we see that the upper bound $\sinob{X}{\xi}\ \subseteq\ \polytope(\epsilon_1(\xi),\ldots,\epsilon_{n}(\xi))$ is sharp. We again mention that one of the reasons for stating our results for $\sinob{x}{\xi}$ is that the tilted versions $\inob{x}{\xi}$ are never Borel-fixed when $n\geq 2$.

\subsection{Examples and applications}
A benefit of the theory above is that once one knows geometrical aspects of the pair $(X,L)$, then we could expect to be able to compute the shape of $\inob{x}{L}$. In this regard we are able to describe many examples in Section \ref{s:examples} (see also \cite{FL25,FL25b,FL25c}). On symmetric products of a curve and on quadric hypersurfaces, these sets are simplices as in Theorem \ref{thm:mintroupperbounds} due to simple geometric reasons. The simplicial form holds also for certain polarizations on products of curves and three-dimensional non-hyperelliptic Jacobian, but the geometry becomes much more involved. In the first case we describe infinitesimal successive minima with toric-inspired arguments, and in the second case the doubling map on the Jacobian plays a significant role. Similar arguments can be made for hyperelliptic Jacobians, but they are not yet sufficient to fully describe the generic iNObody. They do show that its shape is not simplicial.

Another merit of the theory above is its perceived potential towards applications. 
For example, we uncover an interesting connection between $\inob{x}{\xi}$ for an ample class $\xi$ and the Seshadri constant at $x$, as defined in \cite{Ful21}, of the curve class $(\xi^{n-1})$. In Theorem \ref{thm:curveseshadribounds} we prove in particular the following:

\begin{theorem}\label{thm:introcurveseshadribounds}
	Let $X$ be a complex projective manifold of dimension $n$, let $x\in X$ and let $\xi\in N^1(X)_{\bb R}$ be an ample class. Then
\begin{enumerate}[(1)]
\item $\prod\nolimits_{i=1}^{n-1}\epsilon_i(\xi;x)\leq\epsilon(( \xi^{n-1});x)\leq\frac{(\xi^n)}{\mu(\xi;x)}$ ,
\smallskip

\item If $x$ is very general, then $(n-1)!\cdot {\rm vol}_{\bb R^{n-1}}\sinob{x}{\xi}_{\alpha_n=0}\leq\epsilon((\xi^{n-1});x)$.
\end{enumerate}
\end{theorem}
As a consequence of this and Theorem \ref{thm:introreadsuccessive}, we compute the Seshadri constant of the curve class $(\xi^{n-1})$ when $\sinob{x}{\xi}$ is simplicial.
\smallskip

\subsection{{\bf Comparison with other work in the literature}}\label{ss:otherwork} Our main inspiration comes from \cite{KL17,AI} and from our examples in \cite{FL25,FL25b,FL25c} on genus 3 Jacobians, products of curves, and other product threefolds. \cite{KL17} prove the inclusions \eqref{eq:introKL17} that suggest how simplices appear naturally when studying iNObodies. 
	\cite{CHPW18} also consider nonstandard simplices $\simplex(t_1,\ldots,t_n)$ in the context of (non-infinitesimal) Newton--Okounkov bodies from flags on $X$. %They recover inequalities similar in spirit to the inclusions \eqref{eq:introKL17}.
	
\cite{AI} introduced a different notion of successive minima for a line bundle $L$ at a possibly singular point $x\in X$. When $x\notin{\bf B}_+(L)$ is smooth, then $\epsilon_i^{\rm loc}(L;x)$ measures when the ${\bf B}_+(L_t)$ start having components of dimension at least $i$ that meet $E$ properly. See Definition \ref{def:aiminima}. These successive minima also form an increasing sequence. Possibly out of a desire to work intrinsically on $X$, irreducible components of ${\bf B}_+(L_t)$ that are fully contained in $E$ are not considered. Such components can appear, see Example \ref{ex:hyper3}, and were  identified in \cite{CN14} as obstacles in improving known lower bounds on Seshadri constants at very general points. 
Interestingly the infinitesimal minimum $\epsilon_{n-1}(L;x)$ also appears in \cite{CN14}. 

We show that $\epsilon_i(L;x)\leq\epsilon_{i}^{\rm loc}(L;x)$. The inequality is strict in Example \ref{ex:hyper3}. This may explain why it was difficult to obtain a jet interpretation for $\epsilon_i^{\rm loc}(L;x)$ when $i\neq 1,n$ in \cite{AI}. 
When working directly on $X$, the irreducible components in $\Bplus(L_t)$ that are fully contained in $E$ do not appear to interfere with intersection theoretic arguments. This explains why the bounds on ${\rm vol}(L)$  in terms of local successive minima from \cite{AI} can sometimes be better than those coming from \eqref{eq:introsimplexinclusion}. An instance of this is also our Theorem \ref{thm:curveseshadribounds}.(2). Other times, like in Theorem \ref{thm:mintroupperbounds}, the bounds are weaker. 

The idea that we exploit in this paper is that volume bounds arise from a convex geometric picture. Working infinitesimally (on the blow-up), instead of locally around $x$ on $X$, helped us achieve it. The infinitesimal perspective also played an important role in the main results of \cite{Loz24} and generic iNObodies were used extensively in \cite{Loz20}.

%Our study of $\inob{x}{\xi}$ was also inspired by the ongoing explicit computations of generic iNObodies that we are describing in \cite{FL25} on genus 3 Jacobians, and in \cite{FL25b} on products of curves.
\smallskip

\subsection*{Acknowledgments} The first named author was partially supported by the Simons travel grant no.~579353.  The second named author was partially supported by the Research Project PRIN 2020 - CuRVI, CUP J37G21000000001, and PRIN 2022 - PEKBY, CUP J53D23003840006. Furthermore, he wishes to thank the MIUR Excellence Department of Mathematics, University of Genoa, CUP D33C23001110001. He is also a member of the INDAM-GNSAGA. 

We thank Sam Altschul, Florin Ambro, Aldo Conca, Robert Lazarsfeld, and Takumi Murayama for useful conversations. The idea of translating the volume bounds of \cite{AI} into bounds on generic iNObodies also came up in discussions between the authors and Alex K\" uronya.
The second author is furthermore thankful for the conversations with Alex K\"uronya (and indirectly with Lawrence Ein) about the impact that differentiating sections has on the shape of generic iNObodies at very general points.

\section{Infinitesimal successive minima and partial jets}

In this section we introduce and study the basic properties of the infinitesimal successive minima and the asymptotic partial jet separation order. They form increasing sequences of strictly positive numerical invariants that satisfy superadditivity and continuity properties on an open cone inside the Néron--Severi space. The main result of this section is Theorem \ref{thm:mintrosuccessjets} where we identify these two invariants for normal ambient spaces. In Section \ref{ss:motivation} we explain how classical invariants like the Seshadri constants and the Fujita--Nakayama invariants (widths) as well as the more recent local successive minima of \cite{AI} motivated our constructions. The experts may wish to jump straight to Definition \ref{def:infinitesimalsuccessiveminima} and call back as needed.

\subsection{Base loci}This subsection collects basic classical properties of base loci. 
 Let $X$ be a projective variety over an algebraically closed field and  $L$ be a line bundle on $X$. The \emph{base locus} of $L$ is the degeneracy locus ${\rm Bs}(L)\subseteq X$ with its subscheme structure of the evaluation morphism $H^0(X,L)\otimes\cal O_X\to L$. Its asymptotic version the \emph{stable base locus} of $L$ is the closed subset 
 \[
 {\bf B}(L)\ = \ \bigcap\nolimits_{m\geq 1}{\rm Bs}(mL) .
 \]
By Noetherianity ${\bf B}(L)={\rm Bs}(mL)$ for $m\gg 0$ sufficiently divisible. Thus, ${\bf B}(L)={\bf B}(mL)$ for all $m\geq 1$. This definition extends naturally to $\bb Q$-Cartier $\bb Q$-divisors. 

The base loci  above are not numerical invariants and cannot be extended to real classes in the N\' eron--Severi space of $X$. A fix is to consider the \emph{numerical base locus} of $L$ 
 \[
 {\bf B}_{\rm num}(L)=\bigcap\nolimits_D{\rm Supp}(D)
 \]
  as $D$ ranges through effective $\bb R$-Cartier $\bb R$-divisors numerically equivalent to $L$. It is a homogeneous numerical invariant and can be defined for any class $\xi\in N^1(X)_{\RR}$.

Motivated by similar concerns, \cite{ELMNP06} introduces two new base loci by perturbing the stable base locus with small ample classes. In \cite[\S 1]{ELMNP06} it is required for $X$ to be normal, but this assumption does not appear to be used in the proofs. First,  the \emph{augmented (or nonample) base locus} of $L$ is 
\[
{\bf B}_+(L)\ = \ {\bf B}(L-\frac 1mH)
\]
for all $m\gg 0$ and any ample divisor $H$ on $X$.  It is also the intersection of all ${\bf B}(L-A)$, where $A$ is an ample $\bb R$-Cartier $\bb R$-divisor, and $L-A$ is a $\bb Q$-Cartier $\bb Q$-divisor. It can be defined by using a sequence $A_m$ of ample divisors, numerically converging to $0$. Thus the augmented base locus is a numerical homogeneous invariant and can be defined for any class $\xi\in N^1(X)_{\bb R}$. In terms of positivity, $\xi$ is \emph{ample} precisely when $\Bplus(\xi)=\varnothing$, and $\xi $ is \emph{big} whenever $\Bplus(\xi)\neq X$. Other basic properties are summarized below:
\begin{proposition}\label{prop:bplus}
Let $X$ be a projective variety and $\xi\in N^1(X)_{\RR}$ a big numerical class on $X$. Then
\begin{enumerate}
\item ${\bf B}_+(\xi-\xi')\subseteq{\bf B}_+(\xi)$ for any small enough numerical class  $\xi'\in N^1(X)_{\bb R}$. If $\xi'$ is ample (and small), then equality holds.
\smallskip

\item
 ${\bf B}_+(\pi^*\xi)=\pi^{-1}{\bf B}_+(\xi)\cup{\rm Exc}(\pi)$, for a birational morphism $\pi:\overline X\to X$ of normal varieties.
\smallskip

\item When $X$ is normal and $\xi$ is the class of a  $\bb Q$-Cartier $\bb Q$-divisor $L$, then ${\bf B}_+(\xi)$ is the closed subset of points where $|mL|$ is not an isomorphism for $m\gg 0$ sufficiently divisible.
\end{enumerate}
\end{proposition}
\noindent The first statement is proved in \cite[Corollary 1.6]{ELMNP06} when $X$ is normal, but the proof does not use this condition. The second statement was proved in \cite[Proposition 2.3]{BBP13} and the third in \cite{BCL14}. 

 Second, the \emph{restricted (or diminished, or nonnef) base locus} of $L$ is defined to be
 \[
 {\bf B}_-(L)=\bigcup\nolimits_{m\geq 1}{\bf B}(L+\frac 1mH).
 \]
 It is also the union of ${\bf B}_+(L+A_m)$, where $A_m$ is a sequence of ample $\bb R$-Cartier $\bb R$-divisors, numerically converging to 0. Consequently, it is a numerical, homogeneous invariant, defined for any class $\xi\in N^1(X)_{\RR}$. It detects when $\xi$ is \emph{nef} (${\bf B}_-(\xi)=\varnothing$), and \emph{pseudo-effective} (${\bf B}_-(\xi)\neq X$).
 
 By definition, ${\bf B}_-({\xi})$ is a countable union of closed subsets.  It is not known if ${\bf B}_-(L)$ is always closed for Cartier divisors/line bundles, but \cite{Les14} provides examples of an $\bb R$-divisor whose ${\bf B}_-$ locus is not closed.  If $X$ is normal then ${\bf B}_-(\xi)$ contains no prime divisors if and only if $\xi$ is \emph{movable} (in the closure of the cone generated in $N^1(X)_{\bb R}$ by divisors moving in linear series without fixed divisorial components). Moreover, when $X$ is normal we can associate to ${\bf B}_-({\xi})$ a ``subscheme structure'', via a sequence of asymptotic multiplier ideals, as shown in \cite[Corollary 2.10]{ELMNP06}.

The following proposition summarizes other important proprieties of these base loci.
\begin{proposition}\label{prop:bminus}
Let $X$ be a projective variety and $\xi\in N^1(X)_{\RR}$ a pseudo-effective class on $X$. Then
\begin{enumerate}
\item $ {\bf B}_-(\xi)\subseteq{\bf B}_{\rm num}(\xi)\subseteq{\bf B}_+(\xi)$. If $L$ is a $\bb Q$-Cartier $\bb Q$-divisor, then ${\bf B}_-(L)\subseteq{\bf B}_{\rm num}(L)\subseteq{\bf B}(L)\subseteq{\bf B}_+(L)$.
\smallskip

\item ${\bf B}_-(\xi-\xi')={\bf B}_{\rm num}(\xi-\xi')={\bf B}_+(\xi-\xi')$ for any small enough ample class  $\xi'\in N^1(X)_{\bb R}$.  
\smallskip

\item ${\bf B}_{\pm}(\xi+\xi')\subseteq{\bf B}_{\pm}(\xi)\cup{\bf B}_-(\xi')$ for any $\xi'\in N^1(X)_{\RR}$.
\smallskip

\item  When $X$ is smooth and $\pi:\overline X\to X$ is a surjective morphism, then ${\bf B}_-(\pi^*\xi)=\pi^{-1}{\bf B}_-(\xi)$.
\end{enumerate}
\end{proposition}
\noindent The second statement follows from Proposition \ref{prop:bplus}. For example, the statement for $\Bplus$ in $(3)$ is due to the inclusion ${\rm Bs}(\xi+\xi'-A)\subset{\rm Bs}(\xi-2A)\cup{\rm Bs}(\xi+A)$ for any small enough ample class $A$ on $X$. The fourth statement was proved in \cite{FR23}. We introduce an important open convex cone.

\begin{lemma}
	Let $X$ be a projective variety and $x\in X$. The topological closure of the open cone 
	\begin{equation}\label{def:bigxcone}
		{\rm Big}_x(X)\deq \{\xi\in N^1(X)_{\bb R}\ \mid\ x\not\in{\bf B}_+(\xi)\}
	\end{equation}
	is ${\rm Psef}_x(X)=\{\xi\in N^1(X)_{\bb R}\ \mid\ x\not\in{\bf B}_-(\xi)\}$. 
\end{lemma}

\begin{proof}
	If $\xi\in {\rm Psef}_x(X)$  and $H$ is ample, then $\xi=\lim_{m\to\infty}\xi+\frac 1mH$ and ${\bf B}_+(\xi+\frac 1mH)\subseteq{\bf B}_+(\frac 1mH)\cup{\bf B}_-(\xi)$ does not contain $x$. It remains to prove that ${\rm Psef}_x(X)$ is a closed cone. Let $(\xi_m)_{m\geq 1}\in{\rm Psef}_x(X)$ be a sequence that converges to $\xi\in N^1(X)_{\bb R}$. If $x\in{\bf B}_-(\xi)$, then $x\in{\bf B}_+(\xi+\frac 1NH)$ for some large $N\geq 1$. For sufficiently large $m$ depending on $N$ we have that $\frac 1NH+\xi-\xi_m$ is ample. Then $x\in{\bf B}_+(\xi+\frac 1NH)\subseteq{\bf B}_+(\frac 1NH+\xi-\xi_m)\cup{\bf B}_-(\xi_m)={\bf B}_-(\xi_m)$ is a contradiction.
\end{proof}

 %\begin{definition}[Stable divisors]
%	Say that $L$ is a \emph{stable} divisor if ${\bf B}_+(L)={\bf B}_-(L)$. 
%\end{definition}
%\noindent When $L$ is a $\bb Q$-Cartier $\bb Q$-divisor, the loci above are also equal to ${\bf B}(L)$.
%\smallskip
%\noindent{{\bf $(6)$. Openness and densitiy of stability.}} If $L$ is stable, then ${\bf B}_{+}(L)={\bf B}_{\pm}(L+D)={\bf B}_-(L)$ for all sufficiently small $D$ in $N^1(X)$. In particular $L+D$ is also stable. If $L$ is not stable, then $L+A$ is stable for all sufficiently small ample $A\in N^1(X)$.
%\smallskip

\subsection{Infinitesimal successive minima}
Let $X$ be a projective variety of dimension $n$ over an algebraically closed field. Let $x\in X$ be a smooth closed point and $\xi\in N^1(X)_{\RR}$ be a pseudo-effective class. Let $\pi:\overline X\to X$ be the blow-up of $x$ with exceptional divisor $E\simeq\bb P^{n-1}$.
We want to study the positivity of the classes
\[
\xi_t\deq \pi^*\xi-tE 
\]
for $t\geq 0$. From this ray in $N^1(\overline X)_{\bb R}$ one determines classical local positivity invariants such as Seshadri constants or Fujita--Nakayama invariants.
More generally, inspired by the local picture developed in \cite{AI}, we introduce here a set of invariants that measure the change in dimension of ${\bf B}_{\pm}(\xi_t)$, when moving on the exceptional ray. In order to do so we first study the basic properties of this family of base loci. The results here extend those of \cite[Lemma 1.3]{Loz18} from ample divisors to big $\bb R$-classes.

\begin{remark}\label{rem:convexcombination}
	Several times in the course of our proofs we will use inclusions
	\[B(L_b)\subseteq B(L_c)\cup(B(L_a)\cap E)\]
	whenever $0\leq a\leq b\leq c$ where $B$ is any of the base loci we work with. Furthermore, when $B\in{\bf B}_{\pm}$, then we can replace $B(L_a)$ by ${\bf B}_-(L_a)$ on the right.
	These hold usually by homogeneity and Proposition \ref{prop:bminus}.$(3)$ because $L_b$ is nonnegative combination of $L_c$ and $E$, and nonnegative convex combination of $L_a$ and $L_c$.
\end{remark}

\begin{lemma}[${\bf B}$ locus properties]\label{lem:Bray}
	Let $X$ be a projective variety, $L$ be a big Cartier divisor on $X$, and $x\not\in{\bf B}(L)$ be a smooth point. 
	The stable base loci ${\bf B}(L_s)$ where $s\in\bb Q_+$ form an increasing family of closed subsets of $\overline X$. In particular, if $t\in\bb R_+$, then $\bigcap_{s>t,\ s\in\bb Q}{\bf B}(L_s)={\bf B}(L_{s_0})$ for all $s_0>t$ rational sufficiently close to $t$. The same holds if $L$ is an $\bb R$-Cartier $\bb R$-divisor, $s$ is real, and we replace ${\bf B}$ by ${\bf B}_{\rm num}$ throughout.
\end{lemma}
\begin{proof}Since $x\not\in{\bf B}(L)$, then ${\bf B}(L_0)\cap E=\varnothing$. Let $0\leq s<s'$ be rational. From Remark \ref{rem:convexcombination}, we deduce ${\bf B}(L_s)\subseteq{\bf B}(L_{s'})\cup (E\cap{\bf B}(L_0))={\bf B}(L_{s'})$.
	The last statement is a consequence of Noetherianity. The $\bb R$-divisor case for ${\bf B}_{\rm num}$ is analogous.
\end{proof}

\begin{lemma}[${\bf B}_+$ locus properties]\label{lem:B+interpretation}
	Let $X$ be a projective variety and $x\in X$ a smooth point. For $\xi\in{\rm Big}_x(X)$: 
	\begin{enumerate}
		\item ${\bf B}_+(\xi_0)=\pi^{-1}{\bf B}_+(\xi)\cup E$.
		\smallskip
		
		\item ${\bf B}_+(\xi_t)=\pi^{-1}{\bf B}_+(\xi)$ is disjoint from $E$ for all sufficiently small $t>0$.
		\smallskip

		\item ${\bf B}_+(\xi_t)$ form an increasing family for $t>0$, and ${\bf B}_+(\xi_t)={\bf B}_+(\xi_{t'})$ for all $t'>t$ sufficiently close to $t$.
	\end{enumerate}
\end{lemma}
\begin{proof}
	Let $L$ be an $\bb R$-Cartier $\bb R$-divisor that represents $\xi$.
	
	$(1)$. When $X$ is normal, it follows from Proposition \ref{prop:bplus}. For our more general setting, let $A$ be a numerically small ample $\bb R$-Cartier $\bb R$-divisor and $r>0$ a small rational number so that $L-A$ is a $\bb Q$-Cartier $\bb Q$-divisor, ${\bf B}_+(L)={\bf B}(L-A)$, the class $\pi^*A-rE$ is ample, and ${\bf B}_+(L_0)={\bf B}(\pi^*(L-A)+rE)$. As $x$ is smooth, $\pi$ is a fiber space and $\pi_*\cal O(mE)=\cal O_X$ for all $m\geq 0$. Thus ${\bf B}(\pi^*(L-A)+rE)=\pi^{-1}{\bf B}(L-A)\cup E$. We did not use the assumption $x\not\in{\bf B}_+(\xi)$ here. See also \cite{FR23} for a more general statement.
	\smallskip
	
	$(2)$.  Let $A$ be a sufficiently small ample $\RR$-Cartier $\RR$-divisor on $X$. Then ${\bf B}_+(\xi)={\bf B}_{\rm num}(\xi-A)$ by Proposition \ref{prop:bminus}.
	 Write $L=A+F$ for an effective $\bb R$-Cartier $\bb R$-divisor $F$. For small $t>0$, the class $\pi^*A-tE$ is ample on $\overline X$. As $L_{t}=(\pi^*A-tE)+\pi^*F$, then ${\bf B}_+(\xi_{t})\subseteq\pi^{-1}{\rm Supp}(F)$ and we get the direct inclusion. Choosing $F$ such that $x\notin \textup{Supp}(F)$, implies also that ${\bf B}_+(\xi_{t})$ and $\pi^{-1}{\bf B}_+(\xi)$ are disjoint from $E$. For the reverse inclusion use ${\bf B}_+(\xi_0)\subseteq{\bf B}_+(\xi_t)\cup E$ for all $t\geq 0$, use $(1)$, and again that $\pi^{-1}{\bf B}_+(\xi)\cap E=\varnothing$.
	\smallskip
	
	$(3)$. Let $0<\epsilon\ll t<t'$. From Remark \ref{rem:convexcombination} and $(2)$ we obtain ${\bf B}_+(\xi_t)\subseteq{\bf B}_+(\xi_{t'})\cup(E\cap{\bf B}_-(\xi_{\epsilon}))={\bf B}_+(\xi_{t'})$. Finally, statement $(1)$ from Proposition \ref{prop:bplus} implies that ${\bf B}_+(\xi_{t'})\subseteq {\bf B}_+(\xi_{t})$ for all $t'>t$ close to $t$. This concludes the second part of the statement.
\end{proof}

\begin{lemma}[${\bf B}_-$ locus properties]\label{lem:B-ray}Let $X$ be a projective variety and $x\in X$ a smooth point. For $\xi\in{\rm Big}_x(X)$:
\begin{enumerate}
				\item ${\bf B}_-(\xi_0)=\pi^{-1}{\bf B}_-(\xi)$ is disjoint from $E$. 
	\smallskip
	
	\item The diminished base loci $({\bf B}_-(\xi_t))_{t\geq 0}$ form an increasing family of subsets of $\overline X$.
	\smallskip

	\item If $t\in\bb R_{>0}$ and $\xi$ represents a Cartier divisor $L$, then
	${\bf B}(\xi_{t'})\subseteq{\bf B}_+(\xi_t)\subseteq{\bf B}(\xi_{t'})\cup\pi^{-1}{\bf B}_+(\xi)$ for all rational $t'>t$ sufficiently close to $t$. Similar statements hold true for any $\RR$-divisor class $\xi\in N^1(X)_{\RR}$ with $t'\in \RR$, when replacing ${\bf B}$ with ${\bf B}_{\rm num}$ or ${\bf B}_{-}$.
			\smallskip
			
	\item If $t>0$, then ${\bf B}_-(\xi_t)=\bigcup_{s>0}{\bf B}_-(\xi_t+s\xi_0)=\bigcup_{0\leq q<t}{\bf B}_-(\xi_q)$.
\end{enumerate} 	
\end{lemma}

\begin{proof}Let $L$ be an $\bb R$-Cartier $\bb R$-divisor representing $\xi$.

	$(1)$. The statement is classical when $X$ is smooth, $L$ is a $\bb Q$-divisor, and $\pi$ is just a birational morphism: it follows from the valuative description of ${\bf B}_-(L)$ in \cite[Proposition 2.8]{ELMNP06}. See also \cite[Lemma 2.2]{DiL21}. For the case of $\bb R$-divisor classes, see \cite[Proposition 2.30]{FR23}. Since our hypotheses are different, we provide a proof when $X$ is just projective and $\pi$ is the blow-up of a smooth point.

	First, $\pi^{-1}{\bf B}_-(\xi)=\bigcup_A{\bf B}(\pi^*(L+A))$, where the union is over all  ample $\bb R$-Cartier $\bb R$-divisors $A$ such that $L+A$ is a $\bb Q$-Cartier $\bb Q$-divisor. Since $\pi$ is a fiber space, then ${\bf B}(\pi^*(L+A))=\pi^{-1}{\bf B}(L+A)$ and it is disjoint from $E$, as $x\not\in{\bf B}_+(\xi)$. Taking $t_A>0$ rational with $\pi^*A-t_AE$ ample, we have the inclusions
\[
{\bf B}(\pi^*(L+A))\subseteq{\bf B}(\pi^*L+(\pi^*A-t_AE))\subseteq{\bf B}_-(\pi^*L),
\]	
as the left side is disjoint from $E$. In particular, we get $\pi^{-1}{\bf B}_-(\xi)\subseteq{\bf B}_-(\xi_0)$. 

For the reverse inclusion, let $H$ be an ample $\bb R$-Cartier $\bb R$-divisor on $\overline X$ and $A$ an ample $\bb R$-Cartier $\bb R$-divisor on $X$ such that $H-\pi^*A$ is ample and $L+A$ is a $\bb Q$-Cartier $\bb Q$-divisor. Then Proposition \ref{prop:bminus}.(3) yields
\[
{\bf B}_+(\pi^*L+H)\subseteq{\bf B}_+(H-\pi^*A)\cup{\bf B}_-(\pi^*(L+A))\subseteq{\bf B}(\pi^*(L+A))\subseteq\pi^{-1}{\bf B}_-(L) .
\]
The union over all $H$ gives the result.
	\smallskip
	
		$(2)$. Follows from Lemma \ref{lem:B+interpretation}.$(2)$ and Remark \ref{rem:convexcombination}.
	\smallskip
	
	$(3)$. We show this statement solely for the stable base locus as the other cases are similar. Note that the rationality of $t'$ is only needed for the definition of ${\bf B}(L_{t'})$. So, first we have ${\bf B}(L_{t'})\subseteq{\bf B}_+(L_{t'})={\bf B}_+(L_t)$ when $t'>t$ is sufficiently close to $t$ by Lemma \ref{lem:B+interpretation}.(3). Fixing $0<\epsilon<t$, then $L_t$ is a convex combination of $L_{t'}$ and $L_{\epsilon}$. So, by Proposition \ref{prop:bminus}.(3) we have the inclusions ${\bf B}_+(L_t)\subseteq{\bf B}_+(L_{\epsilon})\cup{\bf B}_-(L_{t'})\subseteq {\bf B}_+(L_{\epsilon})\cup{\bf B}(L_{t'})$. We finish by taking $\epsilon$ small as in Lemma \ref{lem:B+interpretation}.(2).
	\smallskip

$(4)$. The second equality is clear by the homogeneity of ${\bf B}_-$. For the first equality, the $\supseteq$ inclusion follows from $(2)$ and the second equality.
For the reverse inclusion, let $A$ be ample on $X$. Then 
\[
{\bf B}_-(L_t)=\bigcup\nolimits_{m\geq 1}{\bf B}_-(L_t+\frac1m\pi^*A)
\]
by \cite[Remark 2.24.(c)]{FR23}, i.e., where $\pi^*A$ is only nef. It is sufficient to show that each term in the union is contained in some ${\bf B}_-(L_t+sL_0)$. Fixing an $m$ we can write $L_t+\frac lm\pi^*A$ as the sum of $L_t+sL_0$ and $\frac lm\pi^*A-sL_0$, where the latter is nef for sufficiently small $s>0$ since $\frac 1mA-sL$ is ample. By applying Proposition \ref{prop:bminus}.$(3)$ for the restricted locus, we get the final inclusion.
\end{proof}

We introduce the first important actors in this section, a set of invariants that measure a change in base loci when moving on the ray $(\xi_t)_{t\geq 0}$. 

\begin{definition}[Infinitesimal successive minima]\label{def:infinitesimalsuccessiveminima}
Let $X$ be a projective variety of dimension $n$. Let $x\in X$ be a smooth point and $\xi\in{\rm Big}_x(X)$ (see \eqref{def:bigxcone}).
	For $1\leq i\leq n$, let 
	\[\epsilon_i(\xi;x)\deq\min\{t>0\ \mid\ \dim {\bf B}_+(\xi_t)\cap E\geq i-1\}=\sup\{t>0\ \mid\ \dim{\bf B}_+(\xi_t)\cap E\leq i-2\}\ .\]
\end{definition}
\noindent We explain in Section \ref{ss:motivation} how this construction is motivated by the similar local successive minima of \cite{AI} and the classical (moving) Seshadri constant and Fujita--Nakayama invariant (width). Later we provide another motivation through convex geometry.
For now we use the variation of augmented base loci on the exceptional ray discussed above to study our new invariants.
\begin{lemma}[Basic properties of infinitesimal successive minima]\label{lem:infsuccessiveproperties}
With notation as above, then
\begin{enumerate}
\item The infinitesimal successive minima are an increasing set of homogeneous numerical invariants: \newline $0<\epsilon_1(\xi;x)\leq \epsilon_2(\xi;x)\leq\ldots\leq\epsilon_{n}(\xi;x)$.
\smallskip

 \item Each $\epsilon_i(\cdot ;x):\textup{Big}_x(X)\rightarrow \RR_{>0}$ defines a continuous and superadditive function. 
\smallskip

 \item For a birational morphism  $\rho:Y\to X$ of normal projective varieties that is an isomorphism over the smooth point $x\in X$, one has $\epsilon_i(\rho^*\xi;\rho^{-1}\{x\})=\epsilon_i(\xi;x)$ for all $i=1,\ldots ,n$.
 \smallskip
 
 \item The families $({\bf B}_{\rm num}(\xi_t)\cap E)_{t\geq 0}$ and $({\bf B}_-(\xi_t)\cap E)_{t\geq 0}$ also determine $\epsilon_i(\xi;x)$. When $L$ is a $\bb Q$-Cartier $\bb Q$-divisor, then the family $({\bf B}(L_t))_{t\in\bb Q_+}$ also determines $\epsilon_i(L;x)$.
 \end{enumerate}
\end{lemma}

\begin{proof}
$(1)$. The numerical invariance of our infinitesimal successive minima is due to the same property being satisfied by the augmented base loci. Their homogeneity: $\epsilon_i(m\xi;x)=m\cdotp\epsilon_i(\xi;x)$ for all $m\geq 1$ is due to the equalities ${\bf B}_+(\xi_t)={\bf B}_+(m(\xi_t))={\bf B}_+((m\xi)_{mt})$. Their monotonicity follows from the definition and the positivity of $\epsilon_1(\xi;x)$ can be found in Lemma \ref{lem:B+interpretation}.(2).

$(2)$. For the superadditivity choose as $s,t\in\RR$ such that $\dim {\bf B}_+(\xi_t)\cap E\leq i-2$ and $\dim {\bf B}_+(\zeta_s)\cap E\leq i-2$. In particular, by definition we know that $t<\epsilon_i(\xi;x)$ and $s<\epsilon_i(\zeta;x)$. Thus, by Proposition \ref{prop:bminus}.(2) we get the following containment of base loci:
\[
{\bf B}_+((\xi+\zeta)_{t+s})={\bf B}_+(\xi_t+\zeta_s)\subseteq{\bf B}_+(\xi_t)\cup{\bf B}_+(\zeta_s) \ ,
\]
implying that $t+s<\epsilon_i(\xi+\zeta;x)$. A limiting process finally confirms that $\epsilon_i(\xi;x)+\epsilon_i(\zeta;x)\leq\epsilon_i(\xi+\zeta;x)$.

Each function $\epsilon_i(\cdot;x)$ is homogeneous and superadditive hence concave and continuous on the open convex set ${\rm Big}_x(X)$.

$(3)$. Consider $\rho:Y\to X$ a birational morphism as in the statement. By \cite[Proposition 2.3]{BBP13}, we have that $y\deq\rho^{-1}\{x\}\not\in{\bf B}_+(\rho^*\xi)$.	The base change $\overline{\pi}$ of $\pi$ by $\rho$ is the blow-up of $y$.  \cite{BBP13} again for the induced $\overline{\rho}:\overline Y\to\overline X$ gives that $\overline{\rho}^{-1}{\bf B}_+(\xi_t)$ and ${\bf B}_+({\overline{\rho}}^*\xi_t)$ agree around the exceptional divisor over $y$. Consequently, we get $\epsilon_i(\rho^*\xi;\rho^{-1}\{x\})=\epsilon_i(\xi;x)$ for all $i$.

$(4)$ is a consequence of Lemma \ref{lem:B-ray}.$(3)$.
\end{proof}

The infinitesimal successive minima of a big divisor class are determined by the positive part in the Nakayama $\sigma$-Zariski decomposition \cite{Nak04}.

\begin{lemma}[Zariski decompositions]\label{lem:zariskidecompositionsuccessiveminima}
If $X$ is smooth, $x\in X$ and $\xi\in{\rm Big}_x(X)$, then $\epsilon_i(\xi;x)=\epsilon_i(P_{\sigma}(\xi);x)$, where $P_{\sigma}(\xi)$ is the positive part of $\xi$.
\end{lemma}
\begin{proof}
	
%The negative part of $\xi$ is supported in ${\bf B}_-(\xi)$, in particular it does not contain $x$ in its support. Then ${\bf B}_+(\xi_t)\subseteq{\bf B}_+(P_{\sigma}(\xi)_t)\cup{\rm Supp}(\pi^*N_{\sigma}(\xi))$ shows that $\epsilon_i(\xi;x)\geq\epsilon_i(P_{\sigma}(\xi);x)$.

If $\xi=A+F$ with $A$ an ample $\bb R$-Cartier $\bb R$-divisor class and $F$ effective with $x\notin \textup{Supp}(F)$, then $F- N_{\sigma}(\xi)$ is effective. Thus $x\notin \textup{Supp}(N_{\sigma}(\xi))$ and $P_{\sigma}(\xi)=A+F'$ for $F'=F-N_{\sigma}(\xi)$ effective, with $x$ not in its support. Hence $P_{\sigma}(\xi)\in {\rm Big}_x(X)$. As both the positive part and the infinitesimal successive minima vary continuously and are homogeneous, we may assume $\xi$ is the class of a Cartier divisor $L$. 
Let $N\leq N_{\sigma}(L)$ be a rational approximation and set $P\deq L-N$. We have that $L\geq P\geq P_{\sigma}(\xi)$ is also a rational approximation of the positive part, and $P\in{\rm Big}_x(X)$. We leave the rest of the proof as an exercise. Morally it is because $L$ and $P$ have the same graded section rings and $x\not\in{\rm Supp}(N)$.
\end{proof}

%By Lemma \ref{lem:infsuccessiveproperties}.$(2),(4)$ it is sufficient to prove that ${\bf B}(L_s)\cap E={\bf B}(P_s)\cap E$ for rational $s>0$. Let $m$ be sufficiently divisible. Then ${\rm Bs}(m\pi^*L-msE)\subseteq{\rm Bs}(m\pi^*P-msE)\cup{\rm Supp}(\pi^*mN)$. When intersecting with $E$, this gives  
%${\bf B}(L_s)\cap E\subseteq{\bf B}(P_s)\cap E$. 
%Conversely, if $E\subset{\bf B}(L_s)$, then the previous inclusion is an equality (in fact we prove in Corollary \ref{cor:EnotinB+} that $E$ is not even contained in ${\bf B}_+(L_s)$). Assume $y\in E$ and $m$ is sufficiently divisible such that there exists $D'\in|mL_s|$ not passing through $y$. 
%In particular $D'$ does not contain $E$ in its support. Then $D'$ is the strict transform of a divisor $D\in|mL|$. From the definition of $N_{\sigma}(L)$ and since $N\leq N_{\sigma}(L)$, we can write $D=G+mN$ where $G$ is an effective divisor in $|mP|$. Using that $x\not\in{\rm Supp}(N)$, we have $D'=\overline D=\overline G+m\overline N=\overline G+m\pi^*N$, thus $\overline G\in|mP_s|$ does not pass through $y$, hence $y\not\in{\bf B}(P_s)$.
	
The final property that we want to show is that infinitesimal successive minima are constant at very general base points in our ambient space.
\begin{lemma}\label{lem:tauverygeneralconstant}
	Let $X$ be a projective variety of dimension $n$ over an uncountable algebraically closed field.
	Let $\xi\in N^1(X)_{\bb R}$ be a big class. Then for every $1\leq i\leq n$ the function $x\mapsto\epsilon_i(\xi;x)$ is constant outside the union of at most countably many closed subsets of $X$.
\end{lemma}
\begin{proof}
	We may assume that $x$ is smooth and not contained in ${\bf B}_+(\xi)$.
	Since $\xi$ is the limit of a countable sequence of $\bb Q$-divisor classes, and $\epsilon_i(\cdot;x)$ are homogeneous and continuous on ${\rm Big}_x(X)$, it suffices to show the statement for the case when $\xi$ is the class of a big Cartier divisor $L$.
	
	For fixed $x$, by Lemma \ref{lem:infsuccessiveproperties}.(4), the $\epsilon_i(L;x)$ are determined by the jumps of the nondecreasing function 
	\[\bb Q_+\to\{-1,0,\ldots,n-1\}:\ s  \ \mapsto\ \dim\big({\bf B}(L_s)\cap E\big).
	\] 
	In particular, it is sufficient to prove that for every rational number $s>0$ the dimension of ${\bf B}(L_s)\cap E$ is constant for very general $x$. This is realized by a natural gluing construction. 
	
	Let $\widetilde{X\times X}={\rm Bl}_{\Delta}X\times X$ be the blow-up of the diagonal $\Delta\subseteq X\times X$ with exceptional divisor $\cal E$. Let $p_1,p_2:\widetilde{X\times X}\to X$ be the two induced projections. The fiber $p_1^{-1}\{x\}$ identifies with ${\rm Bl}_xX$.
	Let $\cal B_{pq}\subseteq \widetilde{X\times X}$ be the locus where surjectivity fails for the relative evaluation map 
	\[
	p_1^*p_{1*}(p_2^*L^{\otimes q}(-p\cal E))\ \longrightarrow \ p_2^*L^{\otimes q}(-p\cal E) \ .
	\]
By Semicontinuity and Grauert's theorem, the restriction of this map to a general fiber of $p_1$ is exactly the evaluation map of the line bundle $q\pi^*L-pE$ on the blow-up of $X$ at the point. In particular, the general fiber of ${p_1}|_{\cal B_{pq}}$ is ${\rm Bs}(q\pi^*L-pE)$, and the general fiber of ${p_1}|_{\cal B_{pq}\cap\cal E}$ is ${\rm Bs}(q\pi^*L-pE)\cap E$.
	
	Fix a rational $s>0$. The general fiber of $p_1$ on the closed subset $\bigcap_{\frac pq>s}\cal B_{pq}\cap\cal E$ has constant dimension. The very general one is $\bigcap_{\frac pq>s}{\rm Bs}(q\pi^*L-pE)\cap E$, which is ${\bf B}(L_s)\cap E$ by Lemma \ref{lem:Bray}.
\end{proof}
Based on this statement we make the following definition:
\begin{definition}
	For a big divisor class $\xi\in N^1(X)_{\bb R}$, denote $\epsilon_i(\xi)\deq\epsilon_i(\xi;x)$ for very general $x\in X$.
\end{definition}
\begin{remark}[Positivity]
\noindent By Proposition \ref{prop:bplus}.(1) there are countably many proper closed subsets of $X$ that appear as the augmented base locus of a big real class on $X$. Lemma \ref{lem:infsuccessiveproperties} implies that for very general $x\in X$ we have  $\epsilon_i(\xi;x)>0$ for all $1\leq i\leq n$ and any big class $\xi\in N^1(X)_{\RR}$.
\end{remark}

\subsection{Asymptotic partial jet separation}
\cite{ELMNP09} show that the moving Seshadri constant measures asymptotic jet separation at the base point. \cite{AI} prove that the Fujita--Nakayama invariant is an asymptotic measure of "jet nontriviality". We see these two conditions on jets as the end points of an increasing sequence of invariants, measuring asymptotically a separation of partial jets.

\begin{definition}[$i$-th partial jet separation order]\label{def:partialjets}
	Let $X$ be a projective variety of dimension $n$. Let $L$ be a line bundle on $X$, let $x\in X$ be a smooth point. For $i=1,\ldots ,n$ denote
	\[S_i(L;x)\deq\max\{s\geq -1\ \mid\ H^0(\overline X,L_s)\to H^0(\bb P^{n-i},\cal O(s))\text{ is onto for some (general) linear }\bb P^{n-i}\subset E\}\ .\]
	 Put in algebraic terms, if $x_1,\ldots,x_n$ are local coordinates around $x$ and $z_1,\ldots,z_{i-1}$ are general linear forms in the $x_i$, then $H^0(X,L\otimes\frak m_x^s)\to\frak m_x^s/((z_1,\ldots,z_{i-1})\frak m_x^{s-1}+\frak m_x^{s+1})$ is onto.
\end{definition}

\begin{remark}[Properties of jet separation orders]\label{rmk:propsSc}$ $

	\noindent {\bf $(1)$. Base loci.} When $s=0$, then $H^0(\overline X,L_0)\to H^0(\bb P^{n-i},\cal O)$ is the evaluation at $x$, which is surjective if and only if $x\not\in{\rm Bs}(L)$. Thus $S_i(L;x)=-1$ if and only if $x\in{\rm Bs}(L)$.
	\smallskip
	
\noindent {\bf $(2)$. Genericity.} If the map $H^0(\overline X,L_s)\to H^0(\bb P^{n-i},\cal O(s))$ is onto for some linear subspace $\bb P^{n-i}\subset\bb P^{n-1}$, then the same is true of the general linear subspace of the same dimension. This is a standard argument on the Grassmann variety of $n-i$-dimensional linear subspaces of $\bb P^{n-1}$.
%\begin{proof}Let $\bb G$ be the Grassmann variety parameterizing linear subspaces $\bb P^{n-c-1}\subset\bb P^{n-1}=E$. Consider the composition $\varphi:H^0(\overline X,L_s)\otimes\cal O_{\bb G}\to \Sym^sH^0(E,\cal O(1))\otimes\cal O_{\bb G}\to\Sym^sQ$, where $Q$ is the universal rank $n-c$ quotient bundle on $\bb G$. Over $[\bb P^{n-c-1}]\in \bb G$, the fiber of the morphism is the restriction $H^0(\overline X,L_s)\to H^0(\bb P^{n-c-1},\cal O(s))$. If $\varphi$ is surjective on one fiber, it is surjective on the general fiber.
%\end{proof} 
\smallskip

\noindent{\bf $(3)$. Birational invariance.} If $X$ is normal projective and $\rho:Y\to X$ is a projective birational morphism that is an isomorphism over the smooth point $x\not\in{\bf B}_+(L)$, then $S_i(\rho^*L;\rho^{-1}\{x\})=S_i(L;x)$. As $X$ is normal, $\rho^*$ preserves the evaluation maps $H^0(X,L\otimes\frak m_x^s)\to\frak m_x^s/\frak m_x^{s+1}$ under our assumptions. These identify with the restrictions $H^0(\overline X,L_s)\to H^0(E,\cal O(s))$.
\end{remark}

\begin{lemma}[Superadditivity]\label{lem:superadditivity}
If $L$ and $M$ are line bundles, and $x\not\in{\rm Bs}(L)\cup{\rm Bs}(M)$, then \[S_i(L;x)+S_i(M;x)\leq S_i(L\otimes M;x),\]
for any $i=1,\ldots ,n$.
\end{lemma}
\begin{proof}Indeed, if $H^0(\overline X,L_s)\to H^0(\bb P^{n-i},\cal O(s))$ and $H^0(\overline X,M_t)\to H^0(\bb P^{n-i},\cal O(t))$ are onto (we can choose the same $\bb P^{n-i}$ subspace for both by genericity), we have a commutative diagram
	\[\xymatrix{H^0(\overline X,L_s)\otimes H^0(\overline X,M_t)\ar[r]\ar@{->>}[d]& H^0(\overline X,L\otimes M_{s+t})\ar[d]\\ H^0(\bb P^{n-i},\cal O(s))\otimes H^0(\bb P^{n-i},\cal O(t))\ar@{->>}[r]& H^0(\bb P^{n-i},\cal O(s+t))}\] The bottom horizontal map is onto if $s,t\geq 0$.	In this case, the right vertical map is also onto. 
\end{proof} 
With this in hand we are able to introduce the asymptotic version of separation of partial jets.

\begin{definition}[Asymptotic partial jet separation order]
With the notation above define
	\[
	s_i(L;x)\deq\limsup_{m\to\infty}\frac 1mS_i(L^{\otimes m};x)=\sup_{m\geq 1}\frac 1mS_i(L^{\otimes m};x).
	\]
	for any $i=1,\ldots ,n$.
\end{definition}
\begin{remark}[Limit]\label{rem:lim}
Due to the superadditivity property in Lemma \ref{lem:superadditivity}, we note that restricting to the additive semigroup $\{m\in\bb N\ \mid\ S_i(mL;x)\geq 0\}$, the limsup becomes actually a limit by Fekete's Lemma.
\end{remark}

\begin{lemma}[Basic properties of asymptotic partial jet separation order]\label{lem:propssc}
Let $X$ be a projective variety, $L$ a line bundle on $X$ and $x\notin \Bplus(L)$ a smooth point. Then the asymptotic partial jet separation orders determine an increasing set of homogeneous numerical invariants: 
\[
0<s_1(L;x)\leq s_2(L;x)\leq\ldots\leq s_{n}(L;x).
\] 
In particular, each $s_i$ can be defined for $\QQ$-Cartier $\QQ$-divisors.
\end{lemma}

\begin{remark}[Further properties]$ $ 
We do not use the remarks described here in the sequel, but they might be of general interest. Let $L$ be a line bundle on the projective variety $X$ of dimension $n$, and let $x\in X$ be a smooth point such that $x\not\in{\bf B}_+(L)$.
\smallskip

\noindent {\bf $(1)$. Partial jet separation is eventually not sparse.}	
		Fix $0<s<s_i(L;x)$ rational. For sufficiently divisible $m$ we have that $H^0(\overline X,m\pi^*L-jE)\to H^0(\bb P^{n-i},\cal O(j))$ is onto for all $0\leq j\leq ms$ and general $\bb P^{n-i}\subset E$.

The proof, left as an exercise, is a standard argument that mimics the proof of \cite[Proposition 2.1]{LM09}, using \cite[Proposition 3.3]{Kho92}.
%The motivation of the result comes from the connection between classical jet separation and $1$-partial jet separation that we will observe later. 
\smallskip

	\noindent{\bf $(2)$. Restrictions to subvarieties.}
	Let $Z\subseteq X$ be a subvariety of dimension $d$ that also contains $x$ as a smooth point, then
	\begin{enumerate}[(i)]
		\item If $d\geq i+1$, then $s_i(L;x)\leq s_{i}(L|_Z;x)$.
		\smallskip
		
		\item If $T_xZ\subset T_xX$ is general, then $s_i(L;x)\leq s_{\max\{i+d-n,1\}}(L|_Z;x)$. 
	\end{enumerate}
\noindent The statements extends formally by continuity to $\bb R$-classes in ${\rm Big}_x(X)$.
	
For $(i)$ let $H^0(\overline X,L_s)\to H^0(\bb P^{n-i},\cal O(s))$ be surjective for some (general) $\bb P^{n-i}\subset E$. Then $\bb P^{d-i}\simeq {\bf P}_{\rm sub}(T_xZ)\cap\bb P^{n-i}\subset E|_{\overline Z}$ has codimension $i-1$ and $H^0(\overline Z,(L|_Z)_s)\to H^0(\bb P^{d-i},\cal O(s))$ is still surjective. Up to scaling $L$ and taking limits, this gives the claim.
		
	For $(ii)$ let $d-1\geq n-i$. Then a general $\bb P^{d-1}\simeq{\bf P}_{\rm sub}(T_xZ)$ in $E$ contains a general $\bb P^{n-i}$ with codimension $i+d-n-1$. But whenever  $d-1\leq n-i$, ${\bf P}_{\rm sub}(T_xZ)$ is contained in a general $\bb P^{n-i}$. Argue as in $(i)$. 
	\smallskip

	\noindent{\bf $(3)$. Interpretation via complete linear series $|mL|$.} Let $X$ be normal and $m$ be sufficiently divisible so that ${\bf B}(L)={\rm Bs}(mL)$. Let $\frak b_m$ be the base ideal of $|mL|$ and $\beta_m:X_m\to X$ be a resolution of ${\rm Bl}_{\frak b_m}X$ that is an isomorphism away from $x$. Let $x_m\in X_m$ be the preimage of $x$. Then $\beta_m^*mL=M_m+F_m$ where $M_m$ is globally generated on $X_m$, $F_m$ is effective, without $x_m$ in its support, and $H^0(X_m,M_m)=H^0(X,mL)$. We have
	\[s_i(L;x)=\lim_{m\to\infty}\frac 1ms_i(M_m;x_m).\]
	The limit is taken only over $m$ sufficiently divisible as above so that $x\not\in\frak b_m$. The argument is similar to the ideas around \cite[Proposition 6.4]{ELMNP09} and completing the proof is left as an exercise.
	
	%\begin{proof} The argument is similar to the ideas around \cite[Proposition 6.4]{ELMNP09}. Since $F_m$ does not contain $x_m$ in its support, it is globally generated at $x_m$, so $S_c(F_m;x_m)\geq 0$. Then by superadditivity, $s_c(\beta_m^*mL;x_m)\geq s_c(M_m;x_m)$. By birational invariance and homogeneity, $s_c(L;x)\geq\frac 1ms_c(M_m;x_m)$ for all $m$.
		
	%	Conversely, let $\epsilon>0$ and let $m$ be large enough so that $\frac 1mS_c(mL;x)\geq s_c(L;x)-\epsilon$. When $m$ is also sufficiently divisible, then $\frac 1ms_c(M_m;x_m)\geq\frac 1mS_c(M_m;x_m)=\frac 1mS_c(mL;x)\geq s_c(L;x)-\epsilon$.
		
	%	Finally, we explain why the sequence $\frac 1ms_c(M_m;x_m)$ converges when $m$ is divisible enough. For $m_1,m_2$ sufficiently divisible, we can pass to higher models that are still isomorphisms over $x$, and assume $\beta_{m_1+m_2}=\beta_{m_1}=\beta_{m_2}$. Then since the $F_{m_i}$ and $F_{m_1+m_2}$ are fixed parts of the corresponding $\beta_{m_1}^*|mL|$, we have $F_{m_1}+F_{m_2}=F_{m_1+m_2}+G_{m_1,m_2}$ where $G_{m_1,m_2}$ is effective, and still without $x_{m_1}$ in its support. Thus $M_{m_1+m_2}=M_{m_1}+M_{m_2}+G_{m_1,m_2}$, so $s_c(M_{m_1};x_{m_1})+s_c(M_{m_2};x_{m_2})\leq s_c(M_{m_1+m_2};x_{m_1+m_2})$. The existence of the limit is now a consequence of Fekete's Lemma.
	%\end{proof} 
\end{remark}

\begin{proof}[Proof of Lemma \ref{lem:propssc}]
Homogeneity follows from Remark \ref{rem:lim}. As a consequence we can define $s_i(L;x)$ even when $L$ is a $\bb Q$-Cartier $\bb Q$-divisor.

Our sequence of $s_i$ is increasing  because  $S_1(L;x)\leq \ldots\leq S_{n}(L;x)$, where the latter is due to the fact that when the map $H^0(\overline X,L_s)\to H^0(\bb P^{n-i},\cal O(s))$ is surjective, then the map $H^0(\overline X,L_s)\to H^0(\bb P^{n-i-1},\cal O(s))$ is also surjective for any  codimension one linear subspace $\PP^{n-i-1}\subseteq \PP^{n-i}$. On the other hand, to show $s_1(L;x)>0$ we use the fact that $x\notin \Bplus(L)$ to say that $x\notin \Bstable(L-A)$ for some small ample $A$. So, modulo normalization we are able to get sections of $L$, coming from from $L-A$ that do not pass through $x$ and from $A$. In particular, this line of thought implies $s_1(L;x)\geq s_1(A;x)>0$, where the latter inequality follows from Remark \ref{rem:lim}. Similarly we get $S_i(mL;x)\geq S_1(mL;x)> 0$ for all $m$ sufficiently large and thus the supremum in the definition of each $s_i$ is actually a limit under our assumptions.

Lastly, we show the numerical nature of our invariants in two steps. Fixing $A$ an ample Cartier divisor, we show first that $\lim_{t\to 0}s_i(L+tA;x)=s_i(L;x)$, where $t$ is rational, either positive or negative. For this consider the open cone in $\bb Q^2$ of all $\{(a,b)\in\bb Q_{>0}\times\bb Q\ \mid\ x\not\in{\bf B}_+(aL+bA)\}$. The function $(a,b)\to s_i(aL+bA;x)$ is superadditive and homogeneous, so concave on our cone. This implies the partial continuity property.

For the second step let $N$ be a numerically trivial Cartier divisor. We still have $x\not\in{\bf B}_+(L+N)$ since the augmented base locus is a numerical invariant. Consider an ample divisor $A$. Then $A+mN$ remains ample for all $m\in\bb Z$. 	By the continuity step, $\lim_{m\to\infty}s_i(L+N+\frac 1mA;x)=s_i(L+N;x)$. We also have 
	\[
	s_i(L+N+\frac 1mA;x)\geq s_i(L;x)+s_i(N+\frac 1mA;x)> s_i(L;x) \ ,
	\]
	as $mN+A$ is ample and $x\not\in{\bf B}_+(L)$. Thus $s_i(L+N;x)\geq s_i(L;x)$. This holds for any $L$ with $x\notin \Bplus(L)$ and any numerically trivial $N$. Applying it for $L'=L+N$ and $N'=-N$ implies the opposite inequality.
\end{proof}

\subsection{Infinitesimal successive minima versus partial jet separation}
Demailly realized that the Seshadri constant of an ample line bundle measures both the appearance of base loci on the blow-up along the exceptional ray and quantifies asymptotic jet separation (see \cite[Chapter 5]{Laz04} for more details). Our goal here is to extend these two different perspectives on the Seshadri constant in all codimensions for any big class and prove Theorem \ref{thm:mintrosuccessjets}.
\begin{theorem}\label{thm:successivejets}Let $X$ be a projective variety of dimension $n$, $L$ a $\QQ$-Cartier $\QQ$-divisor and $x\notin \Bplus(L)$ a smooth point. Then $s_i(L;x)\leq\epsilon_i(L;x)$ for all $1\leq i\leq n$. If furthermore $X$ is normal, then equality holds
	\[s_i(L;x)= \epsilon_{i}(L;x)\ .\]
	for all $i=1,\ldots ,n$.
\end{theorem}
\begin{proof} 
	Since $s_i(\cdot;x)$ and $\epsilon_i(\cdot;x)$ are homogeneous and numerical invariants, we may assume that $\xi=L$ is Cartier. 
	For the $\leq$ inequality we want to show that if $s\geq 0$ is an integer and $m\geq 1$ is such that 
	\[
	H^0(\overline X,(mL)_s) \ \longrightarrow \ H^0(\bb P^{n-i},\cal O(s))
	\] 
	is onto for some general $\bb P^{n-i}\subset E$, and if $t>0$ is real such that $\dim {\bf B}_+(L_t)\cap E\geq i-1$, then $t\geq \frac sm$.
	By Lemma \ref{lem:B-ray}.(3), we may assume that $t$ is rational and ${\bf B}_+(L_t)\cap E={\bf B}(L_t)\cap E$, and $\dim{\bf B}(L_t)\cap E\geq i-1$. The surjectivity condition implies that $(mL)_s$ is globally generated along $\bb P^{n-i}$, so this linear subspace does not meet ${\rm Bs}((mL)_s)\cap E$, so $\dim{\rm Bs}((mL)_s)\cap E\leq i-2$. If $t<\frac sm$, then 
	\[
	{\bf B}(L_t)\cap E\subseteq {\bf B}(L_{\frac sm})\cap E\subseteq{\rm Bs}((mL)_s)\cap E
	\]
	 by Lemma \ref{lem:Bray}. This leads to a contradiction for dimension reasons.
	
For the $\geq$ inequality when $X$ is normal, assume $t>0$ and $\dim{\bf B}_+(L_t)\cap E\leq i-2$, so that $\epsilon_{i}(L;x)> t$. We want to prove that $s_i(L;x)\geq t$. 
A general $\bb P^{n-i}\subset E$ does not meet ${\bf B}_+(L_t)\cap E$ for dimension reasons. We may replace $t$ with a small rational $\epsilon_i(L;x)>t'>t$. We want to prove that $H^0(\overline X,m(L_t))\to H^0(\bb P^{n-i},\cal O(mt))$ is onto for $m$ sufficiently large and divisible, i.e., $S_i(mL;x)\geq mt$, so $s_i(L;x)\geq t$. For this apply Proposition \ref{prop:vanishing} below for some Cartier integer multiple of $L_t$.
	\end{proof}

	Inspired by \cite{BCL14}, the following generation statement closes the proof of Theorem \ref{thm:successivejets}.

\begin{proposition}\label{prop:vanishing}Let $X$ be a normal projective variety over an algebraically closed field. Let $L$ be a Cartier divisor on $X$.
	Suppose that $V\subset X$ is a closed subscheme whose support is disjoint from ${\bf B}_+(L)$. Then the natural restriction morphism
	\[H^0(X,mL)\to H^0(V,mL|_V)\]
	is surjective for all $m$ sufficiently large and divisible.
\end{proposition} 
%\noindent We do not know of an example of Proposition \ref{prop:vanishing} failing in the non-normal case.
\begin{proof}
	Let $H$ be a very ample divisor on $X$. There exists $p>0$ such that ${\bf B}_+(L)={\bf B}(pL-H)$. Choose $m$ sufficiently divisible so that ${\bf B}(pL-H)={\rm Bs}(mpL-mH)$. Then in particular $|mpL-mH|$ is basepoint free around the support of $V_{\rm red}$ and $|mpL|$ is thus an embedding around the support of $V$.
	
	Let $\pi_m:X_m\to X$ be the normalization of the blow-up of the base ideal of $|mpL|$. There exists an effective Cartier divisor $E_m$ on $X_m$ such that $\pi_m({\rm Supp}(E_m))={\rm Bs}(mpL)$, the divisor $M_m\deq\pi_m^*mpL-E_m$ is globally generated and $H^0(X_m,M_m)\to H^0(X_m,\pi^*_mmpL)$ is an isomorphism. Since $X$ is normal, the latter is isomorphic to $H^0(X,mpL)$ via $\pi_m^*$. Note that the support of $V$ is disjoint from $\pi_m({\rm Supp}(E_m))$ and thus $\pi_m$ is an isomorphism over $V$. Denoting by $V_m\subset X_m$ the preimage, then $\pi_m|_{V_m}^V:V_m\to V$ is an isomorphism.
	
	Let $g_m:X_m\to{\bf P}_m$ be the morphism to a projective space induced by the basepoint free linear system $|M_m|$ so that $M_m=g_m^*\cal O_{{\bf P}_m}(1)$. Denote $W_m\deq g_m(V_m)$ with the (closed sub-) scheme theoretic image structure. Observe that $|M_m|=\pi^*|mpL|-E_m$ still determines an embedding of $V_m$, in particular $g_m|_{V_m}^{W_m}:V_m\to W_m$ is an isomorphism.
	By Serre vanishing, there exists $q_0>0$ such that 
	\[\Sym^qH^0({\bf P}_m,\cal O(1))\simeq H^0({\bf P}_m,\cal O(q))\to H^0(W_m,\cal O_{W_m}(q))\] is surjective for all $q\geq q_0$.
Then on $X_m$ we also have that
$\Sym^qH^0(X_m,M_m)\to H^0(V_m,qM_m|_{V_m})$
is surjective, hence so is the restriction map
\[H^0(X_m,qM_m)\to H^0(V_m,qM_m|_{V_m})\]	
that it factors through. The commutative diagram
\[\xymatrix{
H^0(X_m,qM_m)\ar@{^{(}->}[r]\ar@{->>}[d]\ar@{}[dr]|{\large{\circlearrowleft}}& \pi_m^*H^0(X,qmpL)\ar[d]\\
H^0(V_m,qM_m|_{V_m})\ar[r]^-{\simeq}& {(\pi_m|_{V_m}^V)}^*H^0(V,qmpL|_V)
}\]
implies that the right vertical map is also surjective, concluding the proof.	
\end{proof}
Since the successive minima behave continuously on the open cone ${\rm Big}_x(X)$, then 
as a consequence of Theorem \ref{thm:successivejets} we get that the asymptotic partial jets order can be extended uniquely to a continuous function on the same cone. In particular this provides a proof of Theorem \ref{thm:mintrosuccessjets}. 

\begin{corollary}[Continuity for asymptotic partial jets order]\label{cor:normaljetloci}
Let $X$ be a normal projective variety and $x\in X$ a smooth point. Then each function $\xi\mapsto s_i(\xi;x)$ on $\textup{Big}_x(X)_{\QQ}$ (see \eqref{def:bigxcone}) extends uniquely to a continuous superadditive function $s_i(\cdot ;x):\textup{Big}_x(X)\rightarrow \RR_{>0}$.
 \end{corollary} 
 
Even without the normality assumption we can obtain the continuity of each $s_i(\cdot ;x)$ for a smooth point $x\in X$ from Lemma \ref{lem:propssc} and from the following statement applied to the open convex cone ${\rm Big}_x(X)\cap\bb Q$ and for the homogeneous superadditive function $s_i(\cdot;x)$ on it.
\smallskip
	
\noindent "\textit{Let $V$ be a finite dimensional vector space over $\bb Q$. Let $G\subset V$ be an open convex cone. Let $d>0$ be a real number, and let $f:G\to\bb R_+$ be a $d$-homogeneous function, such that $f(g+g')\geq f(g)$ for all $g,g'\in G$. Then $f$ is locally uniformly continuous around every $\xi$ in $G$. In particular $f$ admits a unique continuous extension to the cone generated by $G\otimes 1$ in $V\otimes_{\bb Q}{\bb R}$.}"
\smallskip

\noindent The topology on $V$ is induced from the Euclidean topology on $V\otimes_{\bb Q}\bb R$. 
$G$ being a convex cone means that it is stable under addition, and under scaling by $\bb Q_{>0}$. Openness is with respect to the topology above. The function $f$ being $d$-homogeneous means $f(\lambda g)=\lambda^df(g)$ for all $g\in G$ and $\lambda\in\bb Q_{>0}$.
Essentially the same result is also stated without proof in \cite[Lemma 2.7]{Leh16}. One can give a proof inspired by arguments in the proof of the continuity of volumes of big divisors in \cite[Theorem 2.2.44]{Laz04}, and of restricted volumes and moving Seshadri constants in \cite{ELMNP09}.
%\begin{proof}
%Note that $G$ is full dimensional by openness. We may also assume that it is strict, or else the monotonicity property forces $f$ to be a constant function.
%Let $g_1,\ldots,g_n$ be a basis of $V$ with $g_i\in G$. 
%Identify $v\in V$ with $v\otimes 1\in V_{\bb R}$. On $V$ and $V_{\bb R}$ we have a norm given by $||\sum_ia_ig_i||=\max\{|a_i|\ \mid\ 1\leq i\leq n\}$.

%Put $\gamma\deq\sum_{i=1}^ng_i\in G$. Fix $\xi\in G$. For $s>0$, set $B(\xi,s)\deq\{D=\xi+\sum_ia_ig_i\in V\ \mid\ |a_i|\leq s\}$, a closed ``ball'' around $\xi$. Since $\xi\in G$, there exists $r>0$ that depends only on $\xi$, $g_i$, and $G$ such that $B(\xi,2r)\subset G$. Furthermore, for every $D\in B(\xi,r)$, we have $B(D,r)\subset B(\xi,2r)\subset G$.

%Let $D=\xi+\sum_ia_ig_i$ and $D'=\xi+\sum_ia'_ig_i$ be elements of $B(\xi,r)$. Let $b\deq||D-D'||=\max|a_i-a'_i|\leq 2r$. 
%We want to prove that $|f(D)-f(D')|\leq C\cdot b$ for some constant $C$ that depends only on $\xi$, and not on $b$. 
%By the monotonicity property of $f$, we have
%\[f(D)-f(D+b\gamma)\leq f(D)-f(D')\leq f(D'+b\gamma)-f(D').\]
%Let $\delta\deq D-r\gamma\in G$. Then $D+b\gamma=\frac{r+b}rD-\frac br\delta$. By homogeneity and monotonicity this implies
%\[
%f(D+b\gamma)-f(D)\leq \Bigl(\Bigl(\frac{r+b}r\Bigr)^d-1\Bigr)\cdot f(D)\leq \Bigl(\Bigl(\frac{r+b}r\Bigr)^d-1\Bigr)\cdot f(\xi+r\gamma)). 
%\]
%The expression $\frac 1b\bigl((\frac{r+b}r)^d-1\bigr)$ is bounded when $b\in(0,2r]$.
%\end{proof}

\subsection{Historical background and the successive minima of Ambro--Ito}\label{ss:motivation}
Both infinitesimal successive minima and asymptotic partial jet separation order are inspired by two special cases treated previously in the literature. On one hand, we have the classical theory of Seshadri constants and Fujita--Nakayama invariants. On the other hand, Ambro--Ito \cite{AI} have introduced a local version of successive minima. We explain the connection between these invariants.

In the classical theory, when $\xi$ is a nef class on an irreducible projective variety $X$, its \emph{Seshadri constant} at a smooth point $x\in X$ is defined to be 
\[
\epsilon(\xi;x)\ \deq \ \inf_C\frac{\xi\cdot C}{\mult_xC}\ = \ \max\{t\geq 0\ \mid\ \xi_t\text{ is nef}\} \ ,
\]
where the infimum is over closed irreducible curves $C\subseteq X$ through $x$. When $\xi=L$ is a big and nef $\bb Q$-Cartier $\bb Q$-divisor, and $X$ is normal, $\epsilon(L;x)$ measures the asymptotic order of jet separation of $L$ at $x$. When $L$ is ample, the normality assumption is not needed.

This was generalized to the case when $\xi$ is represented by a big Cartier divisor $L$ on the projective variety $X$. The \emph{moving Seshadri constant} 
\[\epsilon(||L||;x)\deq\lim_{m\to\infty}\frac 1m\max\{s\ \mid\ H^0(X,mL)\to{\cal O_X}/{\frak m_x^{s+1}}\text{ is surjective }\},\] 
was introduced in \cite{Nak03} and further studied in \cite{ELMNP09}. When $X$ is smooth, there is an equivalent interpretation as limit of Seshadri constants of Fujita approximations of $L$.
The moving Seshadri constant extends with a continuity argument to all classes in ${\rm Big}_x(X)$. When $X$ is normal, $\xi$ is nef and big and $x\in X$ is a smooth point, then $\epsilon(||\xi||;x)=\epsilon(\xi;x)$ by \cite[Theorem 5.1.17]{Laz04}. A related result valid also when $x$ is singular is \cite[Theorem 5.3]{FM21}.

Another invariant for a class $\xi\in N^1(X)_\RR$ is the Fujita--Nakayama invariant (or \emph{width} as it will soon make sense to call it for convex geometric reasons) of $\xi$ at a smooth point $x\in X$. It is defined to be
\[
\mu(\xi;x)\ \deq\ \max\{t\geq 0\ \mid\ \xi_t\text{ is pseudo-effective}\}.
\] 
It is clear that both invariants measure some change in the base loci of the classes $\xi_t$. The Seshadri constant tells us when the class $\xi_t$ starts having a non-empty base locus, whenever $\xi$ is ample. The width instead describes when the base locus starts being the entire space. This remark leads to the following lemma:
\begin{lemma}[Fujita--Nakayama invariant and Seshadri constant]\label{lem:seshadriwidth} Let $X$ be a projective variety of dimension $n$, consider $x\in X$ a smooth point, and $\xi\in{\rm Big}_x(X)$. Then 
\begin{enumerate}
\item $s_n(\xi;x)=\epsilon_n(\xi;x)=\mu(\xi;x)$.
\smallskip

\item $s_1(\xi;x)=\epsilon(||\xi||;x)\leq \epsilon_1(\xi;x)$. When $X$ is normal, all three invariants are equal. When $X$ is only irreducible but $\xi$ is ample, all three invariants are equal to the usual Seshadri constant $\epsilon(\xi;x)$.
\end{enumerate}
\end{lemma}

\begin{proof}
Since $\epsilon_i(\cdot;x)$, $s_i(\cdot;s)$, $\mu(\cdot;x)$, $\epsilon(||\cdot||;x)$ are homogeneous and continuous on ${\rm Big}_x(X)$, we may assume that $\xi$ is represented by a Cartier divisor $L$.

$(1)$. Theorem \ref{thm:successivejets} implies $s_n(L;x)\leq\epsilon_n(L;x)$. From the definition of the infinitesimal successive minima, $\epsilon_n(L;x)\leq\mu(L;x)$. If $0<t<\mu(L;x)$ is rational then $L_t$ is big, thus $mL_t$ is effective for $m$ sufficiently large and divisible. There exists an integer $p\geq mt$ such that $m\pi^*L-pE$ is linearly equivalent to an effective divisor without $E$ in its support. The associated section has nontrivial restriction to $E$ and thus does not vanish at a general point $\bb P^{n-n}\subset E$, giving $S_n(mL;x)\geq p\geq mt$ and $s_n(L;x)\geq t$, hence $s_n(L;x)\geq\mu(L;x)$. 

$(2)$. We prove first the equality between the first two invariants.
Note $1$-partial $s$-jet separation is related to the classical $s$-jet separation. If $L$ separates $s$-jets at $x$, then $H^0(X,L)\to\cal O_X/\frak m_x^{s+1}$ is surjective and so is \[
H^0(\overline X,L_s)\simeq H^0(X,L\otimes\frak m_x^s)\to\frak m_x^s/\frak m_x^{s+1}\simeq H^0(E,\cal O(s)) ,
\]
therefore $L$ also $1$-partially separates $s$-jets at $x$. In fact the invariants coincide asymptotically: 
	\[s_1(L;x)=\lim_{m\to\infty}\frac 1mS_1(mL;x)=\epsilon(||L||;x)\]
by \cite[Proposition 2.25, and Lemma 2.28]{AI}. 

When $X$ is normal, the equality between all three invariants in $(2)$ follows from Theorem \ref{thm:successivejets} and Corollary \ref{cor:normaljetloci}. When $X$ irreducible, $\epsilon(\cdot ;x)$, $\epsilon(||\cdot||;x)$ and $\epsilon_1(\cdot ;x)$ are continuous and homogeneous inside the ample cone. So, the equality follows when it holds for $\xi=L$ an ample Cartier divisor. In this case, we know $\epsilon(L;x)=\epsilon(||L||;x)$ by our definition of the moving Seshadri constant in the non-normal case. It remains to show that $\epsilon(L;x)\geq \epsilon_1(L;x)$. For $0<t<\epsilon_1(L;x)$, we have ${\bf B}_+(L_t)\cap E=\varnothing$. The class $L_t$ has nonnegative intersection with the strict transforms of any curve $C\subset X$, not passing through $x$. Since $t>0$, it also has positive intersection with any curve $T\subset E$. Now, let $C\subseteq X$ be a curve through $x$, such that its strict transform has the property $L_t\cdot \overline C<0$. This implies $\overline C\subseteq {\bf B}_+(L_t)$ and we are lead to a contradiction as $\varnothing \neq \overline C\cap E\subseteq {\bf B}_+(L_t)\cap E$.  Thus $L_t$ is nef, showing $\epsilon_1(L;x)\leq\epsilon(L;x)$.
\end{proof}

\begin{corollary}\label{cor:EnotinB+}
	Let $X$ be a projective variety of dimension $n$, let $x\in X$ be a smooth point and $\xi\in{\rm Big}_x(X)$. If $0<t<\mu(\xi;x)$, then $E$ is not a component of ${\bf B}_+(\xi_t)$. 	
\end{corollary}
\begin{proof}
	It follows from Lemma \ref{lem:seshadriwidth}.$(1)$. A different proof was found in \cite{KL17}, using \cite{FKL16}.
\end{proof}

\noindent \cite[Theorem 6.2]{ELMNP09} show the continuity of the moving Seshadri constant on $N^1(X)_{\bb R}$ for smooth projective $X$, using the continuity of restricted volumes, a highly nonelementary result. By Lemma \ref{lem:seshadriwidth} it is tempting to use the continuity of $\epsilon_1$ instead. This proves it only inside the cone ${\rm Big}_x(X)$. As with \cite{ELMNP09}, the difficulty is in showing that these invariants tend to zero as $\xi$ approaches the boundary.

In the remaining part of this subsection we explain the connection between our infinitesimal invariants and the local successive minima of Ambro--Ito \cite{AI}, which inspired parts of the material above. The main difference is that \cite{AI} studies base loci directly on $X$. Fix $X$ a projective variety, $x\in X$ a point (possibly singular) and $L$ a line bundle on $X$. Then for $t\geq 0$ consider the closed subset
	\begin{equation}\label{eq:Bs+}
	{\rm Bs}|\frak m_x^{t+}L|_{\bb Q} \ \subseteq \ X \ ,
	\end{equation} 
	the common vanishing locus of sections of $pL$ that have multiplicity bigger than $pt$ at $x$ for all $p\geq 1$.

\begin{definition}[The successive minima of \cite{AI}]\label{def:aiminima}Under the above assumptions, the \emph{local successive minima} of $L$ at $x$ are
	\[\epsilon^{\rm loc}_{i}(L;x)\deq\min\{t\geq 0\ \mid\ \dim_x{\rm Bs}|\frak m_x^{t+}L|_{\bb Q}\geq i\}=\sup\{t\geq 0\mid\ \dim_x{\rm Bs}|\frak m_x^{t+}L|_{\bb Q}\leq i-1\},\]
	where $\dim_xZ$ is the maximal dimension of an irreducible component of $Z$ that contains $x$.
\end{definition}
In \cite{AI} the authors denote $\epsilon^{\rm loc}_i(L;x)$ by $\epsilon_{n-i+1}(L;x)$, preferring to work with codimension and call them the \emph{successive minima} of $L$ at $x$.
In order to connect them to our infinitesimal version, we introduce a concept of \emph{dimension around $E$}.
\begin{definition}
	For a closed subset $Z\subseteq\overline X$ set
	\[
	\dim_EZ \ \deq \ \max \{\dim(Z') \ | \ Z' \textup{-irreducible component of } Z \textup{ with } Z'\nsubseteq E \textup{ and } Z'\cap E\neq \varnothing\}\ .
	\]
This considers only those $Z'$ that are strict transform of a positive dimensional subvariety of $X$, passing through $x$. By convention $\dim_EZ=-1$ if $Z\subseteq E$ or when $Z\cap E=\varnothing$.	
\end{definition}
With this in hand, the following corollary gives a description of the Ambro--Ito successive minima in terms of data coming from the blow-up of a smooth point.
\begin{corollary}\label{cor:B+B}
	Let $X$ be a projective variety of dimension $n$, and let $x\in X$ a smooth point. Let $L$ be a Cartier divisor such that $x\not\in{\bf B}_+(L)$. Let $t\in\bb R_{>0}$. Then
	\begin{enumerate}
		\item ${\rm Bs}|\frak m_x^{t+}L|_{\bb Q}=\pi(\bigcap_{t<s\in\bb Q}{\bf B}(L_s))\cup\{x\}$ and ${\rm Bs}|\frak m_x^{t+}L|_{\bb Q}\subseteq\pi({\bf B}_+(L_t))\cup\{x\}\subseteq{\rm Bs}|\frak m_x^{t+}L|_{\bb Q}\cup{\bf B}_+(L)$.
		\smallskip
		
		\item $\epsilon^{\rm loc}_{i}(L;x)=\textup{min}\{t>0\ | \ \textup{dim}_E\,\Bplus(\xi_t)\geq i\}=\sup\{t>0\ \mid\ \dim_E\,\Bplus(\xi_t)\leq i-1\}$ for every $1\leq i\leq n$.
	\end{enumerate}
\end{corollary}
\begin{proof}
The first statement in $(1)$ is due to the fact that $\pi:\overline{X}\rightarrow X$ is a fiber space. The second one follows from Lemma \ref{lem:B-ray}.(3). For $(2)$, by $(1)$ we know ${\rm Bs}|\frak m_x^{t+}L|_{\bb Q}$ and $\pi({\bf B}_+(L_t))\cup\{x\}$ agree away from ${\bf B}_+(L)$, in particular they agree around $x$. The successive minima of \cite{AI} are determined locally around $x$. 
\end{proof}

The only restriction in the Definition \ref{def:aiminima} from \cite{AI} is that $L$ should be a $\bb Q$-Cartier $\bb Q$-divisor. The point $x$ can be singular, and it may belong to ${\bf B}_+(L)$. Their definition is intrinsic to $X$. Our interpretation in terms of augmented base loci on the blow-up of smooth points has its advantages. For one, ${\bf B}_+(L_t)$ are better studied than ${\rm Bs}|\frak m_x^{t+}L|_{\bb Q}$. Second, our definition works for $\bb R$-divisors. Finally and more importantly, the convex geometric results of this paper are better suited for infinitesimal successive minima.

Before moving forward we do want to emphasize that local successive minima satisfy properties similar to those of the infinitesimal version, when we look at smooth points.
\begin{lemma}[Basic properties of local successive minima]\label{lem:successiveproperties}
Let $X$ be a projective variety of dimension $n$, let $x\in X$ be a smooth point, and $\xi\in{\rm Big}_x(X)$.
\begin{enumerate}
\item The local successive minima form an increasing set of homogeneous numerical invariants: \newline $0<\epsilon^{\rm loc}_1(\xi;x)\leq \epsilon^{\rm loc}_2(\xi;x)\leq\ldots\leq\epsilon^{\rm loc}_{n}(\xi;x)$.
\smallskip

 \item $\epsilon^{\rm loc}_n(\xi;x)=\mu(\xi;x)$ is the width. Moreover, $\epsilon^{\rm loc}_1(\xi;x)=\epsilon_1(\xi;x)$ is the Seshadri constant $\epsilon(\xi;x)$, whenever $\xi$ is ample. When $X$ is normal, then again $\epsilon_1^{\rm loc}(\xi;x)=\epsilon_1(\xi;x)$.
\smallskip

 \item $\epsilon_i(\xi;x)\leq \epsilon^{\rm loc}_{i}(\xi;x)$ for each $i=1,\ldots ,n$. \end{enumerate}
\end{lemma}
\begin{remark}
Example \ref{ex:hyper3} shows that the inequalities in $(3)$ can be strict, when $2\leq i\leq n-1$. This happens, when $\dim_E{\bf B}_+(\xi_t)\leq i$ and $\dim{\bf B}_+(\xi_t)\cap E\geq i$ for some $t>0$ and $i=2,\ldots ,n-1$. Equivalently, there is a nonempty irreducible component of ${\bf B}_+(\xi_t)$ of dimension $\geq i$ entirely contained in $E$.
\end{remark}
\begin{proof}
Part $(1)$ is similar to Lemma \ref{lem:infsuccessiveproperties}.  For part $(3)$ use Corollary \ref{cor:B+B}.(2) and that if $\dim_E{\bf B}_+(\xi_t)\geq i$, then $\dim{\bf B}_+(\xi_t)\cap E\geq i-1$. The description in $(2)$ for the width is due to the fact that $\dim_E\,\Bplus(\xi_t)\leq n-1$ if and only if $\dim\,\Bplus(\xi_t)\leq n-1$.
The equality $\epsilon_1^{\rm loc}(\xi;x)=\epsilon_1(\xi;x)$ in the normal case follows from \cite[Remark 2.23]{AI} and Lemma \ref{lem:seshadriwidth}.(2). For the Seshadri constant in the ample case, by $(3)$, it remains to prove the inverse inequality. Take $0\leq t<\epsilon^{\rm loc}_1(\xi;x)$. We want to show that $\xi_t$ is nef. This is analogous to the argument in Lemma \ref{lem:seshadriwidth}.(2). 
\end{proof}

Successive minima relate to nefness and bigness of divisors. There is also a connection with movability.

\begin{example}
	Let $X$ be smooth and let $L$ be ample. Consider the \emph{movable threshold} \[\nu=\nu(L;x)\deq\sup\{t\geq 0\ \mid\ L_t\in\Mov(\overline X)\}.\] 
	Then ${\bf B}_-(L_t)$ contains divisors precisely for $t>\nu$, since the divisors contained in ${\bf B}_-(L_t)$ are the support of the negative (non-movable) part in the Nakayama $\sigma$-Zariski decomposition \cite{Nak04}. From Lemma \ref{lem:B-ray}.$(3)$ we deduce that $\dim{\bf B}_+(L_t)\geq n-1$ precisely for $t\geq\nu$. In particular $\epsilon^{\rm loc}_{n-1}(L;x)\geq\nu(L;x)$. The inequality could in principle be strict if the divisorial components of ${\bf B}_+(L_{\nu})$ are disjoint from $E$.\qed
\end{example}

Equality holds when $X$ has Picard rank 1 because in this case any prime divisor on $\overline X$ disjoint from $E$ is a pullback from $X$, hence nef, and so it cannot be contained in ${\bf B}_-(L_t)$ for $t\in(\nu,\mu]$.
It is an interesting question whether there exists an ample line bundle $L$ and $t>0$ such that $\dim {\bf B}_+(L_t)=\dim X-1$, and all its divisorial components are disjoint from $E$.

%\begin{example}The first inclusion in Lemma \ref{lem:B-ray}.$(3)$ is not always an equality for $t'>t$ sufficiently close to $t$. This is related to the concept of stability of divisors (cf.~\cite{ELMNP06}). Let $X={\rm Bl}_p\bb P^2$ be the blow-up of a point with exceptional divisor $F$. Let $L$ be the pullback of $\cal O_{\bb P^2}(1)$ from $\bb P^2$. Then ${\bf B}_-(L)=\varnothing$, but ${\bf B}_+(L)=F$. Let $x\in X\setminus F$ and $\pi:\overline X\to X$ the blow-up of $x$ with exceptional divisor $E$. Then $\epsilon(L;x)=1$. Thus ${\bf B}_-(L_t)=\varnothing$ for all $t\in[0,1]$, but ${\bf B}_+(L_t)\supseteq F$ for all $t\geq 0$ since $(L_t\cdot F)=0$. \qed	
%\end{example}

\section{Borel-fixed properties for generic valuations on $\PP^{n-1}$}
Above we showed that the infinitesimal successive minima can be interpreted as invariants from generic linear subspaces of the exceptional divisor $E\simeq \bb P^{n-1}$. Such subspaces are components of generic complete linear flags on $\PP^{n-1}$. Following \cite{LM09}, we can use such a flag and associate a convex set to any graded linear series on $E$. This convex set has a valuative nature essentially capturing the lowest degree-lexicographic term in the local expansion around $x$ of any section from the graded linear series. The genericity condition on the flag invites comparison with the theory of generic initial ideals. These are Borel-fixed monomial ideals by \cite{Gal74, BS87} and then the set of exponent vectors of monomials in these ideals have a certain structure that will resemble the structure of our valuative sets.

In this section we translate these algebraic results to the theory of Newton--Okounkov bodies on $E\simeq \PP^{n-1}$. We show that these convex sets considered with respect to generic linear flags satisfy a Borel-fixed type property. We deduce polytopal bounds of the convex bodies.

\subsection{Basic preliminaries of convex geometry in $\RR^n$}
We first introduce basic notation about objects in $\RR^n$. Note that we will will apply it both to $\bb R^{n-1}$ and $\bb R^n$.
We generally write $(\nu_1,\ldots,\nu_n)\in\bb R^n$ for the coordinates of a vector with respect to the standard basis ${\bf e}_1,\ldots ,{\bf e}_n$.
Any bounded convex set $P\subseteq \RR^n$ has Euclidean volume in $\RR^n$, denoted ${\rm vol}_{\bb R^n}(P)$ so that the volume of the standard simplex is $\frac{1}{n!}$ and the volume of the unit hypercube is $1$.

We introduce the following polytopal shapes in $\RR^n$ that play an important role throughout the paper, especially when describing upper and lower bounds on generic iNObodies.
\begin{definition}[The bounding polytopes]
Fix real numbers $t_1,\ldots,t_n> 0$. We denote the simplicial polytope 
\[
\simplex(t_1,\ldots,t_n)\deq\textup{convex hull}\,({\bf 0},t_1{\bf e}_1,\ldots,t_n{\bf e_n})\ .
\]
When $t_1=\ldots =t_n=t$, then we denote by $\simplex\,_t\deq\simplex(t,\ldots,t)$ the standard simplex of size $t$.
Let 
\[
\polytope{}(t_1,\ldots,t_{n})\deq\bigl\{(\nu_1,\ldots,\nu_{n})\in\bb R^{n}_+\ \mid\ \sum\nolimits_{j=1}^i\nu_j\leq t_i\ \forall\ i\in\{1,\ldots,n\}\bigr\}\ .
\]
\end{definition}
\noindent  The notation is motivated by the picture in $\bb R^2$ for $t_1<t_2$. The simplex $\simplex\,_t$ was denoted by $\Delta_t$ in \cite{KL17}. The polytope $\polytope{}(t_1,\ldots,t_{n})$ is related to $\square(t_{n},\ldots,t_{1})$, defined in \cite[p11]{AI}.

%Our convex sets satisfy the following simple properties.
\begin{lemma}[Properties of the bounding polytopes]\label{lem:boundingpolytopes}Assume $0\leq t_1\leq t_2\leq\ldots\leq t_n$. Then
\begin{enumerate}[(1).]
\item $\simplex(t'_1,\ldots,t'_n)\subseteq\polytope(t_1,\ldots,t_n)$ if and only if $t'_i\leq t_i$ for all $i$. %Moreover, $\simplex(t_1,\ldots,t_n)$ is the largest simplex of this type contained in $\polytope(t_1,\ldots,t_n)$.
\smallskip

\item  $\simplex(t,\ldots,t)=\polytope(t,\ldots,t)=\simplex\,_t$.
\smallskip

\item $\polytope(t'_1,\ldots,t'_n)\subseteq[0,t_1]\times\ldots\times[0,t_{n}]$ if and only if $t'_i\leq t_i$ for all $i$.
\end{enumerate} 
\end{lemma}

\subsection{{\bf Borel-fixed sets}} We begin with the discrete version of Borel-fixed sets. Set 
\[
{\bf f}_i\deq -{\bf e}_i+{\bf e}_{i+1}\text{ for all }i=1,\ldots ,n-2\text{ and }{\bf f}_{n-1}\deq-{\bf e}_{n-1}\ .
\] 
Inspired by \cite{Gal74, BS87}, we start with the following definition:
\begin{definition}
	A finite set $S\subset\bb Z^{n-1}_+$ is called \emph{Borel-fixed}, whenever the inclusion
	\[
	(S+{\bf f}_i)\cap\bb Z^{n-1}_+\subseteq S \ 
	\]
holds for any $i=1,\ldots ,n-1$.
Equivalently one can use all vectors ${\bf e}_i-{\bf e}_j$ and $-{\bf e}_i$ instead of the ${\bf f}_i$ above, for all $1\leq i\leq j\leq n-1$.
\end{definition}
\smallskip

\begin{example}[Borel-fixed sets]~
	\begin{enumerate}
		\item If $0\leq t_1\leq\ldots\leq t_{n-1}$ are real numbers, then $\simplex(t_1,\ldots,t_{n-1})\cap\bb Z^{n-1}_+$ and $\polytope(t_1,\ldots,t_{n-1})\cap\bb Z^{n-1}_+$ are Borel-fixed. Up to decreasing $n$, we may assume $t_1>0$. In this case the simplex is described in $\bb R^{n-1}_+$ by $\sum_{i=1}^{n-1}\frac{\nu_i}{t_i}\leq 1$. It follows easily that the integer points in the simplex form a Borel-fixed set. 
		For the polytope, the claim follows directly from its defining equations.
		\smallskip
		
		\item If $S,S'\subset\bb Z^{n-1}_+$ are Borel-fixed, then so are $S\cap S'$, $S\cup S'$, and $S+S'=\{s+s'\ \mid\ s\in S,\ s'\in S'\}$ - their Minkowski sum. %(The intersection is clear. For the sum, if ${\bf a}=s+s'$ and $a_i\geq 1$ then either $s\geq 1$ or $s'\geq 1$. Say $s\geq 1$. Then $s+{\bf f}_i\in S$ and $(s+s')+{\bf f}_i=(s+{\bf f}_i)+s'\in S+S'$.)		
		\smallskip
		
		\item For fixed $n\geq 2$, there are few Borel-fixed finite sets with a small number of points. If $|S|=1$, then $S=\{{\bf 0}\}$. If $|S|=2$, then $S=\{{\bf 0},{\bf e}_{n-1}\}$. If $|S|=3$, then $S=\{{\bf 0},{\bf e}_{n-1},2\cdotp{\bf e}_{n-1}\}$ or $S=\{{\bf 0},{\bf e}_{n-1},{\bf e}_{n-2}\}$, when $n\geq 3$. All these examples are generated by one element, whereas the Borel-fixed set $S=\{{\bf 0},{\bf e}_{n-2},{\bf e}_{n-1},2\cdotp{\bf e}_{n-1}\}$ is not.
		\smallskip
		
		\item Let $S\subset\bb Z^{n-1}_+$ be a set that is saturated with respect to deglex, i.e., for all ${\bf a}\in S$ and all ${\bf b}\in\bb Z^{n-1}_+$ such that ${\bf b}\leq_{\rm deglex}{\bf a}$, then ${\bf b}\in S$. Then $S$ is Borel-fixed, as ${\bf a}+{\bf f}_i<_{\rm deglex}{\bf a}$ for all $i$. However, note that 
\[
(2,1)<_{\rm deglex}(0,4)\in \simplex(1,4)\cap\polytope(1,4),
\]
but $(2,1)$ is in neither polytope. In particular, the integral points of the convex set on the right form a Borel-fixed set, however it is not saturated.
	\end{enumerate}
\end{example}

Borel-fixed sets satisfy strong bounds in terms of their widths:

\begin{definition}\label{def:widthborelfixedset}
	Let $S\subset\bb Z_+^{n-1}$ be a nonempty finite set.
	For each $i=1,\ldots ,n-1$, the $i$-\emph{th width} of $S$ is
	\[w_i(S)\deq \textup{max}\{w\ |\  \exists\ {\bf a}=(a_1,\ldots ,a_i,\ldots ,a_{n-1})\in S,\ a_i=w\}\ .\]
\end{definition}

\begin{proposition}\label{prop:Borelfixeddiscreteproperties}
	Let $S\subset\bb Z^{n-1}_+$ be a nonempty Borel-fixed set. Then
	\begin{enumerate}
		\item $w_1(S)\leq\ldots\leq w_{n-1}(S)$ and $w_i(S)=\max\{a\ \mid\ a\cdot{\bf e}_i\in S\}$.
		\smallskip
		
		\item $w_i(S)\cdot\simplex(0,\ldots,0,1,\ldots,1)\cap\bb Z^{n-1}_+\subseteq S$ for all $i=1,\ldots ,n-1$ with the first $1$ appearing in the $i$-th spot.
		\smallskip
		
		\item $S\subseteq\polytope(w_1(S),\ldots,w_{n-1}(S))$.
	\end{enumerate}
	In particular,  $1+\sum_{i=1}^{n-1}{{w_i(S)+n-i-1}\choose{n-i}}\ \leq\  |S|\ \leq\  \prod_{i=1}^{n-1}(w_i(S)+1)$.
\end{proposition}
\noindent Usually one finds the $i$-th width of a set $S$ by projecting $S$ on the $i$-th coordinate axis and picking the largest value in the image. Part (1) says that for Borel-fixed shapes we can find the $i$-th width by interesting $S$ with the $i$-th coordinate axis, a much stronger statement. 
\begin{proof}
	 $(1)$ If ${\bf a}=(a_1,\ldots,a_{n-1})\in S$ has $a_i=w_i(S)$ for some $1\leq i\leq n-2$, then ${\bf a}+w_i(S)\cdot{\bf f}_i\in S$ giving us $w_i(S)+a_{i+1}\leq w_{i+1}(S)$. The monotonicity claim follows. 
	 For the second part, note that each ${\bf e}_j=-\sum\nolimits_{k=j}^{n-1}{\bf f}_k$. Thus, from the definition, if ${\bf a}\in S$ has $a_i=w_i(S)$, then $w_i(S)\cdot{\bf e}_i={\bf a}+\sum_{j\neq i}a_j\cdot(-{\bf e}_j)\in S$, implying the new description of the $i$-th width.
	\smallskip
	
	$(2)$ The set on the left is ``generated'' as a discrete Borel-fixed set by the point $w_i(S)\cdot{\bf e}_i$. %We want to show that $(0,\ldots ,0,a_i,\ldots ,a_{n-1})\in S$, whenever $a_i+\ldots +a_{n-1}\leq w_i(S)$. Taking into account the equality
	%\[
	%(0,\ldots ,0,a_i,\ldots ,a_{n-1})  \ = \ (a_i+\ldots +a_{n-1}){\bf e}_{i}\ + \ \sum\nolimits_{j=i+1}^{n-1} a_j\cdot \big(\sum\nolimits_{k=i}^{j}{\bf f}_k\big)
	%\]
	%and the Borel property for $S$, it suffices to show that $m\cdot{\bf e}_{i}\in S$ for any positive integer $m\leq w_i(S)$. Since ${\bf e}_i=-\sum\nolimits_{j=i}^{n-1}{\bf f}_j$, then again by the Borel property, the latter is implied whenever $w_i(S)\cdot{\bf e}_{i}\in S$, done in $(1)$. 
	\smallskip
	
	$(3)$ If ${\bf a}=(a_1,\ldots ,a_{n-1})\in S$, then the Borel-fixed property implies that $(a_1+\ldots+a_i)\cdot{\bf e}_i\in S$ for all $i$.
	%we have the equality
	%\[
	%(a_1+\ldots+a_i)\cdot{\bf e}_i={\bf a}+\sum_{j=1}^{i-1}(a_1+\ldots+a_j)\cdot {\bf f}_j+\sum_{j=i+1}^{n-1}a_j\cdot(\sum\nolimits_{k=j}^{n-1}{\bf f}_k)\in S.
	%\]
	%From the definition of the widths, we deduce $a_1+\ldots+a_i\leq w_i(S)$ and obtain the upper bound. 
	\smallskip
	
	For the last part, the upper bound is a consequence of $(3)$, while the lower bound follows from $(2)$. We leave the details as an exercise.
	%For the last part, note that the upper bound is a consequence of $(3)$ and the fact that the polytope $\polytope(w_1(S),\ldots,w_{n-1}(S))$ is contained in a box with lattice vertices and side lengths $w_i(S)$. For the lower bound use from $(2)$ that $w_i(S)\cdot{\bf e}_i\in S$. The left hand side of $(2)$ has ${{w_i(S)+n-i}\choose{n-i}}$ points. Their union is not disjoint since the slice $\nu_i=0$ which contains ${{w_i(S)+n-i-1}\choose{n-i-1}}$ points is also contained in the corresponding simplex for $w_{i+1}(S)$. This and and Pascal triangle identity imply the lower bound.
\end{proof}

\subsection{{\bf Borel-fixed shapes}} The Borel-fixed property for finite sets extends naturally to convex sets. 

\begin{definition}
	A compact convex set $P\subset\bb R^{n-1}_+$ is called \emph{Borel-fixed} if the inclusion
	 \[
	(P+t\cdot {\bf f}_i)\cap\bb R^{n-1}_+\subseteq P
\]
hold for all $t\geq 0$ and all $i=1,\ldots ,n-1$.  Equivalently, this is saying pointwise that if $(a_1,\ldots,a_{n-1})\in P$, then $(a_1,\ldots,a_{i-1},0,a_{i}+a_{i+1},a_{i+2},\ldots,a_{n-1})\in P$ for all $i=1,\ldots , n-2$, and $(a_1,\ldots,a_{n-2},0)\in P$.
\end{definition}

\begin{example}[Borel-fixed shapes]$ $
	\begin{enumerate}
		\item If $S\subset\bb Z^{n-1}_+$ is a finite Borel-fixed set, then $P\deq\textup{convex hull}(S)$ is a Borel-fixed set in $\bb R^{n-1}_+$.
		\smallskip
		
		\item The intersection, scaling, Minkowski sum, and the convex hull of the union of Borel-fixed sets in $\bb R^{n-1}_+$ is again Borel-fixed.
		\smallskip
		
		\item If $P\subseteq\bb R^{n-1}_+$ is Borel-fixed, then for each $i=1,\ldots ,n-1$, any nonempty slice $P_{\nu_1=t_1,\ldots,\nu_i=t_i}$ is Borel-fixed in $\{(t_1,\ldots,t_i)\}\times\bb R^{n-i-1}_+$. 
		\smallskip
		
		\item Fix an increasing sequence of real numbers $0\leq t_1\leq\ldots\leq t_{n-1}$. The simplex $\simplex(t_1,\ldots,t_{n-1})$ and the polytope $\polytope(t_1,\ldots,t_{n-1})$ are Borel-fixed. 
		%Moreover, fix a real number $t_n\geq t_{n-1}$, then for any $0\leq t\leq t_n$ the vertical slices $\inverted(t_1,\ldots,t_n)_{\nu_1=t}$ and $\invertpoly(t_1,\ldots,t_n)_{\nu_1=t}$ are also Borel-fixed in $\bb R^{n-1}_+$, even though $\inverted(t_1,\ldots,t_n)$ and $\invertpoly(t_1,\ldots,t_n)$ are not. 
		
		%The first part follows easily from the defining descriptions. For the second part, assume $t_1>0$, since otherwise we can reduce $n$. With $T$ as in \eqref{eq:defT}, we have $\inverted(t_1,\ldots,t_n)=T^{-1}(\simplex(t_1,\ldots,t_n))$ which can then be described by
		%\[
		%\nu_1\geq\nu_2+\ldots+\nu_n, \ \nu_i\geq 0, \ \textup{and} \ \sum\nolimits_{i=2}^n\frac{\nu_i}{t_{i-1}}+\frac{\nu_1-\nu_2-\ldots-\nu_n}{t_n}\leq 1.
		%\]
		%Modulo algebraic manipulation, we obtain $\inverted(t_1,\ldots,t_n)_{\nu_1=t}=\simplex\,_t\cap \simplex(\frac{t_1(t_n-t)}{t_n-t_1},\ldots,\frac{t_{n-1}(t_n-t)}{t_n-t_{n-1}})$ which is Borel-fixed. Here by convention the intersection is ${\bf 0}$ when $t_n=t$, and $\frac{t_i(t_n-t)}{t_n-t_i}=t$ whenever $t_n=t_i>t$. Finally, $\invertpoly(t_1,\ldots,t_n)_{\nu_1=t}=\simplex\,_t\cap\polytope(t_1,\ldots,t_{n-1})=\polytope(\min(t_1,t),\ldots,\min(t_{n-1},t))$ is also Borel-fixed.
	\end{enumerate}
\end{example}
\medskip

\begin{example}[Generation of Borel-fixed shapes]~\\
	Fixing an increasing sequence of real numbers $0\leq t_1\leq\ldots\leq t_{n-1}$, then we have the identity
	\[
	(a_1,\ldots,a_{n-1})=(t_1,t_2-t_1,t_3-t_2,\ldots,t_{n-1}-t_{n-2})+\sum\nolimits_{i=1}^{n-1}(t_i-\sum\nolimits_{j=1}^{i}a_j)\cdot{\bf f}_i. 
	\]
	Applying this identity to the points in $\polytope(t_1,\ldots,t_{n-1})$, implies that this convex shape is generated as a Borel-fixed set by the vertex $(t_1,t_2-t_1,t_3-t_2,\ldots,t_{n-1}-t_{n-2})$.\\ 
	\noindent 
	On the other hand, $\simplex(t_1,\ldots,t_{n-1})$ is usually not generated by one of its vertices. So, for example for $\simplex(1,2)$ the vertex $(1,0)$ generates $\simplex(1,1)$ while $(0,2)$ generates the segment $\simplex(0,2)$.\qed
\end{example}

In the setting of Borel-fixed shapes, the widths defined analogously to Definition \ref{def:widthborelfixedset} again play an important role in bounding our convex sets.

\begin{proposition}[Polytopal bounds of Borel-fixed shapes]\label{prop:borelfixedshapesproperties}
	Let $P\subset\bb R^{n-1}_+$ be a non-empty Borel-fixed compact convex set. Then the widths form an increasing sequence $w_1(P)\leq\ldots\leq w_{n-1}(P)$ and
	\[
	\simplex(w_1(P),\ldots,w_{n-1}(P))\subseteq P\subseteq\polytope(w_1(P),\ldots,w_{n-1}(P))\ .
	\]
	In particular $\frac{1}{(n-1)!}\prod_{i=1}^{n-1}w_i(P)\leq{\rm vol}_{\bb R^{n-1}}(P)\leq \prod_{i=1}^{n-1}w_i(P)$. 
\end{proposition}
\noindent Observe that when the Borel-fixed property holds, then the widths, the optimal lower simplicial bound, and optimal upper bound of shape $\polytope$ (or box shaped) determine each other.
\begin{proof}
	The proof is analogous to that of Proposition \ref{prop:Borelfixeddiscreteproperties}. The new input is the use of convexity of $P$ to prove the left inclusion from $w_i(P)\cdot{\bf e}_i\in P$. The upper bound on the volume follows from Lemma \ref{lem:boundingpolytopes}.(3).
\end{proof}
\smallskip

\begin{example}[Abundance of Borel-fixed shapes]~\\
	Let $n=3$ and an increasing sequence of real numbers  $0\leq t_1\leq t_2$. Then in $\RR^2_+$ any intermediate convex body $\simplex(t_1,t_2)\subseteq P\subseteq\polytope(t_1,t_2)$ is Borel-fixed.\\
	\noindent
	This is no longer true in $\RR^3_+$. Consider the vector $(1,1,1)\in\polytope(1,2,3)\setminus\simplex(1,2,3)$. The convex hull $P$ of $\simplex(1,2,3)\cup\{(1,1,1)\}$ is not Borel-fixed since it does not contain the point $(1,0,2)$.
\end{example}

\subsection{Flag valuation sets on the projective space and the Borel-fixed property}
We turn our attention to a more algebro-geometric setting. The goal is to use valuations along generic linear flags in the projective space to produce Borel-fixed sets and shapes. 

Consider on $\PP^{n-1}$ a complete flag of linear subspaces 
\[
Y_{\bullet} \ : \ \PP^{n-1}=Y_1\supseteq Y_2\supseteq \ldots \supseteq Y_n=\{x\} \ ,
\] 
where $Y_i\subseteq \PP^{n-1}$ is a linear subspace of dimension $n-i$, for any $i=1,\ldots ,n$. 

Consider $D$ to be an effective $\RR$-divisor on $\PP^{n-1}$. We can associate a \textit{flag-valuation vector:}  
\[
\nu_{Y_{\bullet}}(D)=(\nu_1(D),\ldots ,\nu_{n-1}(D))\in \RR^{n-1}_+.
\]
This is constructed inductively, by first setting $D_0=D$ and $i=1$ and then via
\begin{equation}\label{eq:nobdef}\nu_i(D)\deq{\rm ord}_{Y_i}(D_{i-1}),\text{ and }D_i\deq \big(D_{i-1}-\nu_i(D)Y_{i}\big)|_{Y_i}\ .
\end{equation} 
For example, consider the standard flag $Y_{\bullet}$ on $\PP^{n-1}$, where $Y_i=\{x_1=\ldots=x_i=0\}$ for each $i=1,\ldots ,n-1$. 

For a polynomial $0\neq s\in H^0(\bb P^{n-1},\cal O(d))$ the valuation vector $\nu_{Y_{\bullet}}(s)=(\nu_1,\ldots ,\nu_{n-1})\in\bb Z^{n-1}_+$ has the property that its homogenzied version
\[
\nu_{Y_{\bullet}}(s)^{h} \ \deq \ (\nu_1,\ldots ,\nu_{n-1}, d-\nu_1 -\ldots -\nu_{n-1})\in \ZZ^n_+.
\]
is the exponent of the smallest monomial with respect to the lexicographic order in the homogeneous degree $d$ form $s$, with respect to the variables $x_1,\ldots ,x_n$.
This connection to monomial orders provides access to ideas developed for generic initial ideals by \cite{BS87}. Our next goal is to show the Borel-fixed property for subsets consisting of valuative vectors constructed with respect to generic linear flags on $\PP^{n-1}$.

\begin{proposition}[Borel-fixed property for valuative sets]\label{prop:Borelfixeddiscrete}
	Let $d\geq 1$ and $n\geq 2$ be positive integers and let $V\subseteq H^0(\bb P^{n-1}_{\CC},\cal O(d))$ be a linear subspace. If $Y_{\bullet}\subset\bb P^{n-1}_{\CC}$ is a general linear flag, then $\nu_{Y_{\bullet}}(V\setminus\{{\bf 0}\})$ is independent of $Y_{\bullet}$ and Borel-fixed.
\end{proposition}

\begin{proof}
	First, consider the standard flag $Y_{\bullet}$ on $\PP^{n-1}$, that we have seen just above. Set
	\[
	S\ = \ S_{Y_{\bullet}}\ \deq \ \nu_{Y_{\bullet}}(V\setminus\{{\bf 0}\})^h\ = \ \{{\bf a}_1,\ldots,{\bf a}_q\} \ \subset\bb Z^n_+\ ,
	\] 
	where $|S|=\dim V=q$. Besides the above natural interpretation, homogenization in $\bb Z^n$ allows us to treat the directions ${\bf f}_1,\ldots,{\bf f}_{n-2}$ and the slightly changed ${\bf f}_{n-1}=-{\bf e}_{n-1}+{\bf e}_n$ (used only here) in the same way. 
	
	We will use monomial weights to study $V$ and the line $\bigwedge^qV\subset\bigwedge^qH^0(\bb P^{n-1},\cal O(d))$. There exist such weights that determine the lexicographic order on monomials of degree $d$. For example, if ${\bf a}\in\bb Z^n_+$, the \emph{weight} of ${\bf x}^{\bf a}=x_1^{a_1}\cdots x_n^{a_n}$ with respect to ${\bf w}={\bf w}_d\deq ((d+1)^{n-1},\ldots, d+1,1)\in\bb Z^n_+$ is defined to be ${\bf w}\cdot{\bf a}=\sum_{i=1}^nw_ia_i$. In particular, for ${\bf a},{\bf b}\in \ZZ^{n}_+$ of degree $d$, it holds that ${\bf a}>_{\rm lex}{\bf b}$ if and only if ${\bf w}\cdot {\bf a}>{\bf w}\cdot {\bf b}$. The weight of a simple monomial tensor ${\bf x}^{{\bf b}_1}\wedge\ldots\wedge{\bf x}^{{\bf b}_q}\in \bigwedge^qH^0(\bb P^{n-1},\cal O(d))$ is then ${\bf w}\cdot\sum_{i=1}^q{\bf b}_i$. By \cite[Lemma 1.4]{LM09} there is a basis $\{s_1,\ldots,s_q\}$ of $V$ such that $\nu_{Y_{\bullet}}(s_k)^h={\bf a}_k$ for all $1\leq k\leq q$. Note that ${\bf x}^{{\bf a}_1}\wedge\ldots\wedge{\bf x}^{{\bf a}_q}$ is the only monomial tensor with the smallest possible weight among all the monomial tensors in the expansion of $s_1\wedge\ldots\wedge s_q$, a basis of $\bigwedge^qV$. Moreover, for any basis $\{\sigma_1,\ldots,\sigma_q\}$ of $V$ we recover $S$ from the smallest weight monomial term in $\sigma_1\wedge\ldots\wedge \sigma_q$.
	
	Consider now any complete linear flag $Z_{\bullet}$ on $\PP^{n-1}$. It is of form $g\cdot Y_{\bullet}$, for some $g=(g_{ij})\in{\rm Gl}_n(\bb C)$, where  $Z_i=\{z_1=\ldots=z_i=0\}$ is the $i$-th term and ${\bf x}=g^T\cdot{\bf z}$. As acting with upper triangular matrices preserves $Y_{\bullet}$, we work instead with the Borel subgroup of lower triangular matrices, that acts transitively on all flags. The vector $\nu_{g\cdot Y_{\bullet}}(s)^h$ is the smallest lexicographic term with respect to ${\bf z}$ in $s({\bf x})=s(g^T\cdot{\bf z})$. Thus, the set $g\cdot S\deq S_{g\cdot Y_{\bullet}}=\{{\bf a}'_1,\ldots,{\bf a}'_q\}$ is determined by the unique lowest weight monomial tensor in $s_1(g^T\cdot{\bf z})\wedge\ldots\wedge s_q(g^T\cdot{\bf z})$, which must be ${\bf z}^{{\bf a}'_1}\wedge\ldots\wedge{\bf z}^{{\bf a}'_q}$ by the previous paragraph. So, for our statement it suffices to show that  $g\cdot S$ is constant for general $g\in{\rm Gl}_n(\bb C)$ and is stable under addition by each ${\bf f}_i$. 
	
	For the first statement consider the $g_{ij}$ as parameters. 
	When expanding $s_1(g^T\cdot{\bf z})\wedge\ldots\wedge s_q(g^T\cdot{\bf z})$ as polynomial in ${\bf z}$, the coefficients are polynomial in the $g_{ij}$. There is a unique nonzero one $h(g_{ij})$ corresponding to the unique monomial tensor of smallest weight. The locus $h(g_{ij})\neq 0$ is a nonempty Zariski open subset of ${\rm GL}_n(\bb C)$ where $g\cdot S$ is independent of $g$. Up to a translation in ${\rm GL}_n(\bb C)$, we may assume $g=I_n$ is in this locus, so the standard flag $Y_{\bullet}$ is ``general''.
	
	We focus now on the Borel-fixed conclusion. Let ${\bf a}_i=(a_{i,1},\ldots,a_{i,n})$ for all $i$. Fix $1\leq i\leq n-1$. We want to prove that ${\bf a}_1+{\bf f}_i\in S$ if $a_{i,1}\geq 1$. Let $g[t]=g[t,i]$ be the lower triangular matrix associated to the change of coordinates $x_j=z_j$ for $j\neq i$ and $x_i=z_i+t\cdot z_{i+1}$.  When $t$ is nonzero and general, $h(g[t]_{jr})\neq 0$ since it is nonzero for $t=0$, and thus $g[t]\cdot S=S$ for such $t$. Assume henceforth that $t$ is general.
	For any ${\bf a}=(a_1,\ldots,a_n)\in\bb Z^n_+$ of degree $d$ we compute
	\[{\bf x}^{{\bf a}}=\sum\nolimits_{r=0}^{a_{i}}{{a_{i}}\choose r}t^r\left(\frac{z_{i+1}}{z_i}\right)^r{\bf z}^{{\bf a}}=\sum\nolimits_{r=0}^{a_{i}}{{a_{i}}\choose r}t^r{\bf z}^{{\bf a}+r\cdot {\bf f}_i}\]
	and ${\bf w}\cdot({\bf a}-({\bf a}+r\cdot {\bf f}_i))=r\cdot(w_{i}-w_{i+1})>0$. Thus the exponent of $t$ controls a weight drop. 
	
	Recall $\nu_{Y_{\bullet}}(s_k)^h={\bf a}_k=(a_{k,1},\ldots ,a_{k,n})\in S$ for all $k$. So, assume the opposite that ${\bf a}_1+{\bf f}_i\notin S$ and $a_{1,i}\geq 1$.
	Then ${\bf z}^{{\bf a}_1+{\bf f}_i}\wedge{\bf z}^{{\bf a}_2}\wedge\ldots\wedge{\bf z}^{{\bf a}_q}$ is nonzero and has smaller weight than ${\bf z}^{{\bf a}_1}\wedge{\bf z}^{{\bf a}_2}\wedge\ldots\wedge{\bf z}^{{\bf a}_q}$. Its coefficient in the monomial expansion of $s_1(g[t]^T\cdot{\bf z})\wedge\ldots\wedge s_q(g[t]^T\cdot{\bf z})$ must vanish by our last assumption. Before cancellation, this tensor expands as sum of terms of form 
	\[
	(c_{{\bf b}_1,\ldots,{\bf b}_q,r_1,\ldots,r_q}\cdot t^{r_1+\ldots+r_q})\cdot {\bf z}^{{\bf b}_1+r_1\cdot {\bf f}_i}\wedge\ldots\wedge{\bf z}^{{\bf b}_q+r_q\cdot {\bf f}_i}\ , \]
	where the $c$'s are nonzero constants (at least in characteristic 0), ${\bf b}_k\geq_{\rm lex}{\bf a}_k$ with $|{\bf b}_k|=d$, and $0\leq r_k\leq b_{k,i}$. Varying $t$, the cancellation must also occur when the exponent of $t$ is precisely $1$, which is easier to treat.
	
	We get our contradiction, once we show that ${\bf b}_k={\bf a}_k$ for all $k$ and $r_1=1$, while $r_2=\ldots=r_q=0$ is the only instance when a term as above has exponent $1$ in $t$, so its coefficient cannot be canceled out. To prove the claim, assume first $r_1=1$ while $r_2=\ldots=r_q=0$. So, the weight of ${\bf z}^{{\bf b}_1+r_1\cdot {\bf f}_i}\wedge\ldots\wedge{\bf z}^{{\bf b}_q+r_q\cdot {\bf f}_i}$ is at least the weight of ${\bf z}^{{\bf a}_1+{\bf f}_i}\wedge{\bf z}^{{\bf a}_2}\wedge\ldots\wedge{\bf z}^{{\bf a}_q}$, and the equality holds precisely when ${\bf b}_k={\bf a}_k$ for all $k$. Assume now that $r_k=0$ for all $k$ except some $k>1$. Without loss of generality say $r_2=1$. By comparing weights, we obtain ${\bf b}_k={\bf a}_k$ for all $k$, but the monomial now looks different: ${\bf z}^{{\bf a}_1}\wedge{\bf z}^{{\bf a}_2+{\bf f}_i}\wedge\ldots\wedge{\bf z}^{{\bf a}_q}$. It is in fact different. The nonzero ${\bf z}^{{\bf a}_1+{\bf f}_i}\wedge{\bf z}^{{\bf a}_2}\wedge\ldots\wedge{\bf z}^{{\bf a}_q}$ is only equal up to sign to wedge products of a permutations of its factors, while clearly $\{{\bf a}_1+{\bf f}_i,{\bf a}_2\}\neq\{{\bf a}_1,{\bf a}_2+{\bf f}_i\}$, since ${\bf a}_1\neq {\bf a}_2$.
\end{proof}

\noindent The proof above is similar to that of \cite[Proposition 1]{BS87} with certain differences. We work here with the smallest lexicographic term instead of the largest, and consequently with lower triangular Borel matrices. Furthermore we do not have ``an ideal'' on which the Borel group acts. This is partly responsible for needing $t$ general in the last part of the proof. Another twist is the reduction to terms of degree 1 at the end, where characteristic 0 is needed, unlike in \cite{BS87}.

The ideas from Proposition \ref{prop:Borelfixeddiscrete} can be easily generalized to a more asymptotic setting. On one hand linear subspaces turn to graded subalgebras of the standard graded polynomial ring of $\PP^{n-1}$ and flag valuative sets are asymptotically studied via the theory of Newton--Okounkov bodies in $\RR^{n-1}_+$. For this consider $Y_{\bullet}$ a linear flag on $\bb P^{n-1}$. Fix a rational number $t\in \QQ_+$ and a graded subalgebra 
\begin{equation}\label{eq:subalgebra}
A^t_{\bullet}=\bigoplus\nolimits_{m\in\bb N,\ mt\in\NN}A^t_{m} \ \subseteq \  \bigoplus\nolimits_{mt\in \NN}H^0(\bb P^{n-1},\cal O(mt)) \ .
\end{equation}
Similar to Proposition \ref{prop:Borelfixeddiscrete} define $\Gamma_{m,Y_{\bullet}}\deq\nu_{Y_{\bullet}}(A^t_m\setminus\{0\})\subseteq \ZZ^{n-1}_+$. Note that $\Gamma_{m,Y_{\bullet}}=\varnothing$ when $A_m=0$. 

The \textit{Newton--Okounkov body (NObody)} of $A^t_{\bullet}$ with respect to $Y_{\bullet}$ is defined to be
\[
\nob{Y_{\bullet}}{A^t_{\bullet}}\ \deq\ \text{closure}\,\bigl\{\bigcup\nolimits_{m\in\bb N,\ mt\in\bb N}\frac 1m\Gamma_{m,Y_{\bullet}}\bigr\}\subseteq \ \RR^{n-1}_+
\]
The union under the closure sign is a convex set over $\bb Q$. As an example, when $A_{\bullet}^t=H^0(\bb P^{n-1},\cal O(mt))$, then $\nob{Y_\bullet}{A^t_{\bullet}}=\nob{Y_{\bullet}}{tH}=\simplex\,_t$, where $H\subset\bb P^{n-1}$ is a linear hyperplane. As a consequence, for any subalgebra $A_{\bullet}^t$ its NObody is contained in $\simplex\,_t$, and thus it is compact.

We now show that the Newton--Okounkov body with respect to a generic linear flag of a graded subalgebra of the standard graded polynomial ring is always a Borel-fixed convex body.

\begin{corollary}\label{cor:borelone}
	With $A^t_{\bullet}$ as above, let $Y_{\bullet}$ be a very general linear flag on $\bb P^{n-1}$. Then the convex body $\Delta(A^t_{\bullet})\deq \nob{Y_{\bullet}}{A^t_{\bullet}}\subset\bb R^{n-1}_+$ is independent of $Y_{\bullet}$ and Borel-fixed. In particular, if $w_i(t)$ is the $i$-th width of $\Delta(A^t_{\bullet})$ for each $i=1,\ldots ,n-1$, then 
	\[
	\simplex(w_1(t),\ldots,w_{n-1}(t)) \subseteq \Delta(A^t_{\bullet}) \subseteq \polytope(w_1(t),\ldots,w_{n-1}(t)) 
	\]
and	thus $\frac{1}{(n-1)!}\prod_{i=1}^{n-1}w_i(t)\leq{\rm vol}_{\bb R^{n-1}}(\Delta(A^t_{\bullet}))\leq \prod_{i=1}^{n-1}w_i(t)$.
\end{corollary}

\begin{proof}
By Proposition \ref{prop:Borelfixeddiscrete}, for any $m\in\bb N$ such that $mt\in\bb N$ and $Y_{\bullet}$ is very general, all $\Gamma_{m,Y_{\bullet}}$ are independent of $Y_{\bullet}$ and Borel-fixed as finite sets. Denote the common values by $\Gamma_m$ and fix such a very general flag $Y_{\bullet}$. 
	The $P_m\deq\frac 1m\cdot\textup{convex hull}(\Gamma_m)$ are Borel-fixed as convex sets and contained in $\simplex\,_t$. 
	
	Now, multiplication of sections give rise to inclusions $P_m\subseteq P_{dm}$ for all integers $d\geq 1$. Then $\bigcup_{m\in\bb N,\ mt\in\bb N}P_m$ is a convex set, in particular it is the convex hull of the union. It is also bounded above by $\simplex\,_t$, thus its closure $\nob{Y_{\bullet}}{A^t_{\bullet}}$ is compact and independent of $Y_{\bullet}$. It is furthermore Borel-fixed.
\end{proof}

\begin{remark}	Denote by $r$ the \emph{index} of $A^t_{\bullet}$, the greatest common divisor of all $m$ such that $A_m\neq 0$. 
	For every $1\leq i\leq n-1$ the sequence $w_i(P_{rm})=\frac 1{rm}w_i(\Gamma_{rm})\leq t$ is increasing for the divisibility order on positive integers $m$. Fekete's Lemma implies that $\lim_{m\to\infty}w_i(P_{rm})=\sup_{m\geq 1}w_i(P_{rm})$. As $\Delta(A^t_{\bullet})=\overline{\bigcup\nolimits_{m\geq 1}P_{rm}}$, then 
	\[w_i(\Delta(A^t_{\bullet}))=\lim_{m\to\infty}\frac 1{rm}w_i(\Gamma_{rm})=\sup_{m\geq 1}\frac 1{rm}w_i(\Gamma_{rm})\leq t\ 
	\]
	for each $i=1,\ldots ,n-1$.
\end{remark}

\begin{remark}
	Let $A^t_{\bullet}$ be as above with index $r\geq 1$.  Let $\kappa+1$ be the dimension of the subgroup of $\bb Z^{n}$ generated by all $\{m\}\times\Gamma_m$. We may call $\kappa$ the Iitaka--Kodaira dimension of $A^t_{\bullet}$.  
	
	Denoting $\gamma_m\deq|\Gamma_m|$, then the limit $\lim_{m\to \infty}\frac{\gamma_{rm}}{(rm)^{\kappa}}$ exists and is a positive (finite) real number, because of the Borel property and \cite[Proposition 2.1]{LM09}. In more general situations this was originally proved in \cite[Theorem 2]{KK12} and \cite[Theorem 8.1]{Cut14}.
	
For an arbitrary flag $Y_{\bullet}$, the above tells us that $\Gamma_{m,Y_{\bullet}}$ is contained in a $\kappa$-dimensional subspace of $\bb R^{n-1}$. For a very general flag $Y_{\bullet}$, the Borel-fixed property implies that this subspace is $\{{\bf 0}_{\bb R^{n-1-\kappa}}\}\times\bb R^{\kappa}$.
\end{remark}

\section{Bounds on generic iNObodies}
At any smooth point of a projective variety we can associate a Newton--Okounkov body to a line bundle, using complete linear flags on the tangent space at the point. Alternatively, the same construction can be done using the exceptional divisor of the blow-up of the point and we call these convex sets infinitesimal NObodies (iNObodies). It turns out that the vertical slices of these convex sets are NObodies of subalgebras on the exceptional $\bb P^{n-1}$ with respect to very general flags and are thus Borel-fixed. The largest simplex contained in the generic iNObody is determined by the variation of the slices. 
Using B\' ezout we prove that the lengths of the simplex are the infinitesimal successive minima of the line bundle at the base point.

When the point $x$ is very general, the ability to differentiate sections at $x$, allows us to prove that the entire straightened up body is Borel-fixed. Furthermore the infinitesimal successive minima are its widths and they determine a polytopal upper bound on the body. This bound is sometimes sharp.

\subsection{Background and preliminaries on infinitesimal Newton--Okounkov bodies}
We collect the basic definitions and properties and settle the notation and convention around infinitesimal Newton--Okounkov bodies that we will use in the sequel. Some of the material below works also over algebraically closed fields of any characteristic, but unless otherwise stated we will work over the complex numbers.

Let $X$ be a complex projective variety of dimension $n$, $L$ a big line bundle on $X$ and $x\in X$ a smooth point. Denote by $\pi:\overline{X}\rightarrow X$ the blow-up of $X$ at the point $x$ with $E={\bf P}_{\rm sub}(T_xX)\simeq\PP^{n-1}$ being the exceptional divisor.
Consider an infinitesimal linear flag over $x$
\[
Y_{\bullet}: \overline X=Y_0\supset E=Y_1\supset Y_2\supset\ldots\supset Y_n
\]
so that $Y_i$ is a linear subspace of dimension $n-i$ for any $i=1,\ldots ,n$. We will continue to denote by $Y_{\bullet}$ any truncation of $Y_{\bullet}$ where we leave out the first one or more terms.

With this data in hand for any effective $\RR$-Cartier $\RR$-divisor $D$ we can construct a valuative vector 
\[
\nu_{Y_{\bullet}}(D)=(\nu_1(D),\nu_2(D), \ldots ,\nu_n(D)) \ \in \RR^n_+,
\]
where $\nu_1(D)=\mult_x(D)$ and $(\nu_2(D), \ldots ,\nu_n(D))$ is the vector associated to the effective divisor $D_1\deq \big(\pi^*(D)-\nu_1(D)E\big)|_E$ by using $(\ref{eq:nobdef})$ with respect to the linear flag $Y_{\bullet}$ on $E$. 

In more details, if $s\in H^0(X,L)$ is a nonzero section, and ${\rm ord}_xs=p$, then we can see $s$ as a section of $\pi^*L-pE$ on $\overline X$, and let \emph{the first nonvanishing jet} of $s$ to be the degree $p$ homogeneous polynomial
	\[
	J^p_x(s) \ \deq \ s|_E\in H^0(E,\cal O(p)), 
	\]
which can be seen as an element in $\frak m_x^p/\frak m_x^{p+1}$. If $D\subset X$ is the divisor given by $s=0$, then $D_1$ is the divisor in $E$ of $J^p_x(s)=0$, the \emph{projectivized tangent cone} to $\pi(D)$ at $x$ and $\nu_{Y_{\bullet}}(s)=(p,\nu_{Y_{\bullet}}(D_1))$. 

We introduce the main players of this section. 
\begin{definition}[iNObodies]\label{def:inob}
Under the above assumptions:
\begin{enumerate}
\item The \textit{infinitesimal Newton--Okounkov body (iNObody)} of $L$ with respect to $Y_{\bullet}$ is defined to be
\[
\nob{Y_{\bullet}}{L}\ \deq \ \textup{closure}\bigl\{\frac{1}{m}\nu_{Y_{\bullet}}(D) \ | \ m\geq 1 \textup{ and } D\in |mL|\bigr\}\subseteq \ \RR^{n-1}_+ \ . \footnote{Historically, this convex set is denoted by $\nob{Y_{\bullet}}{\pi^*(L)}$.}
\]
\item When $Y_{\bullet}$ is a very general infinitesimal linear flag over $x$, then $\nob{Y_{\bullet}}{L}\subset\bb R^n$ is independent of $Y_{\bullet}$, by \cite[Proposition 5.3]{LM09}. We denote it $\inob{x}{L}$ and call it the \emph{generic iNObody} of $L$ at $x$, or by some abuse the iNObody of $L$ at $x$.
\smallskip
\item When additionally $x$ is very general, then $\inob{x}{L}\subset\bb R^n_+$ is independent of $x$. We denote it $\inob{X}{L}$ and call it the \emph{generic iNObody} of $L$, or the iNObody of $L$.
\end{enumerate} 
\end{definition}
The basic properties of this convex set that we list below have been studied in \cite{LM09} 
\begin{theorem}[Properties of iNObodies]\label{thm:lm}
Under the above notation, we have the following properties:
\begin{enumerate}
\item {\bf Numerical invariance.} $\nob{Y_{\bullet}}{L}\subset\bb R^n$ is invariant under numerical equivalence. 
\smallskip
\item  {\bf Homogeneity.} $\nob{Y_{\bullet}}{mL}=m\cdotp\nob{Y_{\bullet}}{L}$ for all $m\geq 1$. Consequently we can extend the definition of iNObodies to $\bb Q$-divisor classes in $N^1(X)_{\bb Q}$.
	\smallskip
	
\item {\bf Gluing.} There is a global iNObody (in fact a cone) $\nob{Y_{\bullet}}{X}$ in $N^1(X)_{\bb R}\times\bb R^n$ whose fiber over all big classes $\xi\in N^1(X)_{\bb Q}$ is $\nob{Y_{\bullet}}{\xi}$. A similar picture holds for the generic versions.
\end{enumerate}
\end{theorem}

A consequence of these properties is that 
we can use effective $\bb R$-Cartier $\bb R$-divisors \emph{numerically} (instead of rationally) equivalent to $L$ in the definition of the iNObody. From this and \cite[Proposition 4.4]{Bou14} it makes sense to define $\inob{x}{\xi}$ for the class $\xi\in N^1(X)_{\bb R}$ as the fiber over $\xi$ of the global iNObody $\nob{Y_{\bullet}}{X}$.

Going back to Okounkov, (i)NObodies were introduced to elucidate convexity properties satisfied by intersection numbers and various volume functions of divisors. \cite{LM09} describe two such constructions:
\begin{theorem}[Volumes of divisors and iNObodies]\label{thm:slices}
Let $L$ be a big $\QQ$-Cartier $\bb Q$-divisor on $X$ and $Y_{\bullet}$ an infinitesimal linear flag. Let $L_t=\pi^*L-tE$ for any $t\geq 0$.
\begin{enumerate}
\item {\bf Vertical slice.} If $x\notin \Bplus(L)$ and $L_t$ is big for rational $t> 0$, then the slice $\nob{Y_{\bullet}}{L}_{\nu_1=t}$ is the NObody $\Delta_{Y_{\bullet}}(A_{\bullet})$ of the graded algebra $A_{\bullet}^t$ on $E$, with $A_m^t\deq {\rm Im}\big(H^0(\overline{X},m L_t)\to H^0(E,\sO_E(mt))\big)$ as in $(\ref{eq:subalgebra})$. 
\item {\bf Volumes.} If $m$ is sufficiently divisible, then
\[
{\rm vol}(L)\deq \limsup_{m\to\infty}\frac{h^0(X,mL)}{m^n/n!} = n!\cdot {\rm vol}_{\RR^n}(\nob{Y_{\bullet}}{L}).
\]
 A similar equality holds for the volume of the subalgebra $A_{\bullet}^t$ and the vertical slice $\nob{Y_{\bullet}}{L}_{\nu_1=t}$.
\end{enumerate}
\end{theorem}  
\noindent The statement $(1)$ follows from Corollary \ref{cor:EnotinB+} and \cite[Theorem 4.26]{LM09}. The second statement on the volume of the vertical slices follows from $(1)$ and \cite[Corollary 4.27]{LM09}.

An important consequence of Theorem \ref{thm:slices} is that the Euclidean volume of any iNObody of a fixed class is independent of the choice of the infinitesimal flag. Moreover, the existence of global iNObody above allows to extend the notion of volume for a Cartier divisor to any big numerical class in $N^1(X)_{\RR}$. Finally, the \emph{volume} of the subalgebra $A_{\bullet}^t$ is the \textit{restricted volume} of the class $\xi_t$ along $E$, denoted $\vol_{\overline X|E}(\xi_t)$. It was introduced in \cite{ELMNP09} in any codimension.

Our goal is to understand the shape of generic iNObodies. Their slices are NObodies of algebras defined on $E\simeq \PP^{n-1}$, thus we deduce the following from Corollary \ref{cor:borelone}.

\begin{corollary}\label{cor:slicesborel}
	Let $X$ be a normal complex projective variety of dimension $n$. Let $x\in X$ be a smooth point and $\xi\in{\rm Big}_x(X)$. If $0\leq t\leq\mu(\xi;x)$, then the slice $\inob{x}{\xi}_{\nu_1=t}$ is Borel-fixed in $\bb R^{n-1}_+$.
\end{corollary}

Unlike their vertical slices, the generic iNObodies tend to have ``tilted'' shapes, which prevents them from being Borel-fixed. To fix this, we introduce a \emph{straightening}.
It corresponds to a different type of NObody, constructed using the degree-lexicographic valuation, e.g., see \cite{KK12,WN18}. 

%In this latter statement an important role was the homogenized versions of our valuative vectors defined on $E$. We can do the same thing in this new setting. For this let $x_1,\ldots,x_n$ be local analytic coordinates around $x$. We can see $x_1,\ldots,x_n$ as homogeneous coordinates on $E$ and we can consider the infinitesimal flag where $Y_i$ is given by $x_1=\ldots=x_{i-1}=0$ in $E$ for $i\geq 2$. If $D\geq 0$ is a Cartier divisor on $X$, then $\nu_{Y_{\bullet}}(\pi^*D)=(d,\alpha_1,\ldots,\alpha_{n-1})$ where ${\bf a}=(\alpha_1,\ldots,\alpha_{n-1},d-\sum_{i=1}^{n-1}\alpha_i)$ is the exponent of the lowest degree-lexicographic term in the expansion of a local equation of $D$ around $x$. It also makes sense to consider convex bodies constructed from the exponent vectors ${\bf a}$ as above.

\begin{definition}[Straightened up generic iNObody]
Let $X$ be an irreducible projective variety, $x\in X$ a smooth point, and $\xi\in N^1(X)_{\RR}$ a big class. Then define by
\[
\sinob{x}{\xi}\ = \ S(\inob{x}{\xi})
\]	
where $S:\RR^n\rightarrow \RR^n$ the linear map defined by $S(v_1,\ldots,v_n)=(v_2,\ldots,v_n,v_1-\sum_{i=2}^nv_i)$. 
\end{definition}
\noindent When $\xi=L$ is a line bundle on $X$, then $\sinob{x}{L}$ is the closure of all $\frac 1m{\bf a}$ where ${\bf a}$ appears as the exponent of the lowest degree-lexicographic term in the expansion around $x$ of some section in $|mL|$. 
While the flag version of iNObodies carries geometry in a cleaner fashion, the straightened up version has a nicer shape, as the name suggests.

\begin{example}
	On $\bb P^n$, we have $\inob{x}{\cal O(1)}=S^{-1}(\simplex_1)$, i.e.,  $\sinob{x}{\cal O(1)}=\simplex_1$ for every $x\in X$.
\end{example}
The first result about polytopal upper and lower bounds for our iNObodies was proved in \cite{KL17} for smooth ambient spaces, inspired by the two-dimensional case from \cite{KL18} of the same authors.
\begin{theorem}[First bounds on the iNObody.] \label{thm:firstbounds} Assume that $X$ is normal and $x\in X$ smooth. If $\xi\in \textup{Big}_x(X)$, then for every infinitesimal flag $Y_{\bullet}\subset\overline X$ we have 
\[
\simplex_{\epsilon(||\xi||;x)}\ \subseteq \ \sinob{Y_{\bullet}}{\xi}\ \subseteq\ \simplex_{\mu(\xi;x)}\ \subset \ \RR^n_+.
\]
In particular this holds for the generic infinitesimal body $\sinob{x}{\xi}$. 
\end{theorem}
\noindent The normal case above is implied by the smooth case in \cite{KL17}. Ou invariants are preserved by pullback by birational morphisms that are fiber spaces, and isomorphisms around $x$. The condition $x\not\in{\bf B}_+(\xi)$ is  also preserved when passing to a resolution of singularities that is an isomorphism over $x$, by \cite{BBP13}.

%as follows. Modulo a continuity argument we can assume that $\xi=L$ is a big Cartier divisor on $X$. Previously we defined moving Seshadri constants on normal varieties via asymptotic jet separation. Let $\sigma:X'\to X$ be a resolution of singularities that is an isomorphism around $x$. Denote $x'=\sigma^{-1}\{x\}$. Since $\sigma$ is a fiber space, it follows that $mL$ and $m\sigma^*L$ separate the same jets at $x$ and respectively $x'$. We deduce that $\epsilon(||\sigma^*L||;x')=\epsilon(||L||;x)$. Additionally, $\mu(L;x)=\mu(\sigma^*L;x')$ and $\nob{Y_{\bullet}}{\pi^*L}=\nob{Y'_{\bullet}}{{\pi'}^*\sigma^*L}$ with the natural notation. We also have $x'\not\in{\bf B}_+(\sigma^*L)$ by \cite{BBP13}.
%	Thus the inclusions above extend to the normal case.	

\subsection{Infinitesimal successive minima and generic iNObodies}
For very general infinitesimal linear flags on the tangent space we can refine Theorem \ref{thm:firstbounds} to include all infinitesimal successive minima. 
\begin{theorem}\label{thm:lowerbounds} Let $X$ be a projective variety, $x\in X$ a smooth point, and $\xi\in{\rm Big}_x(X)$. Then
	\begin{enumerate}
		\item $\simplex(s_1(\xi;x),\ldots,s_{n}(\xi;x))\subseteq\sinob{x}{\xi}$. 
		In particular, $\prod_{i=1}^{n}s_i(\xi;x)\leq\vol(\xi)$.
		\smallskip
		
		\item If $X$ is normal, then  $\simplex(\epsilon_1(\xi;x),\ldots,\epsilon_{n}(\xi;x))\subseteq\sinob{x}{\xi}$.
		In particular $\prod_{i=1}^{n}\epsilon_i(\xi;x)\leq\vol(\xi)$.
		\smallskip
		
		\item When $X$ is normal, the inclusions in (1) and (2) are equalities if and only if $\prod_{i=1}^{n}\epsilon_i(\xi;x)=\vol(\xi)$. In particular if $\prod_{i=1}^n\epsilon_i(\xi;x)\geq\vol(\xi)$, then $\inob{x}{\xi}$ is simplicial.
	\end{enumerate}	
\end{theorem}
\noindent 
The lower bound on volume in $(1)$ or $(2)$ is weaker than $\vol(L)\geq\prod_{i=1}^n\epsilon^{\rm loc}_i(L;x)$ from \cite[Lemma 3.2]{AI}. However, our convex geometric conclusion is new and offers more geometric information about the line bundle $L$ and the tangent cones of all the effective divisors in its numerical class. The inclusion in $(2)$ cannot always be improved to $\simplex(\epsilon^{\rm loc}_1(L;x),\ldots,\epsilon^{\rm loc}_n(L;x))\subseteq\sinob{x}{L}$, as seen in Example \ref{ex:hyper3}. We will strengthen $(2)$ and $(3)$ in Theorem \ref{thm:successiveminimaviaslicewidths} and \ref{cor:simplicialcriterion} respectively, completing the proof of Theorem \ref{thm:introreadsuccessive}.
\begin{proof}
The objects $\inob{x}{\cdot}$, $s_i(\cdot;x)$, $\epsilon_i(\cdot;x)$ are homogeneous, continuous, numerical invariants. Thus for $(1)$ and $(2)$ we may assume that $\xi=L$ is a line bundle.

$(1)$. With $S$ as in \eqref{eq:introS}, and $\inob{x}{L}=S^{-1}\sinob{x}{L}$, we want to prove that 
\[
{\bf 0}, \mu(L;x)\cdotp{\bf e}_1=s_{n}(L;x)\cdotp{\bf e}_1, s_1(L;x)\cdotp ({\bf e}_1+{\bf e}_2), \ldots, s_{n-1}(L;x)\cdotp({\bf e}_1+{\bf e}_n) \in \inob{x}{L} .
\]
By Lemma \ref{lem:B+interpretation} and \cite[Theorem 3.1]{KL17} we get ${\bf 0}\in \inob{x}{L}$. By \cite[Proposition 2.5]{KL17} the same can be said about the second vertex on the list, since our flag is very general.

For the other vertices, by homogeneity and continuity it suffices to show that if $0\leq s\leq S_i(L;x)$ is an integer such that  $H^0(\overline X,L_s)\to H^0(\bb P^{n-i},\cal O(s))$ is surjective for general $\bb P^{n-i}\subset E$,
 then $s\cdotp({\bf e}_1+{\bf e}_{i+1})$ is the valuation of a divisor in $|L_s|$ with respect to a very general infinitesimal flag $Y_{\bullet}$ where $Y_{i}$ is a $\bb P^{n-i}$ as above. For local computations, let $x_1,\ldots,x_n$ be an analytic coordinate set  around $x$ so that as homogeneous coordinates on $E$, we have that $Y_{i+1}$ are given by $x_i=0$ in $Y_{i}$ for all $i\geq 1$.

The divisor $D\in|L_s|$ of a section $\sigma\in H^0(X,L)$ satisfies $\nu_{Y_{\bullet}}(D)=s\cdotp({\bf e}_1+{\bf e}_{i+1})$ if and only if $\sigma\in H^0(X,L\otimes\frak m_x^s)$ and its image in $H^0(E,\cal O(s))\simeq\frak m_x^s/\frak m_x^{s+1}$ has the form $F_s(x_1,\ldots,x_n)+a\cdotp x_{i}^s+\frak m_x^{s+1}$ where $F_s\in \cal I_{Y_{i}}=(x_1,\ldots,x_{i-1})$ is homogeneous of degree $s$, and $a\in\bb C^*$ is nonzero. Equivalently, the further restriction on $Y_{i}$ is $a\cdotp x_{i}^s$. Such a section exists by the surjectivity in the definition of $i$-partial jet separation.

$(2)$. Follows from $(1)$ and Theorem \ref{thm:successivejets}.

$(3)$ The two convex bodies from $(2)$ have the same volume, hence the inclusion forces equality. 
\end{proof}

\noindent From Corollary \ref{cor:slicesborel} and Corollary \ref{cor:borelone} the vertical slices $\inob{x}{\xi}_{\nu_1=t}$ are Borel-fixed and thus satisfy strong polytopal bounds in terms of their widths.

%Theorem \ref{thm:slices} has another important consequence. The vertical slices of a generic iNobody is the NObody of a graded algebra on the exceptional divisor $E\simeq \PP^{n-1}$ with respect to a very general flag. By Corollary \ref{cor:borelone} these latter sets are independent of the choice of the flag and Borel-fixed. Consequently, our vertical slices have polytopal upper and lower bounds in terms of their associated widths. 

\begin{definition}
	Let $X$ be a complex projective variety, $x\in X$ a smooth point and $\xi\in{\rm Big}_x(X)$. The functions
	\[
  w_i=w_i(\cdot ,\xi,x) \ : \ (0,\mu(\xi;x)) \ \longrightarrow \ \RR_+ 
  \]
 are called the \emph{slice widths} of $\inob{x}{\xi}$.
\end{definition}
\begin{remark}[Properties of slice widths]\label{rmk:slicewidthsproperties}$ $
	
	\begin{enumerate}
		\item The functions $w_i$ are continuous and concave on $(0,\mu(\xi;x))$ by the convexity of $\inob{x}{\xi}$. 
		\smallskip
		
		\item 	We have $w_i(t)\leq t$ for all $2\leq i\leq n$ and $t\in (0,\mu(\xi;x))$. The equality $w_i(t)=t$ holds for $t\leq s_{i-1}(\xi;x)$, by Theorem \ref{thm:lowerbounds}. The continuity and concavity properties extend each function to $t=0$.
		\smallskip
		
		\item 
		The widths are positive and form an increasing sequence: $w_2(t)\leq w_3(t)\leq\ldots\leq w_n(t)$ on the defining interval, as each slice is a full $n-1$-dimensional convex set, which is also Borel-fixed.
			
		\item 
		The existence of the convex global body on $X$ with respect to a very general linear flag $Y_{\bullet}$ in $E$ implies that $w_i(\cdot,\cdot,x)$ is continuous and superadditive on pairs $(\xi,t)$ where $\xi_t$ is big: each slice $\inob{x}{\xi}_{\nu_1=t}$ is the fiber over $(\xi,t)$ of the projection $\Delta_{Y_{\bullet}}(X)\subset N^1(X)_{\bb R}\times\bb R^n_+\to N^1(X)_{\bb R}\times\bb R_+$ given by $(\zeta,{\bf v})\mapsto(\zeta,\nu_1)$.
	\end{enumerate}
\end{remark}

Next we show that we can improve the inequalities $w_i(t)\leq t$ of Remark \ref{rmk:slicewidthsproperties} when precise knowledge of certain multiplicities along components of ${\bf B}_+(L_t)$ is known.

\begin{lemma}\label{lem:boundsfrommultiplicity}
Under the assumptions and notation from Theorem \ref{thm:slices}, assume $X$ is normal. Suppose for some rational number $t\in (0,\mu(L;x))$ the graded algebra $A_{\bullet}^t$ satisfies the property that $\mult_{Z}(||A_{\bullet}^t||)\geq \delta$ for a subvariety $Z\subset E\subset \overline X$ of dimension $d\leq n-2$.  Then
\[
\inob{x}{L}_{\nu_1=t} \ \subseteq \RR^{n-1}_+\cap \{\nu_2+\ldots+\nu_{d+2}\leq t-\delta\}\ .
\]
In particular, $w_{d+2}(t)\leq t-\delta$. Similar result when  $\mult_{Z}(||L_t||)\geq\delta$.
	\end{lemma}

\begin{proof}
It is enough to show the statement for the rational valuative points in the vertical slice, since they form a dense subset. Let $Y_{\bullet}$ be a very general infinitesimal linear flag such that $Y_{d+1}\cap Z$ is finitely many points and $Y_i\cap Z=\varnothing$ for $i>d+1$.
		
		Let $D$ be an effective $\bb Q$-Cartier divisor, numerically equivalent with $L$, such that $\nu_{Y_{\bullet}}(D)=(t,\nu_2,\ldots,\nu_n)$. Let $\overline{D}$ be the strict transform of $D$ on $\overline X$, numerically equivalent to $\xi_t$. Its restriction $D_1=\overline{D}|_E$ on $E\simeq\bb P^{n-1}$ is a hypersurface of degree $t$, and the valuation vector $\nu_{Y_{\bullet}}(D_1)=(\nu_2,\ldots,\nu_n)$ with respect to the truncated flag $Y_{\bullet}$ is constructed as in $(\ref{eq:nobdef})$. The condition on $A^t_{\bullet}$ implies $\mult_ZD_1\geq\delta$. 
The result follows by checking inductively that for $i\leq d$ and $D_{i}$ on $Y_{i}$ as in \eqref{eq:nobdef} we have
\begin{align*}
t-\nu_2-\ldots-\nu_{i+2}=\deg_{Y_{i+1}} \big(D_{i+1}-\nu_{i+2}Y_{i+2}\big)&\geq 
\mult_{Z\cap Y_{i+1}}\big(D_{i+1}-\nu_{i+2}Y_{i+2})\\
& =\mult_{Z\cap Y_{i+1}}(D_{i+1})\geq\mult_{Z\cap Y_i} \big(D_i-\nu_{i+1}(D)Y_{i+1}\big)\geq\delta\ ,
\end{align*}
The first inequality is B\' ezout's Theorem in $\bb P^{n-i-1}\simeq Y_{i+1}$.
For the second equality note that $Y_{i+1}$ does not contain any of the components of $Y_i\cap Z$ by our very general choices and thus $D_{i+1}$ and $D_{i}$ have equal multiplicity at the generic point(s) of $Y_{i+1}\cap Z$ for all $i\leq d$. The final inequality is to due to the fact that multiplicity does not decrease by restriction.
\end{proof}

We are ready to explain how to read the infinitesimal successive minima from the generic iNObodies, completing the proof of Theorem \ref{thm:mintroslicebounds}. 
\begin{theorem}\label{thm:successiveminimaviaslicewidths}
	Let $X$ be a normal complex projective variety of dimension $n$, let $x\in X$ be a smooth point, and $\xi\in{\rm Big}_x(X)$. Then for every $i=1,\ldots ,n-1$ we have
	\[
	\epsilon_{i}(\xi;x)=\max\{t\in[0,\mu(\xi;x)]\ \mid\ w_{i+1}(t)=t\}\ ,
	\]
or equivalently, $\epsilon_i\deq \epsilon_i(\xi;x)=\max\{t>0 \mid t\cdot({\bf e}_1+{\bf e}_{i+1})\in\inob{x}{\xi}\}$. Consequently, $\simplex(t_1,\ldots,t_{n})\subseteq\sinob{x}{\xi}$ if and only if $t_i\leq\epsilon_i$ for all $i$. In particular, $\simplex(\epsilon_1,\ldots,\epsilon_{n})$ is the largest simplex contained in $\sinob{x}{\xi}$. 
\end{theorem}

\begin{proof}
As $\epsilon_i=s_i\deq s_i(\xi;x)$ by Theorem \ref{thm:successivejets}, we first prove that $w_{i+1}(s_i)=s_i$. By Remark \ref{rmk:slicewidthsproperties}  it suffices to show the reverse inequality. So, let $S$ as in \eqref{eq:introS}. Theorem \ref{thm:lowerbounds} says that $S^{-1}\simplex(s_0,\ldots,s_{n-1})\subseteq\inob{x}{\xi}$.  As $s_i\cdotp({\bf e}_1+{\bf e}_{i+1})$ is in the simplex on the left, then $s_i\cdotp({\bf e}_1+{\bf e}_{i+1})\in\inob{x}{\xi}_{s_i}$, and so $w_{i+1}({s_i})\geq s_i$.

	It remains to show that $w_{i+1}(t)<t$ for $t\in(\epsilon_i,\mu)$. 
	Let $A\in N^1(X)_{\bb R}$ be an ample class such that $\xi+A$ is a $\bb Q$-class. We have $\inob{x}{\xi}\subset\inob{x}{\xi}+\inob{x}{A}\subseteq\inob{x}{\xi+A}$, which implies $w_{i+1}(t;\xi,x)\leq w_{i+1}(t;\xi+A,x)$. Furthermore, $\epsilon_i(\xi;x)<\epsilon_i(\xi;x)+\epsilon_i(A;x)\leq\epsilon_i(\xi+A;x)$, where the latter is due to Lemma \ref{lem:infsuccessiveproperties}. Since $\epsilon_i(\cdot;x)$ is continuous on ${\rm Big}_x(X)$,  then choosing $A$ sufficiently small we may assume $\epsilon_i(\xi;x)\leq\epsilon_i(\xi+A;x)<t$. So, by showing that $w_{i+1}(t;\xi+A,x)<t$, we also get $w_{i+1}(\xi;x)<t$. Then up to replacing $\xi$ by $\xi+A$, we may assume that $\xi$ is a $\bb Q$-class.  Let $L$ be a $\bb Q$-Cartier $\bb Q$-divisor that represents it. 
	
	Assume for a contradiction that $w_{i+1}(t)=t$ for a fixed $t\in(\epsilon_i,\mu)$. By concavity $w_{i+1}(t')=t'$ for all $\epsilon_i(L;x)<t'<t$. Since $t'>\epsilon_i(L;x)$, by definition we have $\textup{dim}(\Bplus(L_{t'})\cap E)\geq i-1$. Let $Z$ be a component of ${\bf B}_+(L_{t'})\cap E$ with $\dim Z\geq i-1$. For all $t''\in (t',t)$ with $t''$ sufficiently close to $t'$ we have that $Z$ is contained of ${\bf B}_-(L_{t''})\cap E={\bf B}_+(L_{t'})\cap E$, with the latter equality due to Lemma \ref{lem:B-ray}.$(3)$. Thus $Z\subseteq {\bf B}_-(L_{t''})$ and assuming $t''$ is rational, this implies that $\delta\deq\mult_Z(||L_{t''}||)>0$ by \cite[Proposition 2.8]{ELMNP06}. Applying Lemma \ref{lem:boundsfrommultiplicity} to this yields $w_{i+1}(t'')\leq t''-\delta<t''$, getting us the final contradiction.
	
For the second conclusion use Theorem \ref{thm:lowerbounds} to deduce that $\epsilon_i\cdot({\bf e}_1+{\bf e}_{i+1})\in\inob{x}{\xi}$. If $t>\epsilon_i$, then $w_{i+1}(t)<t$ by the first part of the proof. In particular $t\cdot ({\bf e}_1+{\bf e}_{i+1})$ cannot be contained in $\inob{x}{L}$. Applying $S$ the second conclusion implies easily the last two and this finishes the proof.
\end{proof}

As a consequence of Theorem \ref{thm:successiveminimaviaslicewidths}, we characterize when the generic iNObody $\inob{x}{L}$ is simplicial, improving the sufficient condition from Theorem \ref{thm:lowerbounds}.

\begin{corollary}[Simplicial characterization of generic iNObodies]\label{cor:simplicialcriterion}
	Under the assumptions and notation of Theorem \ref{thm:successiveminimaviaslicewidths}, the following statements are equivalent:
	\begin{enumerate}[(1)]
		\item $\inob{x}{\xi}$ is simplicial;
		\smallskip
		
		\item $\sinob{x}{\xi}=\simplex(\epsilon_1(\xi;x),\ldots,\epsilon_{n}(\xi;x))$;
		\smallskip
		
		\item $\vol(\xi)=\prod_{i=1}^{n}\epsilon_i(\xi;x)$.
	\end{enumerate} 
In this case we also necessarily have that $\epsilon_i(\xi;x)=\epsilon^{\rm loc}_{i}(\xi;x)$ for all $1\leq i\leq n$, connecting the local and infinitesimal successive minima.
\end{corollary}
\noindent We will see several examples of this phenomenon in Section $6$.
\begin{proof}
The equivalence of $(2)$ and $(3)$ is proved in Theorem \ref{thm:lowerbounds}. Moreover, we have the inclusions 
\[
\simplex(\epsilon_1,\ldots,\epsilon_{n})\subseteq\sinob{x}{\xi}\subseteq\simplex\,_{\mu}
\]
from Theorem \ref{thm:lowerbounds} and Theorem \ref{thm:firstbounds} respectively. They imply that $\sinob{x}{\xi}$ is simplicial if and only if it is of form $\simplex(t_1,\ldots,t_{n})$ for some $t_i\in[\epsilon_i,\mu]$. 
The equivalence of $(1)$ and $(2)$ follows from Theorem \ref{thm:successiveminimaviaslicewidths}.
	
The last statement is a consequence of the inequality $\vol(\xi)\geq \prod_{i=1}^n\epsilon^{\rm loc}_i(\xi;x)$ of \cite[Lemma 3.2]{AI} and from $\epsilon^{\rm loc}_i(\xi;x)\geq\epsilon_{i}$ for all $1\leq i\leq n$ in Lemma \ref{lem:successiveproperties}.$(3)$.
\end{proof}

\subsection{Bounds on generic iNObodies at very general points}
Recall that a property of points of $X$ is \emph{very general} if it holds away from a countable union of proper closed subvarieties.

An important feature when working at very general points is the ability to differentiate sections. Using it we extend the Borel-fixed property from the slices of the tilted body to the entire straightened up body. As result we find a global polytopal upper bound on the generic straightened up iNObody in terms of infinitesimal successive minima.
In order to do so, we first discuss differentiation techniques at very general points, and more precisely the effect on first nonvanishing jets of sections.

Let $X$ be a complex projective variety of dimension $n$, and $L$ a big divisor on $X$. Let ${\rm pr}_1, {\rm pr}_2:X\times X\to X$  be the two projections. 
Let $x\in X$ be a very general point such that the natural morphisms 
\[
{\rm pr}_{1*}(\cal I_{\Delta}^p(q\cdot {\rm pr}_2^*L))(x)\to H^0(X,\frak m_x^p(qL))\]
from the fiber at $x$ of the sheaves on the left are isomorphisms for all integers $p$ and $q$. This is where very generality is used in arguments concerning differentiation of sections.

Fix a nonzero section $s\in |qL|$ with ${\rm ord}_x(s)=p$. Let $U\subseteq X$ be an affine Zariski open neighborhood of $x$. Then $s$ admits a (nonunique) lift $\widetilde s\in H^0(U\otimes X,\cal I_{\Delta}^p(q\cdot{\rm pr}_2^*L))$.
For $u\in U$, denote $X_u={\rm pr}_1^{-1}\{u\}=\{u\}\times X$, and $s_u=\widetilde s|_{X_u}$. Since $s_x=s$ has order $p$ at $x$ by assumption and ${\rm ord}_{\Delta}(\widetilde s)\geq p$, we deduce that the latter inequality is an equality and furthermore ${\rm ord}_u(s_u)=p$ for general $u\in U$.

Any differential operator $D$ on $U$ acts on $H^0(U\times X,q\cdot{\rm pr}_2^*L)$. The idea is that the transition functions of the line bundle ${\rm pr}_2^*L^{\otimes q}$ are constant with respect to the first $U$ factor.
For details see \cite[Section 2]{EKL95} (or \cite[Section 5.2.B]{Laz04}).
%If $D^m$ is a differential operator on $U$ with order at most $m$, then $D^m\widetilde s\in H^0(U\times X,\cal I_{\Delta}^{p-m}(q\cdot {\rm pr}_2^*L))$.
To be more precise, let $(u_1,\ldots,u_n)$ be local analytic coordinates around a fixed very general $u_0$ on $U$, and denote by $(x_1,\ldots,x_n)$ the same local coordinates thought of in the second factor $X$ in $U\times X$. Locally analytically around $(u_0,u_0)$ we have that $\cal I_{\Delta}=(x_1-u_1,\ldots,x_n-u_n)$, and
	\[
\widetilde s\ =\ \sum_{|b|=p}(x-u)^bP_b(u)+R(x,u) \ ,
\] 
where $P_b(u)$ are holomorphic, and $R(x,u)\in\cal I_{\Delta}^{p+1}$.  We denoted here by $(x-u)^b\deq (x_1-u_1)^{b_1}\cdots(x_n-u_n)^{b_n}$ and $|b|\deq\sum_{i=1}^nb_i$. Note that the locus where $P_b(u)$ are not simultaneously zero is the locus where ${\rm ord}_us_u=p$, thus it is nonempty and Zariski open. Consequently, the $p$-th jet of $s_{u_0}$ at $x=u_0$,
\[
J_{u_0}^p(s_{u_0})\ = \ \sum_{|b|=p}P_b(u_0)(x-u_0)^b\in H^0(E_{u_0},\cal O(p)) 
\] 
is the first nonvanishing jet of $s_{u_0}$. Here the $x_i-u_{0i}$ are seen as homogeneous coordinates in the exceptional divisor $E_{u_0}$ resulting from blowing-up $u_0$.
On $U$ we consider an arbitrary differential operator $D$ of order at most $m$. In our local analytic coordinates around $u_0$ we can write it as
\[
D\ = \ \sum_{|v|=m}c_{v}(u)\partial_u^v+B(u),
\] 
where $c_{v}(u)$ are holomorphic, $\partial_u^v\deq \frac{\partial^m}{\partial u_1^{v_1}\cdots\partial u_n^{v_n}}$, and $B(u)$ has order at most $m-1$. Then
\[
D\widetilde s\ = \ \sum\nolimits_{|b|=p,|v|=m}P_b(u)c_{v}(u)\partial_u^v(x-u)^b+Q(x,u),
\]
where $Q(x,u)\in\cal I_{\Delta}^{p-m+1}$. In particular we recover that $D\widetilde s\in H^0(U\times X,\cal I^{p-m}_{\Delta}(q\cdot{\rm pr}_2^*L))$. The corresponding $p-m$ jet of $D\widetilde s|_{X_{u_0}}$ at $x=u_0\in U$ is 
\begin{align*} \sum\nolimits_{|b|=p,|v|=m}P_b(u_0)c_{v}(u_0)\partial_u^v|_{u=u_0}(x-u)^b&=
	(-1)^m\sum\nolimits_{|b|=p,|v|=m}P_b(u_0)c_{v}(u_0)\partial_x^v|_{x=u_0}(x-u_0)^b\\
	&=(-1)^m(\sum\nolimits_{|v|=m}c_{v}(u_0)\partial_x^v|_{x=u_0})\sum\nolimits_{|b|=p}P_b(u_0)(x-u_0)^b\\
	&=(-1)^m(\sum\nolimits_{|v|=m}c_{v}(u_0)\partial_x^v|_{x=u_0})J^p_{u_0}(s_{u_0})
\end{align*} 
On the right, $\partial_x^{v}|_{x=u_0}:H^0(E_{u_0},\cal O(p))\to H^0(E_{u_0},\cal O(p-m))$ are seen as linear transformations. The index $x=u_0$ marks that we work over $u_0$.

Since the lift $\widetilde s$ of $s$ is non-unique, a priori we do not have a well-defined action of the operator on sections of $qL$. Our calculation above shows that the situation improves when we consider the first nonvanishing jet of $s$. More precisely we proved the following:

\begin{proposition}[Differential operators acting on jets]\label{prop:differentiatejets}
	With notation as above, let $x_0\in X$ be a very general point and $D$ a differential operator of order at most $m$ defined around $x_0$. If $s\in |qL|$ is nonzero and ${\rm ord}_{x_0}(s)=p$, then for every lift $\widetilde s$ of $s$ to a section of $\cal I_{\Delta}^p(q\cdot{\rm pr}_2^*L)$ defined around $\{x_0\}\times X$, we have \[J_{x_0}^{p-m}((D\widetilde s)|_{\{x_0\}\times X})=(-1)^mD_{m,x_0}(J_{x_0}^ps)\ ,\]
	where $D_{m,x_0}=\sum_{|v|=m}c_v(x_0)\partial^v_x|_{x=x_0}$ denotes the homogeneous order $m$ part of $D$ at $x_0$. If this $p-m$ jet is nonzero, it is the first nonvanishing jet of $(D\widetilde s)|_{X_{x_0}}$. In particular $(D\widetilde s)|_{X_{x_0}}$ is a section of $|qL|$ with order exactly $p-m$ at $x_0$.
\end{proposition}

The idea of differentiation at very general points with control over how the multiplicity drops was used in the literature also for applications such as the lemma below.

\begin{lemma}\label{lem:AIverygeneral}
	Let $X$ be a complex projective variety and $L$ be a big divisor on $X$. Let $x\in X$ be a very general point. Let $Z$ be a non-trivial component of $\pi\big(\Bplus(L_t)\big)$ through $x$ for some fixed $t>0$. If $s\in H^0(X,L^{\otimes q})$ satisfies ${\rm ord}_x(s)\geq p$, then ${\rm ord}_Z(s)\geq p-qt$.
\end{lemma}

\noindent Through Corollary \ref{cor:B+B}, this statement is proved in \cite[Lemma 2.18]{AI} and was outlined previously in \cite[Lemma 1.3]{Nak05}. On the blow up $\overline X$ Lemma \ref{lem:AIverygeneral} is saying that if $s\in|q\pi^*L-pE|$, then ${\rm ord}_{\overline Z}s\geq p-qt$, for any component $\overline Z$ of ${\bf B}_+(L_{t})$ meeting $E$ properly. Equivalently, ${\rm mult}_{\overline Z}|L_s|_{\bb Q}\geq s-t$ for all rational $s\geq t$. 

\begin{remark}
Applying Proposition \ref{prop:differentiatejets}, one can show an infinitesimal version of Lemma \ref{lem:AIverygeneral} on the behaviour of the multiplicity function along subvarieties of $E$ and when restricting sections to $E$. For abelian varieties a similar statement appeared in \cite[Proposition 3.3]{Loz18} with a different proof. 
\end{remark}

We turn our attention to the infinitesimal picture at very general points. 
%Using the differentiation techniques developed above for jets we study the shape of the straightened up generic iNObody at a very general point.

\begin{theorem}\label{thm:upperbound}
	Let $X$ be a normal projective complex variety of dimension $n$, and $\xi\in N^1(X)_{\bb R}$ a big class. Let $x\in X$ be a very general point and denote by $\epsilon_i\deq\epsilon_i(\xi)$, for $\inob{X}{\xi}\deq\inob{x}{\xi}$ and $\sinob{X}{\xi}\deq\sinob{x}{\xi}$. Then
	
	\begin{enumerate}
		\item $\sinob{X}{\xi}$ is Borel-fixed.
		\smallskip
		
		\item $\simplex(\epsilon_1,\ldots,\epsilon_{n})\subseteq\sinob{X}{\xi}\subseteq\polytope(\epsilon_1,\ldots,\epsilon_{n})\subseteq\prod_{i=1}^n[0,\epsilon_i]$.
		\smallskip
		
		\item $\prod_{i=1}^n[0,\epsilon_i]$ is the smallest box that contains $\sinob{X}{\xi}$.
		\smallskip
		
		\item $\vol(\xi)\leq n!\cdot \prod_{i=1}^{n}\epsilon_i$.
	\end{enumerate}
\end{theorem}

\noindent Recall that infinitesimal successive minima are constant due to Lemma \ref{lem:tauverygeneralconstant} and the same can be said about our generic iNobodies by \cite[Theorem 5.1]{LM09}. 
Ambro--Ito \cite[Lemma 3.4 and Theorem 3.6]{AI} prove a slightly weaker variant of (4). In  Example \ref{ex:verygeneralpointnecessary} we see that asking $x$ to be very general is necessary. \cite[Remark 3.7]{AI} discuss how $(4)$ is strict in dimension $n\geq 2$, but asymptotically sharp in special examples. Example \ref{ex:upperboundsharp} shows that the second inclusion in $(2)$ is sometimes sharp even when the first is not.

\begin{proof}  Parts $(2)$, $(3)$ and $(4)$ follow from $(1)$ together with Proposition \ref{prop:borelfixedshapesproperties}, Lemma \ref{lem:boundingpolytopes}, Theorem \ref{thm:successiveminimaviaslicewidths} and Theorem \ref{thm:slices}. It remains to show $(1)$. 

By the usual techniques, we may assume that $\xi$ is the class of a Cartier divisor $L$. Consider a point $(\alpha_1,\ldots,\alpha_n)\in\sinob{x}{L}$, or equivalently after tilting to $(\alpha_1+\ldots+\alpha_n,\alpha_1,\ldots,\alpha_{n-1})\in\inob{x}{L}$.
	We want to prove that $(\alpha_1,\ldots,\alpha_{i-1},0,\alpha_i+\alpha_{i+1},\alpha_{i+2},\ldots,\alpha_n)\in\sinob{x}{L}$ for $i\leq n-1$, and $(\alpha_1,\ldots,\alpha_{n-1},0)\in\sinob{x}{L}$. The first is equivalent to $(\alpha_1+\ldots+\alpha_n,\alpha_2,\ldots,\alpha_{i-1},0,\alpha_i+\alpha_{i+1},\alpha_{i+2},\ldots,\alpha_{n-1})\in\inob{x}{\xi}$. This is in the same vertical slice as $(\alpha_1+\ldots+\alpha_n,\alpha_2,\ldots,\alpha_{n-1})$. In this case the conclusion is a consequence of the Borel-fixed shape of the vertical slices from Corollary \ref{cor:borelone} and Theorem \ref{thm:slices} and does not need the assumption that $x$ is very general.
	It remains to prove that $(\alpha_1,\ldots,\alpha_{n-1},0)\in\sinob{x}{L}$.
	
	It is sufficient to assume that $(\alpha_1,\ldots,\alpha_n)$ is the deglex valuation of a section $s\in|L|$ with $\alpha_n>0$. We seek to prove that  $(\alpha_1,\ldots,\alpha_{n-1},\alpha_n-1)\in\sinob{x}{L}$.
	Let $(x_1,\ldots,x_n)$ be a local analytic coordinate system around $x$. For readability, denote the very general point $x$ by $x_0$.
	There exists a regular differential operator $D$ of order at most $1$ defined Zariski locally around $x_0$ such that $D_{1,x_0}=\frac{\partial}{\partial x_n}\big|_{x=x_0}$ in the notation of Proposition \ref{prop:differentiatejets}. So, differentiation by $D$ it produces a section $\widehat s\in|L|$ with valuation $(\alpha_1,\ldots,\alpha_{n-1},\alpha_n-1)$ as needed.
\end{proof}

\begin{remark}\label{rmk:tiltedshapes} We often found it more convenient to state our results for $\sinob{x}{\xi}$ since it is closer to standard simplices. The more geometric ``tilted'' version $\inob{x}{\xi}=S^{-1}\sinob{x}{\xi}$ is bounded:
	\begin{itemize}
		\item Below in the normal case by \[\inverted(\epsilon_1,\ldots,\epsilon_n)\deq S^{-1}\simplex(\epsilon_1,\ldots,\epsilon_n)\ .\] This is the simplex with vertices ${\bf 0}$, $\epsilon_n\cdot{\bf e}_1$, and $\epsilon_i\cdot({\bf e}_1+{\bf e}_{i+1})$ for $1\leq i\leq n-1$.
		When $\epsilon_i=t$ for all $i$, this is the \emph{inverted simplex} $\Delta_t^{-1}$ of \cite{KL17}, whose shape could be represented as $\invertsimplex_t$.
		\smallskip
		
		\item For a very general point $x\in X$ define 
		\[\invertpoly(\epsilon_1,\ldots,\epsilon_n)\deq S^{-1}\polytope(\epsilon_1,\ldots,\epsilon_n)\ ,\]
	described by the equations $\nu_2+\ldots+\nu_{n}\leq\nu_1\leq\epsilon_n$ and $\nu_2+\ldots+\nu_{i+1}\leq\epsilon_i$ and $\nu_i\geq 0$ for all $1\leq i\leq n-1$.
	\end{itemize}
The symbols $\inverted,\invertsimplex,\invertpoly$ are again representative for the 2-dimensional pictures.
The Borel-fixed property is another reason why we prefer to work with $\sinob{x}{\xi}$. Under $S^{-1}$ this property changes. In particular, if $x$ is very general, then $\inob{x}{\xi}$ satisfies the following
\[(\nu_1,\ldots,\nu_n)\in\inob{x}{\xi}\Longrightarrow\begin{cases}(\nu_1,\ldots,\nu_{i-1},0,\nu_i+\nu_{i+1},\nu_{i+2},\ldots,\nu_n)\in \inob{x}{\xi}&\text{, if }2\leq i\leq n-1\\ 
	(\nu_1,\ldots,\nu_{n-1},0)\in\inob{x}{\xi}&\text{ , if }i=n\\
(\nu_2+\ldots+\nu_n,\nu_2,\ldots,\nu_n)\in\inob{x}{\xi}&\text{, corresponding to }i=1 \end{cases}\ .
	\]
For the coordinates $\nu_2,\ldots,\nu_n$, we see Borel-fixed conditions, but in the first coordinate we see a different condition. In retrospect, this explains why the vertical slices of $\inob{x}{\xi}$ are Borel-fixed, but $\inob{x}{\xi}$ itself is not. Note that Theorem \ref{thm:mintroslicebounds} held for arbitrary smooth points outside ${\bf B}_+(\xi)$, not necessarily very general.
\end{remark}

\begin{corollary}[Convex geometric perspectives on infinitesimal successive minima at very general points] Under the assumptions and notation of Theorem \ref{thm:upperbound}, for each $i=1,\ldots ,n$ we can read the corresponding infinitesimal width $\epsilon_i=\epsilon_i(\xi;x)$ from $\inob{X}{\xi}$ in any of the following equivalent ways:
	\begin{enumerate}[(1)]
	\item Along ${\bf e}_1+{\bf e}_{i+1}$: Largest $t>0$ such that $t\cdot({\bf e}_1+{\bf e}_{i+1})\in\inob{X}{\xi}$;
\smallskip

		\item Largest $\nu_{i+1}$ coordinate: $i+1$-st global width of $\inob{X}{\xi}$, i.e., largest $w_{i+1}(t;\xi,x)$.
\smallskip

	\item Variational: Largest $t>0$ such that $w_{i+1}(t;\xi,x)=t$;
	\end{enumerate}
\end{corollary}

\noindent Example \ref{ex:verygeneralpointnecessary} shows that the interpretation $(2)$ which is the new input coming from Theorem \ref{thm:upperbound} may fail when $x$ is a special point, thus the very general assumption is not just an artifact of our proof.

\section{Examples}\label{s:examples}
In this section we discuss the full description of generic iNObodies in certain examples. On one hand, they will show the variation of these convex sets when passing from an arbitrary base point to a very general one. On the other hand, we see the multitude of convex shapes that appear as generic iNobodies.
\subsection{Simplicial examples}

\begin{lemma}\label{lem:simplicial}
Let $X$ be normal projective variety of dimension $n$, and let $x\in X$ be a smooth point. Let $\xi\in{\rm Big}_x(X)$. If
	$\epsilon_1(\xi;x)=\ldots=\epsilon_{n}(\xi;x)$,
	then $\sinob{x}{\xi}=\simplex\,_{\tau}$, where $\tau=\sqrt[n]{\vol(\xi)}$.
\end{lemma}
\begin{proof}
We have $\epsilon_1(\xi;x)=\epsilon(||\xi||;x)\leq\sqrt[n]{\vol(\xi)}\leq \mu(\xi;x)=\epsilon_{n}(\xi;x)$, whence $\vol(\xi)=\prod_{i=1}^{n}\epsilon_i(\xi;x)$. The conclusion follows from Theorem \ref{thm:lowerbounds}.$(3)$.
\end{proof}

\begin{example}
	Let $C$ be a smooth projective curve, let $n\geq 1$ and $X\deq C_n={\rm Sym}^nC$. For arbitrary $c_1\in C$, let $L_n$ be the divisor $c_1+C_{n-1}\subset C_n$. For every $x\in X$ we have
	\[\sinob{x}{L_n}=\simplex\,_1\ \subseteq \ \RR_+^n \ .\]
	%This applies in particular for $C=\bb P^1$ when $C_n=\bb P^n$ and $L_n=\cal O(1)$.
\begin{proof} The divisor $L_n$ is ample. Furthermore ${L_n}|_{L_{n}}\equiv L_{n-1}$ and $(L_n^n)=1$ by \cite[Theorem 2]{Sch63}.
	By Lemma \ref{lem:simplicial} it is sufficient to prove that $\epsilon(L_n;x)\geq 1$ for the classical Seshadri constant.
	Let $T$ be an irreducible curve through $x=c_1+\ldots+c_n\in C_n=X$.
	If $T$ is not contained in $L_n=c_1+C_{n-1}$, then $(L_n\cdot T)\geq\mult_xL_n\cdot\mult_xT=\mult_xT$, thus $\frac{L_n\cdot T}{\mult_xT}\geq 1$. If $T$ is contained in $L_n\simeq C_{n-1}$, then 
	\[
	\frac{L_n\cdot T}{\mult_x T}\geq\epsilon(L|_{L_n};x)=\epsilon(L_{n-1};c_2+\ldots+c_n) \ .
	\] The latter is 1 by induction. The base case $n=1$ is trivial.
\end{proof} 
\end{example}

\begin{example}
	Let $X$ be a smooth complex projective surface of Picard rank 1 and let $x\in X$ be a point. If $\xi$ is an ample $\bb R$-divisor class on $X$, then \[\sinob{x}{\xi}=\simplex(\epsilon(\xi;x),\mu(\xi;x))\ \subseteq \ \RR^2_+ \ .\]
It is sufficient to check that $\epsilon(\xi;x)\cdot\mu(\xi;x)=(\xi^2)$. This follows from Kleiman duality. \qed
%between the two-dimensional cones ${\rm Eff}(\overline X)=\bigl\langle \xi_{\mu(\xi;x)},E\bigr\rangle$ and ${\rm Nef}(\overline X)=\bigl\langle \xi_{\epsilon(\xi;x)},\pi^*\xi\bigr\rangle$.\qed
\end{example}

\begin{example}
	Let $X\subset\bb P^{n+1}$ be a smooth quadric hypersurface and let $L=\cal O_{\bb P^{n+1}}(1)|_X$. Then  \[\sinob{x}{L}=\simplex(1,1,\ldots,1,2) \ ,\]
for all $x\in X$. In particular, this applies to the Grassmann variety ${\rm Gr}(2,4)\subset\bb P^5$. 
	\begin{proof}Let $x\in X$. 
		We have $\vol(L)=\deg X=2$. Since $L$ is very ample, $\epsilon(L;x)\geq 1$ for all $x$.
		The intersection $Q_x\deq{\bf T}_xX\cap X$ in the projective tangent space ${\bf T}_xX\simeq\bb P^n$ is a quadric cone with vertex at $x$, and a divisor in $|L|$ with multiplicity 2 at $x$ in $X$. Thus $\mu(L;x)\geq 2$. 
		Conclude by Theorem \ref{thm:lowerbounds}.(3).
	\end{proof} 
\end{example}

Next we give a simplicial example in arbitrary dimension where all $\epsilon_i(L;x)$ are distinct.

\begin{example}\label{ex:productcurves}
	Let $X=C_1\times\ldots\times C_n$ be a product of smooth projective curves polarized by $L=D_1\boxtimes\ldots\boxtimes D_n$ where the $D_i$ have degree $1$ on $C_i$. Then for every $x\in X$ we have
	\[\sinob{x}{L}=\simplex(1,2,\ldots,n)\ .\]
It is somewhat surprising that the genera of the curves play no part in the computation of the iNObody. But the entire picture can be derived from the blow-up of $\PP^{n-1}$ at $n$ linear general points. This is a toric variety and  we use this idea in \cite{FL25b} to compute the same example. This toric approach explains the geometric interpretation of the family of graded linear series $A_{\bullet}^t$ from Theorem \ref{thm:slices}, computing the vertical slices of the generic iNobody. This is a very different approach than the one explained here, which uses the connection between infinitesimal successive minima and embedded simplices in our convex sets.

	\begin{proof}[Proof of Example \ref{ex:productcurves}]
	 We claim that $\epsilon_{i}(L;x)=i$ for all $i$. Assuming this, since $\vol(L)=n!$ would equal $\prod_{i=1}^{n}\epsilon_i(L;x)$, we conclude by Theorem \ref{thm:lowerbounds}.$(3)$.
	Let $[n]\deq\{1,\ldots,n\}$. For a subset $I\subset[n]$, denote by $X_I\simeq\bigtimes_{i\in I}C_i$ the fiber through $x$ of ${\rm pr}:X\to\bigtimes_{j\not\in I}C_j$. Denote by $\overline X_I$ their strict transforms in $\overline X$.
	
	For the claim, it is sufficient to prove that if $t$ is rational and $m$ an integer with $0\leq m<t<m+1$, then the numerical base locus ${\bf B}_{\rm num}(L_t)$ is the union $\bigcup_{|I|=m}\overline X_I$. 
	In particular it has no exceptional components.
	By Corollary \ref{cor:B+B}.(1) and (3), the successive minima $\epsilon_i(\xi;x)$ and $\epsilon_i^{\rm loc}(\xi;x)$ are determined by the numerical base loci ${\bf B}_{\rm num}(\xi_t)$. 
	We deduce that $\epsilon_{i}(L;x)=\epsilon^{\rm loc}_i(L;x)=i$ for all $1\leq i\leq n$.
	
	With an inductive argument, one can prove that a linear combination of the divisors $\overline X_{[n]\setminus\{i\}}$ and the exceptional $E$ is pseudoeffective if and only if the coefficients are all nonnegative. It follows in particular that $\mu(L;x)=n$ with $L_n$ being represented by $\sum_{i=1}^n\overline X_{[n]\setminus\{i\}}$. With the appropriate convention, this takes care of the case $m=n$. We run a descending induction on $m$.
	
	Assume $m<n$, thus $L_t$ is big, and that the description of ${\bf B}_{\rm num}(L_{t+1})$ is known. The restriction of $L_t$ to $\overline X_I$ with $|I|=m$ is $L'_t$ for $L'=\boxtimes_{i\in I}L_i$, still of $L_t$ ``type''.
	By the previous paragraph, it is not pseudo-effective, hence $\overline X_I\subseteq{\bf B}_{\rm num}(L_t)$. We have ${\bf B}_{\rm num}(L_t)\subset{\bf B}_{\rm num}(L_{t+1})$, and by assumption ${\bf B}_{\rm num}(L_t)=\bigcup_{|J|=m+1}\overline X_J$. To conclude ${\bf B}_{\rm num}(L_t)=\bigcup_{|I|=m}\overline X_I$, by symmetry it is sufficient to prove that if $x=(c_1,\ldots,c_n)$, and $x\neq y=(b_1,\ldots,b_m,b_{m+1},c_{m+2},\ldots,c_n)\in X_{[m+1]}\setminus\bigcup_{I\subset[m+1],\ |I|=m} X_{I}$ for some $b_i\neq c_i\in C_i$, then $y\not\in{\bf B}_{\rm num}(L_t)$. By abuse we identify $y$ with its preimage in $\overline X_{[m+1]}$. Equivalently we want an effective divisor numerically equivalent to $L$, with multiplicity $t$ at $x$, and that does not contain $y$. Such an example is the divisor 
	\[
	\sum\nolimits_{i=1}^mX_{[n]\setminus\{i\}}+(t-m)\cdotp X_{[n]\setminus\{m+1\}}+(1-t+m)\cdotp X'_{m+1}+\sum\nolimits_{j=m+2}^nX'_j \ ,
	\]
	where for $j\geq m+1$ the $X'_j$ is a fiber of ${\rm pr}:X\to C_j$ that does not contain $x$ or $y$.\end{proof} 
\end{example}
In comparison to the previous example, the next two examples show how a natural question about genus 3 curves has different answers depending on whether the curve is hyperelliptic or not.

\begin{example}[Jacobians of smooth plane quartics]
	Let $C$ be a non-hyperelliptic genus 3 curve. Let $J={\rm Jac}\, C={\rm Pic}^0(C)$ be its Jacobian, principally polarized by the theta divisor $L\deq\theta$. Then  \[\sinob{x}{\theta}=\simplex\bigl(\frac{12}7,\frac 74,2\bigr)\]
	for every $x\in J$.
\begin{proof}[Sketch of proof]Since $J$ is homogeneous, the results are independent of $x$. We will assume $x=o$ is the origin of $J$, but otherwise suppress it from notation.  	
	We want to prove that $\epsilon_1(L)=\frac{12}7$, $\epsilon_2(L)=\frac 74$ and $\epsilon_3(L)=2$. Their product is $\frac{12}7\cdotp\frac 74\cdotp 2=6=3!=(\theta^3)=\vol(L)$, so Theorem \ref{thm:lowerbounds} would conclude the proof. 
	
	The computation of the Seshadri constant and of the width are known, cf.,~ \cite{BS01} and respectively \cite{Ful21}\footnote{A conjectural picture for the infinitesimal width and its relationship to the behaviour of the multiplicity function for theta divisors is attempted in \cite{Loz24}.}. We outline the argument below, because we need the constructions for the computation of $\epsilon_2(L;o)$.
	Inside $J$ we have the following important subvarieties through $o$:
	
	\noindent$(1)$. The curve $R\deq\{\cal O_C(x-x')\ \mid \ \exists\ y\in C,\ x+x'+2y\in|K_C|\}$, where for every $y\in C\subset\bb P^2$ the projective tangent line ${\bf T}_yC$ meets $C$ again in two points $x,x'$. Bitangent lines show that $o\in R$. Numerically $R\equiv 16\theta^2$. Furthermore $\mult_oR=56$, corresponding to the 28 bitangents. 
	The curve $R$ realizes the minimum of $\frac{\theta\cdot C'}{\mult_oC'}$ among all curves $C'$ through $o$. In particular \[\epsilon_1(\theta)=\epsilon(\theta)=\epsilon^{\rm loc}_1(\theta)=\frac{12}7.\]
	
	\noindent $(2)$. The difference surface $D\deq C-C=\{\cal O_C(x-y)\ \mid\ x,y\in C\}$. 	It has class $D\equiv2\theta$ and multiplicity $4$ at $o$, the degree of the canonical embedding of $C$ seen as $\overline D\cap E\subset E$. We have $\overline D\simeq C\times C$ and one verifies that $\overline D|_{\overline D}$ is not pseudo-effective, hence $\overline D$ is a negative surface in the $\overline D=N_{\sigma}(\overline D)$ sense of the Nakayama $\sigma$-Zariski decomposition \cite{Nak04}. In particular
	\[\epsilon_3(\theta)=\epsilon^{\rm loc}_3(\theta)=\mu(\theta)=2.\]
	
	\noindent$(3)$.	The surface $T\deq\{\omega_C^{\vee}(2x+2y)\ \mid\ x,y\in C\}$. that is essentially the image of the theta divisor under the doubling map on $J$. It has class $T\equiv16\theta$ and ${\rm mult}_oT=28$, corresponding again to when $x,y$ are on a bitangent. Since $\overline T$ is irreducible and its numerical class is a positive combination of $\overline D$ and $E$, the class $\overline T\equiv 16\pi^*\theta-28E=16L_{\frac 74}$ is movable. For $t>\frac 74$, one checks that $L_t|_{\overline D}$ is not pseudo-effective. In fact $\overline T|_{\overline D}\equiv 2\overline R$ is an extremal ray in the pseudo-effective cone of $\overline D\simeq C\times C$ since $\overline R$ is a negative curve in $\overline D$. It follows that $\overline T$ is in the boundary of the movable cone of $\overline J$. 
	
	With similar methods one checks that for $\frac 74\leq t\leq 2$, the $\sigma$-Zariski decomposition of $L_t$ is a convex combination of $L_{\frac 74}$ and $L_2$, with $L_2$ supporting the negative part. Then $\overline D\subseteq{\bf B}_-(L_t)\subseteq{\bf B}_+(L_t)$ for $t>\frac 74$. For any $0<\delta<\frac{12}7$, the class $L_{\delta}$ is ample. We can write any $L_t$ with $\delta\leq t\leq 2$ as convex combination of $L_{\epsilon}$ and $L_2$. We deduce that ${\bf B}_+(L_t)\subseteq\overline D$ for all $0<t<2$. Thus ${\bf B}_+(L_t)=\overline D$ for $\frac 74<t<2$.
	
	If $t\geq\frac{12}7$, then $L_t\cdot\overline R\leq 0$, and so $\overline R\subseteq{\bf B}_+(L_t)$ for all $t\geq\frac{12}7$ and $\overline R\subseteq{\bf B}_-(L_t)$ for all $t>\frac{12}7$. 
	We know ${\bf B}_-(L_{\frac 74})\subseteq{\bf B}(L_{\frac 74})\subseteq\overline T$, and also ${\bf B}_-(L_{\frac 74})\subseteq{\bf B}_+(L_{\frac 74})\subseteq\overline D$. Thus ${\bf B}_-(L_{\frac 74})\subseteq\overline T\cap\overline D=\overline R$, giving ${\bf B}_-(L_{\frac 74})=\overline R$. For all $\frac{12}7<t<\frac 74$ we must also have ${\bf B}_-(L_t)=\overline R$. By Lemma \ref{lem:B-ray}, we conclude that \[\epsilon_2(L)=\epsilon^{\rm loc}_2(L)=\frac 74.\]
This finishes the proof.	
\end{proof} 
The geometry sketched above is detailed in \cite{FL25c} where $\inob{o}{\theta}$ is computed via Fujita--Zariski decompositions of some birational modifications of the blow-up of the origin of the Jacobian. Consequently, the family of graded linear series $A_{\bullet}^t$ from Theorem \ref{thm:slices}, computing the vertical slices of our generic iNObody, can be asymptotically realized as a complete graded linear series on the blow-up of the exceptional divisor $E\simeq \PP^2$ along the $56$ bitangents points, and again in one infinitely near point over each.
 This geometric realization of $A_{\bullet}^t$ falls in line with the Example \ref{ex:productcurves}, as explained in \cite{FL25b}. This provides a full characterization of the hypersurfaces that asymptotically appear as the tangent cone of a pluritheta divisor.
\end{example}

\subsection{Nonsimplicial examples}
 The three-dimensional hyperelliptic Jacobian is surprisingly much harder to tackle. There are few phenomena that take place here. On one hand, on the exceptional ray there will be base loci that will have entire irreducible components contained in the exceptional divisor. In particular, the infinitesimal successive minima will be distinct than there local counterpart from \cite{AI}. On the other hand, this forces our generic iNObody to be nonsimplicial. Finally, the Zariski decompositions that we study in \cite{FL25c} on some birational modification of the Jacobian suggests that the family of graded algebras $A_{\bullet}^t$ from Theorem \ref{thm:slices} may not have a geometric realization as complete linear series on some birational modification of the exceptional divisor.
 
\begin{example}[Hyperlliptic Jacobian threefolds]\label{ex:hyper3}
	Let $C$ be a hyperelliptic curve of genus 3. Let $J$, $o$, and $L=\theta$ as in the previous example. Then $\inob{o}{\theta}$ is not simplicial, and $\epsilon_2(L)<\epsilon^{\rm loc}_2(L)$. 

	\begin{proof}[Sketch of proof] 
Let $m:J\to J$ be the multiplication by $2$ map. Let $p_1,\ldots,p_8$ be the Weierstrass points of $C$, i.e. the ramifications of the hyperelliptic pencil. We have the following subvarieties of $J$ through the origin $o$:
	\smallskip
	
	\noindent $(1)$. $F\deq\{\cal O_C(2x-2p_1)\ \mid\ x\in C\}=m(C)$ where $C\subset J$ is the Abel--Jacobi embedding $x\mapsto\cal O_C(x-p_1)$. We have $F\equiv2\theta^2$ and $\mult_oF=8$, corresponding to the 8 Weierstrass points. As in \cite{BS01}, the curve $F$, is the Seshadri curve for $L$, meaning it determines the Seshadri constant \[\epsilon_1(L)=\epsilon^{\rm loc}_1(L)=\epsilon(L)=\frac 32.\]
	
	\noindent$(2)$. Let $D=C-C$ as for non-hyperelliptic Jacobian. In the hyperelliptic case  $\mult_oD=2$ and $D=C+C-\tau\equiv \theta$, where $\tau$ is the class of the hyperelliptic pencil.
	After blow-up, $\overline D\simeq C_2\deq{\rm Sym}^2C$, and \[\Gamma\deq\overline D\cap E\] 
	is the rational curve $\{x+x'\ \mid\ x\in C\}$ in $C_2$, where $x'$ is the hyperelliptic conjugate. In $E$, this curve is a conic, the canonical image of $C$ in $\bb P^2$. The strict transform $\overline F$ is the diagonal in $C_2$. The classes $\overline F$ and $\Gamma$ are extremal in $\Eff(C_2)$ and one computes that $\overline D|_{\overline D}\equiv\frac 12\overline F$. Consequently, one can prove that $\overline D$ is not big. See also \cite{Ful21}. Therefore
	\[\epsilon_3(L)=\epsilon^{\rm loc}_3(L)=\mu(L)=2.\]
	\noindent$(3)$. The surface $T\deq m(D)$ is defined as in the non-hyperelliptic case, but this time $T\equiv 16\theta$, and $\mult_oT=30$, corresponding to the $28$ lines through pairs of distinct $p_i$, and the conic $\Gamma\subset E$. We prove in \cite{FL25c} that the strict transform $\overline T$ is in the boundary of $\Mov(\overline J)$. From this, as in the non-hyperelliptic case, one proves that ${\bf B}_+(L_t)=\overline D$ for $\frac{15}8\leq t<2$ whereas $\dim{\bf B}_+(L_t)=1$ for $\frac 32\leq t<\frac{15}8$. We deduce that 
	\[\epsilon^{\rm loc}_2(L)=\frac{15}8.\]
	Then $\prod_{i=1}^3\epsilon^{\rm loc}_i(L)=\frac 32\cdotp\frac{15}8\cdotp 2=\frac{45}8<6=\vol(L)$. Thus $\inob{o}{L}$ is not simplicial by Corollary \ref{cor:simplicialcriterion}.
	
	To study $\epsilon_2(L)$, let $\overline{m}:\overline J'\to\overline J$ be the base change of $m$. Thus $\overline J'$ is the blow-up of $J$ in the 64 points of $2$-torsion. By abuse, denote by $\pi:\overline J'\to J$ the blow-up map, and by $\overline D'\subset\overline J'$ the strict transform so that $\overline T=\overline m(\overline D')$. The surface $D$ contains only 29 of the 64 points of torsion. Seen in ${\rm Pic}^2(C)$, these points of $D=C+C$ are the 28 points $\cal O_C(p_i+p_j)$ with $i\neq j$, and $\tau$ the class of the hyperelliptic pencil. Note that $\mult_{\tau}D=2$. After the blow-up of only $\tau\in {\rm Pic}^2(C)$, as before, the strict transform of $D$ is isomorphic to $C_2$. Thus $\overline D'$ is the further blow-up of $C_2$ in the $28$ points $p_i+p_j$ with $i<j$.
	
	The strict transforms of the curves $C+p_i\subset{\rm Pic}^2(C)$ are disjoint in $\overline J'$ and contained in $\overline D'$. Each of them is mapped by $\overline m$ isomorphically onto $\overline F$. Let 
	\[R\deq\sum\nolimits_{i=1}^8\overline{C+p_i}.\]
	The curves $\overline F$ and $\Gamma$ in $C_2\simeq{\rm Bl}_{\tau}D$ do not pass through any of the 28 points $p_i+p_j$ with $i\neq j$. Denote their strict transforms in $\overline D'$ by $\overline F'$ and $\Gamma'$. We have the intersection relations
	\begin{center}
		\begin{tabular}{|c|r|r|r|}
			\hline 
			& $\Gamma'$& $\overline F'$& $R$\\
			\hline\hline 
			$\Gamma'$& $-2$&$8$&$8$\\
			\hline 
			$\overline F'$& $8$&$-8$&$16$\\
			\hline 
			$R$& $8$& $16$& $-48$\\
			\hline 
		\end{tabular} 
	\end{center}
	Using them, one computes that 
	\[\overline{m}^*\overline T|_{\overline D'}\equiv 2(16\Gamma'+\overline F'+3R)+(6\Gamma'+9R)\]
	is a Zariski decomposition, whence $\Gamma'\cup R={\bf B}_-(\overline{m}^*\overline T|_{\overline D'})\subseteq {\bf B}_-(\overline{m}^*\overline T)=\overline{m}^{-1}{\bf B}_-(\overline T)$. We deduce that $\Gamma\cup\overline F\subseteq{\bf B}_-(L_{\frac{15}8})$ on $J$.
	From $L_{\frac 32}\cdot\overline F=0$, we find $\overline F\subset{\bf B}_+(L_t)$ for all $t\geq\frac 32$. By running the Zariski decomposition algorithm on $\overline{m}^*{L_t}|_{\overline D'}$ for $\frac 32<t\leq\frac{15}8$, one shows that $\overline R$ appears in all these decompositions, and then that $\Gamma'$ is present precisely when $\frac{24}{13}<t\leq \frac{15}8$. Thus the exceptional $\Gamma$ appears in ${\bf B}_-(L_t)$ for all $\frac{24}{13}<t\leq \frac{15}8$, and is therefore a component of ${\bf B}_+(L_{\frac{24}{13}})$. These imply that 
	\[
	\epsilon_2(L)\leq\frac{24}{13}<\frac{15}8=\epsilon^{\rm loc}_2(L) \
	\]
 as claimed. The first inequality is actually an equality and it follows from the geometric descriptions sketched above as detailed in \cite{FL25c}. For example, if $\rho:\widetilde J\to \overline J$ denotes the blow-up of $\overline F$ with exceptional divisor $\widetilde E$, then $P_{\sigma}(\rho^*L_t)=\rho^*L_t-(\frac 43t-2)\widetilde E$ is nef (and $N_{\sigma}(\rho^*L_t)$ is supported on $F$) for all $\frac 32\leq t\leq\frac{24}{13}$. From these, one can show that ${\bf B}_-(L_t)\subseteq\overline F$ for $\frac 32\leq t\leq\frac{24}{13}$ and the desired equality.  
\end{proof}
\end{example}

Easier non-simplicial examples can be found on surfaces. Moreover they can be used to show that the very general assumption on the point of the blow-up is necessary for both the upper bound on the volume, and for the upper bound on the iNObody in Theorem \ref{thm:upperbound}.

%\begin{example}
%	The last inclusion in (1) in Theorem \ref{thm:upperbound} is not sharp for some pairs $(X,L)$ in dimension $n\geq 2$. Take for example $X=C_1\times C_2$ the product of two curves and $L=L_1\boxtimes L_2$, where $\deg L_i=d_i$ on $C_i$. Assume $d_1\leq d_2$. Then $\epsilon(L;x)=d_1$ and $\mu(L;x)=d_1+d_2$ and the iNObody is the shaded isosceles trapeze in Figure \ref{fig:C1C2}.
%	
%	\begin{figure}[h]
%		\begin{tikzpicture}[scale=1.5]
%			\draw[thick,fill=gray!30] (0,0) node[left]{$(0,0)$}-- (4,0) node[right]{$(d_1+d_2,0)$} -- (2.8,1.2) node[above]{$(d_2,d_1)$} -- (1.2,1.2) node[above]{$(d_1,d_1)$}--cycle;
%			\draw[dashed] (4,0)--(4,1.2) node[right]{$(d_1+d_2,d_1)$}--(2.8,1.2);
%		\end{tikzpicture}\centering\caption{Box-product class on the product of two curves. $\inob{x}{L}$ vs. expected upper bound}\label{fig:C1C2}
%	\end{figure}
%The computation is detailed in \cite{FL25}. It is achieved by determining the Zariski decomposition for every $L_t$ with $t\in[0,\mu(L;x)]$ and applying \cite[Theorem 6.4]{LM09}.\qed
%\end{example}

\begin{example}\label{ex:verygeneralpointnecessary}
	Let $X$ be a smooth projective surface, $x\in X$, and $\xi\in{\rm Big}_x(X)$. From \cite[Theorem 6.4]{LM09} it follows that $\inob{x}{\xi}$ can be computed and it is the area under the graph of 
	\[
	[0,\mu(\xi;x)]\to\bb R_+:t\mapsto t-(N(\xi_t)\cdot E) \ ,
	\]
	where $N(\xi_t)$ is the negative part in the Zariski decomposition of $\xi_t$ on $\overline X$.
	
	With this data in hand, let $X={\rm Bl}_p\bb P^2$ with exceptional divisor $F$. Let $x$ be any point in $F$, so not a very general point of $X$. Let $\ell\subset X$ be the strict transform of the unique line through $p$ in $\bb P^2$ such that $\{x\}=\ell\cap F$. Let $\pi:\overline X\to X$ be the blow-up of $x$ with exceptional divisor $E$. We have the intersection pairing
	\begin{center}
		\begin{tabular}{|c|r|r|r|}
			\hline & $\overline\ell $& $\overline F$ & $E$\\ 
			\hline\hline 
			$\overline \ell$& $-1$&$0$&$1$\\
			\hline 
			$\overline F$& $0$& $-2$& $1$\\
			\hline 			$E$& $1$&$1$&$-1$\\
			\hline 
		\end{tabular} 	
	\end{center}
	where $\overline\ell$ and $\overline F$ denote the strict transforms.
	Let
	\[L=(u+v)\ell+uF\]
	on $X$ where $u,v>0$. Then $L$ is ample, $\vol(L)=(L^2)=u^2+2uv$ and $\pi^*L=(u+v)\overline\ell+u\overline F+(2u+v)G$. Using that $\overline X$ is toric, one computes that if $t\geq 0$, then $L_t$ is effective if and only if $t\leq 2u+v$ and furthermore nef if and only if $t\leq\min\{u,v\}$.
	Assume $u\geq v$. Then $\epsilon(L;x)=v$ and $\mu(L;x)=2u+v$. The polytope $\invertpoly(v,2u+v)$ (see Remark \ref{rmk:tiltedshapes}) is the right trapeze $\textup{convex hull}\bigl((0,0),(2u+v,0),(2u+v,v),(v,v)\bigr)$ which has normalized area $4uv+v^2$. When $u>(1+\sqrt 2)v$, this is less than $u^2+2uv= \vol(L)=(L^2)=2\cdot {\rm vol}_{\bb R^2}\inob{x}{L}$, thus the iNObody cannot be contained in $\invertpoly(v,2u+v)$. In fact, one can compute
	\[\inob{x}{L}=\textup{convex hull}\bigl((0,0),(2u+v,0),(u,\frac 12(u+v)),(v,v)\bigr).\]
	This is represented in Figure \ref{fig:outofbounds} where the shaded area is the iNObody for $u=3$ and $v=1$ and the dashed segments are the expected upper bound (after tilting).\qed
	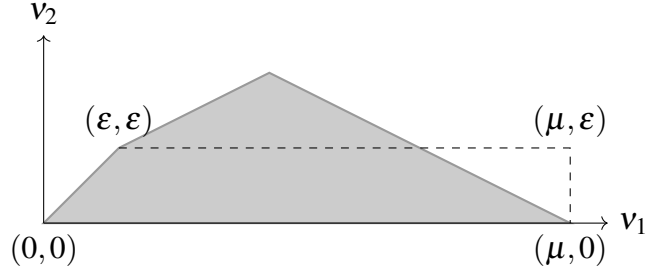
\begin{figure}
	\begin{tikzpicture}
	\draw [->](0,0) -- (7.5,0) node[right]{$\nu_1$};
	\draw [->](0,0) -- (0,2.5) node[above]{$\nu_2$};
	\draw (0,0) node[below]{$(0,0)$};
	\draw (7,0) node[below]{$(\mu,0)$};
	\draw (1,1) node[above]{$(\epsilon,\epsilon)$};
	\draw [thick, fill=gray, opacity=0.4] (0,0) --(7,0) -- (3,2) -- (1,1) -- cycle;
	\draw [dashed] (1,1)--(7,1) node[above]{$(\mu,\epsilon)$}--(7,0);
\end{tikzpicture}\centering\caption{Failure of upper bounds at special points even for very general flags}\label{fig:outofbounds} 
\end{figure} 
\end{example}

%\begin{example}
%Let $C$ be a complex projective curve of genus 2. Let $X=C\times C$, and $x\in X$ a point not lying on the diagonal $\Delta$. Let $L=d_1f_1+d_2f_2+\Delta$ with $d_2\geq d_1>1$. Then $L$ is ample. In \cite{FL25} we compute that $\inob{x}{L}$ is the pentagon in Figure \ref{fig:Cgenus2}.\qed 
%	\begin{figure}[h]
%	\begin{tikzpicture}[scale=.6]
%		\draw[thick,fill=gray!30] (0,0) node[left]{$(0,0)$}-- (6,0) node[right]{$(d_1+d_2,0)$} -- (6,2) node[right]{$(d_1+d_2,2)$} -- (5,3) node[right]{$(d_2+1,d_1+1)$}-- (3,3) node[left]{$(d_1+1,d_1+1)$}--cycle;
%	\end{tikzpicture}\centering\caption{A non-box-product class on the self product of a genus 2 curve}\label{fig:Cgenus2}
%\end{figure}
%\end{example}

The upper bound on $\sinob{X}{L}$ in Theorem \ref{thm:upperbound} can be sharp.
We have seen this already in the cases where $\sinob{X}{L}=\simplex\,_t$ when the lower bound is also sharp. The next example, inspired by the previous one on surfaces, is a case where the upper bound is sharp, while the lower bound is not.

\begin{example}\label{ex:upperboundsharp}
	Let $X={\rm Bl}_p\bb P^n$ with exceptional divisor $F$. Let $x\in X\setminus F$. Let $H$ be the class of the pullback of a general hyperplane in $\bb P^n$ to $X$.  Then $L=aH-F$ is ample if $a>1$. For $x\in X\setminus F$ one can compute that $\epsilon(L;x)=a-1$ with the unique Seshadri curve for $L$ through $x$ being the strict transform of the line through $p$ and the image of $x$ in $\bb P^n$. Furthermore $\mu(L;x)=a$, and the base locus of $L_a$ on $\overline X$ is the strict transform of the Seshadri curve. In particular, we deduce that $\epsilon_i(L;x)=a$ for $2\leq i\leq n$. We have 
	\[
	n!\cdot {\rm vol}_{\bb R^n}(\polytope(a-1,a,a,\ldots,a))=a^n-1=(L^n)=\vol(L) \ ,
	\] 
	as $\polytope(a-1,a,\ldots,a)$ is $\simplex\,_a$ with the $(a-1){\bf e}_1+\simplex\,_1$ tip cut off. By Theorem \ref{thm:upperbound} we get then
	\[\sinob{x}{L}=\polytope(a-1,a,\ldots,a)\ .\]
	for any $x\in X\setminus F$.\qed
\end{example}

More non-simplicial examples are computed in \cite{FL25} for box-product polarizations on product threefolds like $C\times\bb P^2$ or $C\times J$, where $J$ is a Jacobian surface. 

%	\begin{figure}[h]
%	\begin{tikzpicture}[scale=1]
%		\draw[thick,fill=gray!30] (0,0) node[left]{$(0,0)$}-- (3,0) node[right]{$(a,0)$} -- (3,2) node[right]{$(a,a-1)$} -- (2,2) node[left]{$(a-1,a-1)$}--cycle;
%	\end{tikzpicture}\centering\caption{The upper bound can be sharp}\label{fig:uppersharp}
%\end{figure}

\section{Applications}
In this section we provide three applications of the theory developed so far. The first one eyes a particular case of the Ein--Lazarsfeld conjecture on the lower bound of the Seshadri constant for ample line bundles at very general points. 
Next, we identify an interval where a restricted volume function is nonincreasing.
Finally we provide some interesting connections between certain invariants of the iNObody of a big divisor class and the Seshadri constant of an associated curve class. 

\subsection{Bounds on Seshadri constants} 
Inspired by \cite{CN14,KL18,Bal23}, we prove in any dimension a particular case of the Ein--Lazarsfeld conjecture.
\begin{proposition}
	Let $X$ be a normal complex projective variety of dimension $n$. Let $L$ be an ample line bundle on $X$. Denote $\epsilon\deq\epsilon(L;1)=\epsilon(L;x)$ for $x$ very general.
	Assume that for $x$ very general we have that ${\bf B}_+(L_{\epsilon})$ contains a prime divisor $Z$ in $E\simeq\bb P^{n-1}$. Then $\epsilon\geq 1$.
\end{proposition}

\begin{proof}
	By definition, $L_{\epsilon}$ is nef and not ample.
	If $L_{\epsilon}$ is not big, then $\epsilon_i(L)=\epsilon$ for all $i$. Thus, $\epsilon=\epsilon_{n}(L)=\mu(L)\geq\sqrt[n]{(L^n)}\geq 1$, where $\mu\deq\mu(L)$ is the common value of $\mu(L;x)$ at a very general point $x$.
	
	Let $L_{\epsilon}$ be big now, i.e. ${\bf B}_+(L_{\epsilon})\neq\overline X$. Since $L$ is nef and $L_{\epsilon}|_E$ is ample, then by \cite[Theorem 10.3.4]{Laz04} $Z$ cannot be a component of ${\bf B}_+(L_{\epsilon})$. So, ${\bf B}_+(L_{\epsilon})$ contains a component $\overline Y$, through $Z$, coming from a prime divisor $Y\subset X$. Further this implies that $(L_{\epsilon}^{n-1}\cdot\overline Y)=0$. So, if $W\deq \overline Y|_E$ is of degree $d\deq \mult_xY$, then $(L^{n-1}\cdot Y)=d\cdot\epsilon^{n-1}$. In particular, when $d=1$, $Y$ is smooth at $x$ and by definition $\epsilon\geq 1$. 
	
	Assume $d\geq 2$ and suppose $\epsilon<1$. As $(L^{n-1}\cdot Y)$ is an integer, then this implies $\epsilon^{n-1}\leq\frac{d-1}d$, i.e.
	  \begin{equation}\label{eq:specialseshadriup}\epsilon\leq\sqrt[n-1]{\frac{d-1}d}\ .\end{equation} 
	From Lemma \ref{lem:AIverygeneral}, we deduce $m_t\deq\mult_{\overline Y}||L_t||\geq t-\epsilon$ for all rational $t\in[\epsilon,\mu)$. This is where the very general condition on $x$ is used. Then $(L_t-m_t\overline Y)|_E=tH-m_tW$ is (pseudo)effective, thus
	$t\geq m_t\cdot d\geq d(t-\epsilon)$, and in particular
	\[\mu\leq \frac d{d-1}\epsilon\ .\]
	As $\epsilon_1(L)=\ldots=\epsilon_{n-1}(L)=\epsilon$ by assumptions, then for any 
	 $t\in[\epsilon,\mu)$ and $(t,\nu_2,\ldots,\nu_n)\in\inob{X}{L}$, we necessarily have
	\[\nu_2+\ldots+\nu_n\leq t-m_t\cdot d\leq t-d(t-\epsilon) \ .\]
	This is because if $t$ is rational, then $m_t\cdot W$ is in the fixed part of $|L_t|_{\bb Q}\big|_E$ and by choosing $Y_n\not\in W$, this fixed part does not affect $\nu_{Y_{\bullet}}(D)$ for any $D\in|L_t|_{\bb Q}$.
	Then 
	\[(L^n)\leq\epsilon^n+n!\cdotp\int_{\epsilon}^\mu\frac 1{(n-1)!}(t-d(t-\epsilon))^{n-1}dt=\epsilon^n+\frac 1{d-1}(\epsilon^n-(d\epsilon-(d-1)\mu)^n)\ .\]
	The RHS is at most $\frac d{d-1}\epsilon^n$, thus
	$\epsilon\geq\sqrt[n]{\frac{d-1}d}$, which contradicts \eqref{eq:specialseshadriup}.
\end{proof}

\subsection{Monotonicity of a restricted volume function}
If $Y\subset X$ is a prime divisor on a projective variety of dimension $n$, and $L$ is a Cartier divisor, then the \emph{restricted volume} of $L$ to $Y$ is
\[\vol_{X|Y}(L)\deq\lim_{m\to\infty}\frac{\dim_{\bb C} {\rm Im}\bigl(H^0(X,mL)\to H^0(Y,mL|_Y)\bigr)}{m^{n-1}/(n-1)!}\ .\]
This is a $(n-1)$-homogeneous numerical invariant of $L$. 
By \cite[Theorem A]{ELMNP09} it extends naturally in continuous fashion to those classes in $N^1(X)_{\bb R}$ without $Y$ in their augmented base locus.  

With the assumptions of Theorem \ref{thm:mintrosuccessjets}, in this section we study the restricted volume function
\[t\mapsto \vol_{\overline X|E}(\xi_t)\ .\]
It follows from the result of \cite{LM09} in Theorem \ref{thm:slices} that  
\[\vol_{\overline X|Y}(\xi_t)=(n-1)!\cdot\vol_{\bb R^{n-1}}\inob{x}{\xi}_{\nu_1=t}\ \]
for all $t\in(0,\mu(\xi;x))$.
This invites a convex geometric study of the function above.
For example the slices are the simplex $\simplex\,_t$ for $t\in[0,\epsilon_1(\xi;x)]$ by \eqref{eq:introKL17}, and in particular the restricted volume function agrees with $t^{n-1}$ on this interval. When $x$ is very general, we can also control the behavior of this function towards the opposite end of the interval $(0,\mu(\xi;x))$.

\begin{proposition}
	Let $X$ be a normal complex projective variety of dimension $n$ and $\xi\in N^1(X)_{\bb R}$ a big class. If $x\in X$ is very general and $E\subset\overline X={\rm Bl}_xX$ the exceptional divisor, then the function
	\[t\mapsto{\rm vol}_{\overline X|E}(\xi_t)=(n-1)!\cdot\vol_{\bb R^{n-1}}\inob{x}{\xi}_{\nu_1=t}\]
	is nonincreasing on the interval $[\epsilon_{n-1}(\xi;x),\mu(\xi;x))$. In particular $t\mapsto\vol(\xi_t)$ is convex on this interval.
\end{proposition}
\noindent See also \cite[Lemma 2.4]{CN14}. The result fails if $x$ is not very general by Example \ref{ex:verygeneralpointnecessary}. 
\begin{proof} 
	By Theorem \ref{thm:upperbound}, $S\cdot\inob{x}{\xi}=\sinob{x}{\xi}$ is Borel-fixed, where $S$ is as in \eqref{eq:introS}.
	More generally, we prove that if $P\subset\bb R^n_+$ is a Borel-fixed convex body, then $t\mapsto {\rm vol}_{\bb R^{n-1}}((S^{-1}P)_{\nu_1=t})$ is nonincreasing for $t\in [w_{n-1}(P),w_n(P))=[w_n(S^{-1}P),w_1(S^{-1}P))$. Modulo a constant scaling factor, this is the same as the function $t\mapsto{\rm vol}_{\bb R^{n-1}}(P_{\alpha_1+\ldots+\alpha_n=t})$.
	
	For the proof, we claim that if $t<t+\delta$ are in the interval above, then the vertical translation by $-\delta{\bf e}_n$ is a well-defined (and clearly injective) map $P_{\alpha_1+\ldots+\alpha_n=t+\delta}\to  P_{\alpha_1+\ldots+\alpha_n=t}$. It is sufficient to observe that $P_{\alpha_1+\ldots+\alpha_n=t+\delta}-\delta{\bf e}_n$ is fully contained in the upper half space $\alpha_n>0$. This follows from the defining inequality $\alpha_1+\ldots+\alpha_{n-1}\leq w_{n-1}(P)$ of $\polytope(w_1(P),\ldots,w_n(P))\supseteq P$.
	
	For the last statement, use that $-n\cdot\vol_{\overline X|E}(\xi_t)=\frac{d}{dt}\vol(\xi_t)$ by \cite[Corollary C]{LM09}.
\end{proof} 

\subsection{Seshadri constants for curves}
Let $X$ be an $n$-dimensional complex smooth projective variety. A curve class $C\in \textup{N}_1(X)$ is \textit{movable}, if it is the push-forward by a birational modification of a complete intersection of $(n-1)$ ample classes, see \cite[Section 11.4.C]{Laz04} for more detailed definition and properties. This concept can be extended to $\RR$-classes and the set of movable curve classes forms a cone $\Mov_1(X)\subseteq \textup{N}_1(X)_{\RR}$. By \cite{BDPP13}, this cone is the dual of the cone of pseudoeffective Cartier divisor classes. 

Lehamnn--Xiao \cite{LX19} show that for every curve class $C$ in the interior of $\Mov_1(X)$ there is a unique divisor class $P(C)$ in the interior of the cone of movable divisors of $X$ such that 
\[C=\langle P(C)^{n-1}\rangle\ ,\]
where the right side is the positive product introduced by Boucksom in \cite{Bou02}, and further exploited in \cite{BFJ09,BDPP13}.
When $D$ is big and nef, $\langle D^{n-1}\rangle$ is the usual self-intersection class $(D^{n-1})$. 
Devey \cite{Dev22} brings convex geometry into this picture, and defines Newton--Okounkov bodies for $C$ by using the class $P(C)$.

Given the above, and our work for divisors, it is reasonable to expect that starting with a class $\xi\in{\rm Big}_x(X)$, we should be able to read invariants of the class $C=\langle\xi^{n-1}\rangle=\langle P_{\sigma}(\xi)^{n-1}\rangle$ from $\inob{x}{\xi}=\inob{x}{P_{\sigma}(\xi)}$. One such invariant is the Seshadri constant for a movable curve class introduced by the first named author in 
\cite{Ful21}. For a movable curve class $C\in{\rm Mov}_1(X)$ and a point $x\in X$, it is defined as
\[\epsilon(C;x)\deq\inf_D\frac{C\cdotp D}{\mult_xD}\]
as $D$ ranges through effective Cartier divisors containing $x$ in its support. When $x$ is smooth, then
\[\epsilon(C;x)=\sup\{s\geq 0\ \mid\ \pi^*C-s\ell\in\Mov_1(\overline X)\}\ ,\]
where $\ell$ is the class of a line in $\bb P^{n-1}\simeq E\subset X$. \cite{Ful21} proves analogues of the Seshadri ampleness criterion, jet separation, and describes the locus $\{x\in X\ \mid\ \epsilon(C;x)=0\}$.

Our goal is to provide strong bounds for the Seshadri constant of the curve $\langle\xi^{n-1}\rangle$ with respect to the normalized volume of a slice of the straightened up body $\sinob{x}{\xi}$. As we shall see afterwards, the last lower bound is an equality in many cases, but not always.

\begin{theorem}\label{thm:curveseshadribounds}
	Let $X$ be a complex projective manifold of dimension $n$, let $x\in X$ and $\xi\in {\rm Big}_x(X)$. Then
\begin{enumerate}[(1)]
	\item $\prod_{i=1}^{n-1}\epsilon_i(\xi;x)\leq (n-1)!\cdot {\rm vol}_{\bb R^{n-1}}\sinob{x}{\xi}_{\alpha_n=0}\leq\frac{\vol(\xi)}{\mu(\xi;x)}$ ,
\smallskip

\item $\prod\nolimits_{i=1}^{n-1}\epsilon^{\rm loc}_i(\xi;x)\leq\epsilon(\langle \xi^{n-1}\rangle;x)\leq\frac{\vol(\xi)}{\mu(\xi;x)}$ ,
\smallskip

\item If $x$ is very general, then 
	\begin{equation}\label{eq:curveseshadribound} (n-1)!\cdot {\rm vol}_{\bb R^{n-1}}\sinob{x}{\xi}_{\alpha_n=0}\leq\epsilon(\langle \xi^{n-1}\rangle;x)\ .\end{equation} 
\end{enumerate}
\end{theorem}

\begin{remark}
	For all examples in Section \ref{s:examples}, except for Example \ref{ex:hyper3}, the inequality \eqref{eq:curveseshadribound} is an equality. More generally,
	\begin{enumerate}[(i)]
		\item If $X$ is a surface then both terms are equal to the moving Seshadri constant of $\xi$, when $x\not\in{\bf B}_+(\xi)$.
		\smallskip
		
		\item If $\sinob{x}{\xi}$ is simplicial, then all inequalities in Theorem \ref{thm:curveseshadribounds} are equalities by Corollary \ref{cor:simplicialcriterion}. 
		%\item In the toric Example \ref{ex:upperboundsharp}, both terms equal $a^{n-1}-1$, this time at a $T$-invariant point.
	\end{enumerate}
The very general assumption on $x$ is not needed in these cases. We do not know if \eqref{eq:curveseshadribound} holds without this assumption since we use the Borel-fixed property of $\sinob{x}{\xi}$ in our proof. But a natural variant of \eqref{eq:curveseshadribound} that we describe below does not always hold at special points.

When $x$ is very general, the Borel-fixed property of $\sinob{x}{\xi}$ implies that $\sinob{x}{\xi}_{\alpha_n=0}={\rm pr}(\inob{x}{\xi})$, where ${\rm pr}$ is the projection onto the $\nu_1=0$ hyperplane. At special points, the inequality $(n-1)!\cdot\vol_{\bb R^{n-1}}{\rm pr}(\inob{x}{\xi})\leq\epsilon(\langle\xi^{n-1}\rangle;x)$ may fail.
See Example \ref{ex:verygeneralpointnecessary}.
\end{remark}

However \eqref{eq:curveseshadribound} may be strict:

\begin{example}
	We go back to Example \ref{ex:hyper3}, where  $C$ is a hyperelliptic curve of genus 3 and $X={\rm Jac}(C)$ with its natural $\theta$ polarization. We may assume that $x$ is the origin.
	
	The surface $S=C-C$ computes $\mu(\theta;x)=2$. By Theorem \ref{thm:curveseshadribounds}, we have $\epsilon((\theta^2);x)\leq\frac{(\theta^3)}{\mu(\theta;x)}=3$. As $\theta-\frac 32E$ is nef, the class $\theta^2-3\ell\equiv \overline S\cdot(\theta-\frac 32E)$ intersects nonnegatively any prime divisor on $\overline X$ not $\overline S$. By direct verification, it has 0 intersection with $\overline S$. Thus, \cite{BDPP13} yields $\epsilon((\theta^2);x)\geq 3$, and equality holds.
	
	If \eqref{eq:curveseshadribound} were an equality, the upper bound in Theorem \ref{thm:curveseshadribounds}.(1) forces  $\sinob{x}{\theta}$ to be a pyramid from $2{\bf e}_3$ over the slice $\sinob{x}{\theta}_{\alpha_3=0}$. We show that this is not the case. With the notation in Example \ref{ex:hyper3}, we know that $P_{\sigma}(\rho^*L_t)=\rho^*L_t-(\frac 43t-2)F$ is nef for $t=\frac{24}{13}$. In \cite{FL25c} we show that the vertical slice $\inob{x}{\theta}_{\nu_1=\frac{24}{13}}$ of the tilted body is a NObody on the surface $\widetilde E$, the strict transform of $E$ via $\rho$. 
	The surface $\widetilde E$ is the blow-up of $\bb P^2$ in 8 points on a conic, the image of the Weierstrass points in the image of the canonical map. The divisors on $\widetilde E$ we work with are $tH-\frac 6{13}F|_{\widetilde E}=(\rho^*L_t-(\frac 43\cdot\frac{24}{13}-2)F)|_{\widetilde E}$, where $H$ is the pullback of a line on $\bb P^2$ and $F|_{\widetilde E}$ is the sum of the 8 exceptional lines. The flag for the NObody is the strict transform of a general linear flag in $\bb P^2$. 	
	
	After an easy computation, the NObody is equal to $\simplex(\frac{12}{13},\frac{24}{13})$. After straightening to $\sinob{x}{\theta}$, we deduce that $(\frac{12}{13},0,\frac{12}{13})\in\sinob{x}{\theta}$. The ray from $2{\bf e}_3$ through $(\frac{12}{13},0,\frac{12}{13})$ meets the slice $\alpha_3=0$ in $(\frac{12}7,0,0)$. As $\frac{12}7>\frac 32=w_1(\sinob{x}{\theta})$, and so $(\frac{12}7,0,0)\not\in\sinob{x}{\theta}$. Thus, $\sinob{x}{\theta}$ cannot be a pyramid.\qed 
\end{example}

\noindent More examples will be given in \cite{FL25} of threefolds polarized by box products where \eqref{eq:curveseshadribound} is strict even when $\epsilon_i(L;x)=\epsilon_i^{\rm loc}(L;x)$ for all $i$.

\begin{proof}[Proof of Theorem \ref{thm:curveseshadribounds}]
$(1)$ The left side follows from Theorem \ref{thm:introreadsuccessive} and the right side from the convexity of $\sinob{x}{\xi}$ using that $\mu(\xi;x)$ is the $\alpha_n$-width for all $x$.
\smallskip

For $(2)$ and $(3)$ we first treat first the case when $\xi$ is additionally nef and Cartier. These showcase the main ideas. The passage to $\xi$ being only big is a technical application of careful Fujita approximations.

$(2)$.	Let $C\deq\langle \xi^{n-1}\rangle$. The class $\pi^*C-\epsilon(C;x)\cdotp\ell$ is a limit of movable classes, which implies that $(\pi^*C-\epsilon(C;x)\cdotp\ell)\cdot(\pi^*\xi-\mu(\xi;x)\cdotp E)\geq 0$. The right side inequality will follow whenever we show $C\cdot \xi=\langle \xi^{n-1}\rangle\cdot \xi=\langle \xi^{n}\rangle=\vol(\xi)$. For this note that the last two equalities are \cite[Corollary 3.6 and Theorem 3.1]{BFJ09}. They generalize known results on surfaces: If $L=P+N$ is the Zariski decomposition of a big divisor on a surface, then $P=\langle L\rangle$ and $\vol L=P^2=P\cdot L$.

For the left side inequality, assume that $\xi$ is represented by a big and semiample Cartier divisor $H$. Fix $Z$ a prime divisor on $X$ through $x$. 
Let $0<t_i<\epsilon_i^{\rm loc}(H;x)$ be rational numbers close to $\epsilon_i^{\rm loc}(H;x)$. 
By Definition \ref{def:aiminima} there exists an effective $\bb Q$-divisor $D_{n-1}\equiv H$ with $\mult_xD_{n-1}\geq t_{n-1}$ without $Z$ in its support. Inductively we construct $D_1,\ldots,D_{n-1}$ such that $\mult_xD_i\geq t_i$ and such that $x$ is an isolated component of $D_1\cap\ldots\cap D_{n-1}\cap Z$.
We claim that
\[(H^{n-1}\cdot Z)\geq t_1\cdot\ldots\cdot t_{n-1}\cdot\mult_xZ\ ,\]
and so $\epsilon(H^{n-1};x)\geq t_1\cdot\ldots\cdot t_{n-1}$. The conclusion follows by letting the $t_i$ approach $\epsilon_i^{\rm loc}(H;x)$. For the claim, since $H$ is semiample, by \cite[Example 12.2.7]{Ful98} all components of $D_1\cap\ldots\cap D_{n-1}\cap Z$ have nonnegative contribution to the intersection number. This is even true of the components of unexpected dimension. By construction no such component passes through $x$. It follows that 
\[
(H^{n-1}\cdot Z)\geq i(x,H^{n-1}\cdot Z;X)\geq \prod\nolimits_{i=1}^{n-1}\mult_xD_i\cdot\mult_xZ\geq t_1\cdot\ldots\cdot t_{n-1}\cdot\mult_xZ \ .
\]
The second inequality uses \cite[p233]{Ful98}. A similar argument appears in \cite[Lemma 3.2]{AI}.

We now consider the case when $\xi\in{\rm Big}_x(X)$ is arbitrary. 
The idea is to reduce to the previous case by a continuity argument using Fujita approximations on blow-ups of $X$.
By \cite[Proposition 3.7]{Leh13} as in \cite[Theorem 6.22]{FL17}  or \cite[Proposition 3.3]{LX19}, there exist:
\begin{itemize}
	\item a very ample divisor $G$ on $X$,
	\item a sequence of smooth birational models $\rho_m:X_m\to X$ that are each sequences of blow-ups with smooth centers supported in ${\bf B}_+(\xi)$, and
	\item big and semiample $\bb Q$-divisors $B_m$ on $X_m$ that are Fujita approximations for $\xi$. More precisely, we have  $P_{\sigma}(\rho_m^*\xi)-\frac 1m\rho_m^*G\leq_{x_m} B_m\leq_{x_m} P_{\sigma}(\rho_m^*\xi)$, where $P_{\sigma}(D)$ denotes the positive part of the Nakayama $\sigma$-Zariski decomposition \cite{Nak04} of the $\bb R$-divisor $D$, where $x_m\deq\rho_m^{-1}\{x\}$, and the notation $D\leq_xD'$ means that $D'-D$ is represented by an effective $\bb R$-divisor without $x$ in its support.
\end{itemize} 	
The invariants $\epsilon^{\rm loc}_i(\xi;x)$,	 $\vol(\xi)=\langle\xi^n\rangle$, and $\langle \xi^{n-1}\rangle$ do not change if we replace $\xi$ by $P_{\sigma}(\xi)$. The first is by the analogue of Lemma \ref{lem:zariskidecompositionsuccessiveminima} for the Ambro--Ito minima, which shows further that $P_{\sigma}(\xi)\in{\rm Big}_x(X)$. The other two are by the definition of the positive product in terms of Fujita approximations. More precisely it is because if $N$ is a nef Weil $b$-divisor over $X$, and $L$ is big on $X$, then $N\leq L$ as $b$-divisors if and only if $N\leq P(L)$, where $P(L)$ is the positive part Weil $b$-divisor whose incarnation in every smooth birational model $\rho:X'\to X$ is $P_{\sigma}(\rho^*L)$. Hence we can assume $\xi\in {\rm Big}_x(X)$ to be movable.

For $m$ sufficiently large, $\xi-\frac 1mG\in{\rm Big}_x(X)$ and $\rho_m^*(L-\frac 1mG)$ and $\rho_m^*L$ are in ${\rm Big}_{x_m}(X_m)$.
We have 
\[
(P_{\sigma}(\rho_m^*\xi)-\rho_m^*\frac 1mG)-P_{\sigma}(\rho_m^*(\xi-\frac 1mG))=N_{\sigma}(\rho_m^*(\xi-\frac 1mG))-N_{\sigma}(\rho_m^*\xi) .
\]
Since $\rho_m^*G$ is movable, the latter difference is an effective divisor. Its support is contained in the support of the first negative part, in particular in ${\bf B}_-(\rho_m^*(\xi-\frac 1mG))$, thus it does not contain $x_m$. The upshot is
\[P_{\sigma}(\rho_m^*(\xi-\frac 1mG))\leq_{x_m}B_m\leq_{x_m}P_{\sigma}(\rho_m^*\xi)\ .\]
Using Lemma \ref{lem:zariskidecompositionsuccessiveminima} and its proof together with the birational invariance of successive minima, this implies $\epsilon^{\rm loc}_i(\xi-\frac 1mG;x)\leq\epsilon^{\rm loc}_i(B_m;x_m)\leq\epsilon^{\rm loc}_i(\xi;x)$. Thus, $\lim_{m\to\infty}\epsilon^{\rm loc}_i(B_m;x_m)=\epsilon^{\rm loc}_i(\xi;x)$. Moreover, $B_m$ are Fujita approximations, i.e. $\lim_{m\to\infty}\vol(B_m)=\vol(\xi)$.
As in the proof of \cite[Theorem 6.22]{FL17}, we get
\[\rho_m^*\langle(\xi-\frac 1mG)^{n-1}\rangle\leq (B_m^{n-1})\leq \rho_m^*\langle\xi^{n-1}\rangle\ ,\]
where the inequalities signify that the differences are pseudoeffective curve classes on $X_m$. 
After pushforward, this implies $\langle(\xi-\frac 1mG)^{n-1}\rangle\leq\rho_{m*}(B_m^{n-1})\leq\langle\xi^{n-1}\rangle$.

Let $A$ be an ample divisor on $X$. There exists a norm $||\cdot||$ on $N_1(X)$ such that its restriction to the Mori cone $\Eff_1(X)$ is multiplication by $A$. Then $||\rho_{m*}(B_m^{n-1})-\langle\xi^{n-1}\rangle||\leq ||\langle\xi^{n-1}\rangle-\langle(\xi-\frac 1mG)^{n-1}\rangle||$. 
Since by \cite[Proposition 2.9]{BFJ09} the positive intersection product is continuous on the big cone, we deduce that
$\lim_{m\to\infty}\rho_{m*}(B_m^{n-1})=\langle\xi^{n-1}\rangle$.
Via the blow-up definition of Seshadri constants, we easily see that $\epsilon(\cdot;x)$ is upper semicontinuous on $\overline{\Mov}_1(X)$. 
Then
\begin{align*}
	\epsilon(C;x)&\geq\limsup_{m\to\infty}\epsilon(\rho_{m*}(B_m^{n-1});x)&\text{by upper semicontinuity}\\
	&\geq\limsup_{m\to\infty}\epsilon((B_m^{n-1});x_m) &\text{by \cite[Lemma 4.3]{Ful21}}\\
	&\geq\limsup_{m\to\infty}\epsilon^{\rm loc}_1(B_m;x_m)\cdot\ldots\cdot\epsilon^{\rm loc}_{n-1}(B_m;x_m)&\text{by the big and semiample case}\\
	&=\epsilon_1^{\rm loc}(\xi;x)\cdot\ldots\cdot\epsilon_{n-1}^{\rm loc}(\xi;x)&\text{ proved above.} 
\end{align*} 

$(3)$ Similar to $(2)$ we start by assuming first that $\xi$ is represented by a big and semiample Cartier divisor $H$. Fix $D$ an effective divisor on $X$ with multiplicity $M\geq 1$ at $x$. It is sufficient to prove that 
\[
M\cdot {\rm Vol}_{\bb R^{n-1}}\sinob{x}{H}_{\alpha_n=0} \ \leq \ {(H^{n-1}\cdot D)} .
\]
Let $x_1,\ldots,x_n$ be local coordinates around $x$ such that $D$ admits a local equation with nonzero coefficient for $x_n^M$. Since locally it is the smallest degree lexicographic monomial, then $\nu_{\rm deglex}(D)=(0,\ldots,0,M)$. Note, that any general set of coordinates will do.	
	For $m\geq 1$, denote 
	\[\Gamma_{m}\deq\nu_{\rm deglex}(H^0(X,mH)\setminus\{{\bf 0}\})\subset\bb N^n\ .\]
	Observe that $\nu_{\rm deglex}$ sends $H^0(X,mH-D)\subset H^0(X,mH)$ inside $\Gamma_{m,\alpha_n\geq M}$.
	For every $\alpha\in\Gamma_{m,\alpha_n<M}$ choose one $s_{\alpha}\in H^0(X,mH)$ with $\nu_{\rm deglex}(s_{\alpha})=\alpha$. Since their valuations are distinct, the $s_{\alpha}$ are linearly independent. Let $V_m$ be their span in $H^0(X,mH)$. Then $\nu_{\rm deglex}(V_m\setminus\{{\bf 0}\})=\Gamma_{m,\alpha_n<M}$ and $V_m\cap H^0(X,mH-D)={\bf 0}$. In particular, the natural composition $V_m\to H^0(X,mH)\to H^0(D,mH|_D)$ is injective, thus 
	\[\#\Gamma_{m,\alpha_n<M}=\dim V_m\leq h^0(D,mH|_D)\approx \frac{m^{n-1}}{(n-1)!}\cdot (H^{n-1}\cdot D)\ .\]
	To conclude, it is sufficient to prove that $\frac{\#\Gamma_{m,\alpha_n<M}}{m^{m-1}}$ approximates $M\cdot {\rm vol}_{\bb R^{n-1}}(\sinob{x}{H}_{\alpha_n=0})$.

	By the proof of Theorem \ref{thm:upperbound}, the $\Gamma_m$ are Borel-fixed discrete sets in $\bb N^n$. So the slice $\Gamma_{m,\alpha_n=0}$ identifies with the projection ${\rm pr}(\Gamma_m)$ onto the first $n-1$ coordinates and remains Borel-fixed. Consider the semigroup $\Gamma_{\alpha_n=0}=\bigcup_{m\geq 0}\Gamma_{m,\alpha_n=0}\times\{m\}$ and $\Delta_0\subset\bb R^{n-1}_+$ its associated convex body, the intersection of the closure of the convex span of $\Gamma_{\alpha_n=0}$ in $\bb R^{n}_+$ with $\bb R^{n-1}\times\{1\}$. From the projection and compactness perspective of $\sinob{x}{H}$, we have $\Delta_0={\rm pr}(\sinob{x}{H})=\sinob{x}{H}_{\alpha_n=0}$, with the last identification being due to the Borel-fixed property for $\sinob{x}{H}$ in Theorem \ref{thm:upperbound}.

	As $\Gamma_{\alpha_n=0}$ is a projection of $\Gamma=\bigcup_{m\geq 0}\Gamma_m\times\{m\}$,  it verifies the conditions \cite[(2.3)--(2.5)]{LM09}, by \cite[Lemma 2.2]{LM09}. Consequently, we have asymptotically that  $\#\Gamma_{m,\alpha_n=0}\approx{m^{n-1}}{\rm vol}_{\bb R^{n-1}}\sinob{x}{H}_{\alpha_n=0}$, by \cite[Proposition 2.1]{LM09}. Since $\Gamma_m$ is Borel-fixed, it implies the inequalities
	\[
	 M\cdot\#\Gamma_{m,\alpha_{n-1}\geq M,\alpha_n=0} \ \leq \ \#\Gamma_{m,\alpha_n<M} \ \leq \ M\cdot\#\Gamma_{m,\alpha_n=0} \ ,
	\]
	where the first is obtained by going in the $-{\bf e}_{n-1}+{\bf e}_n$ direction and the second in the $-{\bf e}_n$ direction.	The conclusion follows as $\#\Gamma_{m,\alpha_{n-1}<M,\alpha_n=0}\leq M\cdot\#\Gamma_{m,\alpha_{n-1}=\alpha_n=0}\in O(m^{n-2})$.

		For the general case consider $\xi\in N^1(X)_{\bb R}$ be an arbitrary big class. 		Let $x\not\in{\bf B}_+(\xi)$ and construct $G$, $\rho_m$, $x_m$ and $B_m$ as in $(2)$. Note that once they have all been constructed, then we can replace $x$ by a very general point so that the $x_m$ are also very general, and the inequalities $\leq_{x_m}$ stay the same. Then

\[
\epsilon(C;x) \ \geq \ \limsup_{m\to\infty}\epsilon((B_m^{n-1});x_m) \ \geq \ (n-1)!\cdot \limsup_{m\to\infty}{\rm vol}_{\bb R^{n-1}}\sinob{x_m}{B_m}_{\alpha_n=0} \ ,
\]
where the first inequality is obtained as in $(2)$, and the second by the big and semiample case.
	
It remains to show that the latter sequence converges to ${\rm vol}_{\bb R^{n-1}}\sinob{x}{\xi}_{\alpha_n=0}$. For this observe that the $B_m$ are Fujita approximations of $\xi$ from below. Then for volume reasons $\sinob{x}{\xi}$ is the closure of $\bigcup_{m\gg 1}\sinob{x_m}{B_m}$. The volume of the  slice $\alpha_n=0$ is in general only expected to be upper-semicontinuous, however in this case the $\sinob{x_m}{B_m}$ are also Borel-fixed, and so the volume of the slice identifies with the volume of the projection of the body on the $\alpha_n=0$ hyperplane. This is a continuous operation. 
\end{proof}

\end{document}